\documentclass[12pt]{amsart}
\makeatletter
\usepackage[nosort]{cite}
\usepackage{amsmath,amsthm,amsfonts,amssymb,amscd,mathrsfs,slashed,graphicx,braket,mathtools}
\usepackage{tikz-cd}
\usepackage[a4paper,top=3cm,bottom=3cm,left=3cm,right=3cm]{geometry}
\usepackage[colorlinks, linkcolor=blue, anchorcolor=blue, citecolor=blue]{hyperref}

\newlength{\xtrawidth}
\newlength{\xtraheight}
\numberwithin{equation}{section}
\numberwithin{table}{section}
\numberwithin{figure}{section}
\theoremstyle{definition}
\newtheorem{definition}{Definition}[section]
\theoremstyle{theorem}
\newtheorem{prop}[definition]{Proposition}
\theoremstyle{theorem}
\newtheorem{theorem}[definition]{Theorem}
\theoremstyle{remark}
\newtheorem{remark}[definition]{Remark}
\theoremstyle{definition}
\newtheorem{exa}[definition]{Example}
\theoremstyle{theorem}
\newtheorem{lem}[definition]{Lemma}
\theoremstyle{theorem}

\theoremstyle{theorem}
\newtheorem{cor}[definition]{Corollary}

\newcommand{\be}{\begin{equation}}
\newcommand{\ee}{\end{equation}}
\newcommand{\ax}{\alpha}
\newcommand{\bx}{\beta}

\newcommand{\ox}{\omega}

\newcommand{\Ox}{\Omega}
\newcommand{\Dx}{\Delta}

\title[]{Stability in Categories and Normal Projective Varieties over Perfect Fields}
\author[]{Hung-Yu~Yeh}
\address{Department of Mathematics, National Central University, No.300, Jhongda Rd., Jhongli City, Taoyuan County 32001, TAIWAN}
\thanks{}
\email{hyyeh@math.ncu.edu.tw}
\date{}

\begin{document}
\begin{abstract}{We present a notion of $\Delta$-stability and stability filtration in arbitrary categories which is equivalent to the existence of Harder-Narasimhan (HN) sequences on objects. Indeed it is equivalent to the existence of a zero morphism, a partial order on objects, and a collection of some universal sequences. In additive categories embedded in an ambient triangulated category, we could obtain a numerical polynomial or central charge of objects by calculating the Euler characteristic of slope sequences and objects, inducing a partial order and HN sequences. In the case of normal projective surfaces over an arbitrary perfect field $k$ we show the existence of $\Delta$-stabilities of degree one on the relevant bounded derived categories which is equivalent to the existence of Bridgeland's stabilities on normal surfaces. This result also leads to new effective restriction theorem of slope semistable sheave on normal projective varieties over perfect fields. Our approach also gives alternative proofs of Hodge Index Theorem and Bogomolov Inequality.}
\end{abstract}
\maketitle
\tableofcontents
\pagenumbering{arabic}

\section{Introduction}

Stability in algebraic geometry, first introduced by Mumford in the 1960's, is used as a tool to construct moduli space of sheaves on projective varieties and generalized by Takemoto, Gieseker, Simpson and Maruyama. Later this notion was extended to objects in arbitrary abelian category by A.~Rudakov~\cite{MR1480783}. On the other hand, in the particle physics the concept of the stability is associated to the formation or decay process, that is, particles combine or split to form other particles. More recently, motivated by the conjecture of homological mirror symmetry and Douglas' $\Pi$-stability on the category of B-branes (see~\cite{MR2567952}) Bridgeland introduces a precise definition of stability conditions on triangulated categories which depends on the existence of Harder-Narasimhan (HN) filtration and central charges on the relevant K group of associated triangulated categories~\cite{MR2373143}, and gives explicit constructions on smooth elliptic curves and K3 surfaces over $\mathbb C$~\cite{MR2376815}. 

Mumford-Takemoto, or $\mu$-semistable coherent sheaves on projective schemes over a field~$k$ have been studied for a long time. One of the main problems we have to meet is the boundedness of the family of semistable sheaves with a given Hilbert polynomial. In particulr, we would like the restriction of a semistable sheaves to general hypersurfaces to be semistable again. The theorem of Mehta and Ramanathan~\cite{MR649194} shows that the restriction of a $\mu$-(semi)stable sheaf on smooth projective variety over an algebraically closed field $k$ of arbitrary characteristic to a general hypersurface of sufficiently large degree is also $\mu$-semistable. On a normal projective variety over $k$ of characteristic zero the restriction theorem of Flenner~\cite{MR780080} gives effective bounds on the degree of hypersurfaces for $\mu$-semistable sheaves. Later Langer~\cite{MR2051393} proves effective restriction theorems for $\mu$-(semi)stable sheaves on smooth projective varieties in characteristic~$p$. In characteristic zero Bogomolov's effective restriction theorem~\cite{MR522939} is the strongest result on smooth projective surfaces. However, for normal projective varieties over algebraically closed fields of arbitrary characteristic, it is still an open and important problem whether there exists any (effective) restriction theorem for $\mu$-semistable sheaves. A detailed overview of moduli spaces of sheaves is given by Huybrechts and Lehn~\cite{MR2665168}.

In another direction, a crucial and open problem is the existence of Bridgeland stability conditions on $\mathrm{D^b}(X)$ for a general smooth scheme $X$ over $\mathbb C$. Even for general smooth projective threefolds it is still open because of the absence of Bogomolov-Gieseker type inequality for the third Chern character of stable objects~\cite{MR3121850}. Indeed, the existence of Bridgeland stability on smooth projectives surfaces is one of the important consequences of Bogomolov Inequality and Hodge Index Theorem~\cite{MR2998828}. For a good introduction to Birdgeland stability conditions, we refer to the notes by Macr{\`\i} and Schmidt~\cite{MR3729077}. 

From the point of view of the derived category of coherent sheaves on Noetherian schemes, it is unnatural to only consider stability conditions on smooth schemes over $\mathbb C$. However, when $X$ is a singular scheme over $k$, Chern character might only be defined on the full triangulated subcategory $\mathrm{Perf}(X) \subset \mathrm{D^b}(X)$, in which each object is locally isomorphic to a bounded complex of locally free sheaves of finite rank. But based on ref.~\cite{MR3935042} there are $K$-theoretic obstructions to the existence of bounded $t$-structure on $\mathrm{Perf}(X)$, which is one of the necessary conditions of Bridgeland stability conditions, and $\mathrm{D^b}(X)$ appears more natural in the theory of stability conditions. Without Chern character and boundedness property for coherent sheaves on singular schemes, the existence of Bridgeland stability conditions on $\mathrm{D^b}(X)$ is an open problem even for general normal projective surfaces over $k$. In contrast, Gieseker stability and other important geometric filtration which are not covered by Bridgeland stability behave better on projective schemes over an arbitrary field $k$. This suggests that we should need some suitable definition of such \textit{binding process} in the category of interest. 

One goal of this paper is to study the structure of an arbitrary category $\mathcal A$ that arises from a subcategory of the functor category from a partial order to the category $\mathcal A$. This leads to a notion of stability filtration in categories which is equivalent to the existence of $\Dx$-stability data, or HN sequences of objects in $\mathcal A$. 
\begin{theorem}
If the category $\mathcal A$ has a $\Delta$-stability data, then there exists a subcategory $\mathcal A^s$ with the same objects of $\mathcal A$ and same $\Delta$-stability data such that $\mathcal A^s$ has a stability filtration. Conversely, if $\mathcal A$ has a stability filtration, then $\mathcal A$ has a $\Delta$-stability data.
\end{theorem}
Indeed it is equivalent to existences of a zero morphism, a partial order on objects and $\Dx$-class, a collection of some universal sequences. In order to determine a $\Dx$-stability data on $\mathcal A$, an important question is to which extent the objects of the semistable subcategory induced by the $\Dx$-stability structure, a partial order, approximate non-semistable objects in $\mathcal A$. In a certain situation, each object in $\mathcal A$ can have a unique HN sequence.
\begin{theorem}
Given an exact $\Delta$-stability structure on a category $\mathcal A$ with $\mathrm{Hom}(A, B) = 0_c$ for $\Delta$-semistable objects $A \succ B$. If $\mathcal A$ is weakly $\Delta$-Artinian and weakly $\Delta$-Noetherian, then for every object $E \in \mathcal A$, there exists a unique Harder-Narasimhan sequence
$$
E_0 \xrightarrow{i_0} E_1 \xrightarrow{i_1} \cdots \xrightarrow{i_{n-1}} E_{n} \xrightarrow{i_{n}} E
$$ 
such that for $j > 0$, $E_{j-1} \xrightarrow{i_{j-1}} E_j \rightarrow F_{j-1} \in \Delta$, and $F_j$ are $\Delta$-semistable objects of $\mathcal A$ with 
$$
F_0 \succ F_1 \succ \cdots \succ F_n.
$$
\end{theorem}
If the category $\mathcal A$ with the zero morphism has finite products which are canonically isomorphic to finite coproduct, we can impose a commutative monoidal structure on $\mathrm{Hom}(\mathcal A)$ inducing additive structure on $\mathcal A$~\cite{MR1712872}. Thus in additive categories it allows us to reliably deduce from some invariants on Hom spaces the resulting $\Dx$-structure. To some extent we can think of these invariants on Hom spaces of the category $\mathcal A$ as the analog of Euler characteristic of a slope sequence (a collection of some objects) and general objects. The chain of supported weakly ample sequence on $\mathcal A$ gives a rather detailed description on how the invariants are constructed from the coefficients of numerical polynomials of Euler characteristics. Specially, we can form a new slope polynomial of lower degree induced by a slope sequence constructed from a tilted ample chain with the boundedness condition, and thus it gives us a new $\Dx$-stability structure, or exhaustive positive system.

In projective schemes over a field $k$, the natural slope sequence is given by an ample invertible sheaf and the slope polynomial is the well-known Hilbert polynomial. For the category of projective shemes of dimension $n$ over $k$, the projective space $\mathbf P^n$ is the weakly terminal object by the assumption. Through Grothendieck-Verdier duality and the vanishing property of semistable sheaves, we prove the boundedness of Euler characteristics for $\hat{\mu}(\mu)$-semistable torsion free sheaves on a normal integral projective scheme $X$ over a perfect field. 
\begin{theorem}
Let $E$ be a torsion free coherent sheaf on a normal integral projective scheme $X$ over a perfect field $k$ of $\dim X = n$ with a fixed very ample invertible sheaf $\mathcal O_X(H)$. Then there exists a polynomial $\hat P(\hat{\mu}_{\max}(E), \hat{\mu}_{\min}(E), \ox_X, d = \deg_H X)$ depending on the HN filtration of $E$, the degree of the canonical sheaf $\ox_X$ of $X$ and the degree of $X$ with respect to $\mathcal O_X(H)$ such that
\begin{align*}
&(-1)^{n-2}\chi(\mathcal O_{H^{n-2}}, E) \leq (-1)^n\chi(\mathcal O_{H^n}, E)\hat P(\hat{\mu}_{\max}(E), \hat{\mu}_{\min}(E), \ox_X, d)\\
&=(-1)^n\chi(\mathcal O_{H^n}, E)\left(\hat P(\hat{\mu}(E), \ox_X, d) + \frac{1}{2}\left(\hat{\mu}_{\max}(E) - \hat{\mu}(E)\right)(\hat{\mu}(E) - \hat{\mu}_{\min}(E))\right)
\end{align*}
with $(-1)^n\chi(\mathcal O_{H^n}, E) = d\cdot \mathrm{rk}(E)$ and $\hat{\mu}(E) = -\frac{\chi(\mathcal O_{H^{n-1}}, E)}{\chi(\mathcal O_{H^n}, E)}$. Here, 
$$
\hat P(\hat{\mu}(E), \ox_X, d) = \binom{\hat{\mu}(E)}{2} + \frac 1 2 (n - \hat{\mu}(\mathcal O_X))(1 + \hat{\mu}(\ox_X)).
$$
\end{theorem}
Moreover, this boundedness theorem implies Hodge Index Theorem and Bogomolov Inequality for locally free coherent sheaves on $X$ by standard arguments in the intersection theory. Our approach also gives easy proofs of the results on smooth projective schemes over $k$. For a crude bound of Euler characteristics of torsion free sheaves on general integral projective schemes, see Remark~\ref{rem0}.

For another crucial application of the boundedness theorem we prove the following effective restriction theorem for $\hat{\mu}$-semistable sheaves.
\begin{theorem}
Let $E$ be a $\hat\mu$-(semi)stable torsion free sheaf of $\mathrm{rk}(E) > 1$ on a normal integral projective scheme $X$ of $\dim X = n$ over a perfect field $k$ with a very ample invertible sheaf $\mathcal O_X(H)$. Assume $D \in |\mathcal O_X(lH)|$ is a normal divisor such that $E|_D$ is torsion free. If
$$
l > 2(1 - \mathrm{rk}(E))\left((-1)^{n-2}\chi(\mathcal O_{H^{n-2}}, E) - d\cdot\mathrm{rk}(E)\hat P(\hat{\mu}(E), \ox_X, d)\right) + \frac{1}{d\cdot\mathrm{rk}(E)(\mathrm{rk}(E) - 1)},
$$
then $E|_D$ is $\hat\mu$-(semi)stable. 
\end{theorem}
Together with the bound of Euler characteristics of torsion free sheaves on integral projective schemes, the similar argument implies the effective restriction theorem for $\hat{\mu}$-semistable sheaves (see Remark~\ref{rem1}).

Finally, combined with this bound of Euler characteristics on normal integral projective schemes and the technique of tilted heart and slope sequences in triangulated categories, we give confirmation for the existence of Bridgeland stability conditions on $\mathrm{D^b}(X)$ of a normal projective surface $X$, as suggested in~\cite{MR3935042}. Indeed, this is a special case of the existence of our $\Dx$-stability of degree $= n - 1$ for $\dim X = n = 2$.
\begin{theorem}
Suppose $X$ is a normal projective surface over a perfect field $k$ with a vary ample invertible sheaf $\mathcal O_X(H)$. Given a rational number $q = \frac{m_1}{m_2} \in \mathbb Q$, there exists the tilted heart $\mathcal A_q$ of a bounded t-structure on the bounded derived category of coherent sheaves $\mathrm{D^b}(X)$ with a one parameter family of slope sequences $\left\{L_{-t}^s(m_0) : m_0 > m_{\min}\right\}$ and $L_{-t}^s(m_0)[1] \in \mathrm{Ext}^2\left(\mathcal O_{H^2}^{\oplus m_2\binom t 2 + m_1\cdot t + m_0} , \mathcal O_X(-tH)^{\oplus m_2}\right)$ such that the associated slope polynomials $P^s_{t, m_0}(-) = \chi\left(L^s_{-t}(m_0), -\right)$ induce $\Dx$-stabilities of degree $1$ on $X$, namely Bridgeland stability conditions.
\end{theorem}

The paper is organized as follows:
In Section~\ref{sec:stab} we present the notions of $\Dx$-stability data and stability filtration in arbitrary categories. We focus on the class of some universal sequences ($\Dx$ class), on the partial orders $\Phi$ of objects, and on the weakly initial and terminal subcategory $\mathcal P_{\Phi}$. With suitable universal property (HN sequences), this $\Dx$-stability data in the category $\mathcal A$ is equivalent to the existence of a subcategory $\mathcal A^s$ with a zero morphism, consist of subcaterories (semistable category) of the functor category $\mathcal A^{\Phi^{\mathrm{op}}}$ fulling weakly universal conditions in some sense (stability filtration).
Section~\ref{nts} discusses the existence of HN sequence of objects in $\mathcal A$. Firstly, we introduce the exact $\Dx$-class with a compatible $\Dx$-stability structure, and thus weakly $\Dx$-Artinian and $\Dx$-Noetherian of $\mathcal A$ would imply the existence of unique HN sequences, that is, $\Dx$-stability on $\mathcal A$. Secondly we summarize the glued $\Dx$-stability data constructed from a suitable set of $\Dx$-stability datum on subcategories. Due to the importance to our analysis, a particular emphasis is put on the additive categories. 
In Section~\ref{sec:was} we describe numerical polynomials induced by the chain of supported weakly ample sequence in a $k$-linear additive category embedded in a triangulated category. We start with a description of the weakly ample sequence as deduced from the ample sequence in an abelian category. We analyze the properties of the chain of supported weakly ample sequences attributed to Grothendieck-Verdier like duality between ambient triangulated categories.
In Section~\ref{sec:nsp} we apply the method of tilted ample chain satisfying boundedness conditions to explicitly construct new exhaustive positive systems from the slope polynomial induced by new slope sequences. We argue that in nonsingular projective curves this construction directly relate to the category of representation of Kronecker quiver $P_2$.  
In Section~\ref{sec:nv} we study the boundedness property of torsion free coherent sheaves and $\Dx$-stability of degree one on a normal projective surface $X$ over a perfect field. For $\hat{\mu}$-semistable torsion free sheaves on $X$ we establish an inequality between the Euler characteristic of the $\hat{\mu}$-semistable sheaf $E$ and the associated polynomial of rank of $E$, degree of $X$, $\hat{\mu}(E), \hat{\mu}(\mathcal O_X)$ and $\hat{\mu}(\ox_X)$. We then illustrate the different aspects of the proposed correspondence with $\Dx$-stability of degree one (Bridgeland stability conditions), Hodge Index Theorem and Bogomolov Inequality on $X$.
Effective restriction theorem and our generalized results in higher dimension are presented in Section~\ref{higher}. We give technical computations on the effective restriction theorem of $\hat{\mu}$-semistable torsion free sheaves on surfaces. For the general results for higher dimensional schemes which can be deduced from the same arguments in surfaces, we explain our general theorems of boundedness property, Bogomolov Inequality, $\Dx$-stability of degree $= \dim X - 1$, and effective restriction theorem on any normal integral projective scheme $X$ over a perfect field.

\section{$\Delta$-stability and stability filtration in categories}  \label{sec:stab}
In this section we present the notions of $\Dx$-stability data and stability filtration in arbitrary categories. We then discuss the explicit construction of correspondence between two notions in categories. This construction offers a systematic way to fulfill the conditions of $\Dx$-stability data in a category that are obtained from the functor subcategory (semistable category) associated to a given partial order. 

\subsection{$\Delta$-stability in categories} 

We first introduce the formal definition of \textit{$\Delta$-stability data} on an arbitrary category $\mathcal A$.
\begin{definition}\label{sd}
Suppose $\mathcal A$ is an arbitrary category and a class $\Delta$ is a category with objects of sequences of the form $F' \xrightarrow{i} F \xrightarrow{p} F''$ and morphisms of sequences, $\Phi$ is a partial order, and $\mathcal P_{\Phi}=\{ \mathcal P_{\phi}\}_{\phi \in \Phi} \subset \mathcal A$ is a weakly initial and terminal subcategory such that for each object $E \notin \mathcal P_{\Phi}$ there exists arrows $E \rightarrow F_{m(E)}, F_{M(E)} \rightarrow E$ with $F_{m(E)}, F_{M(E)} \in \mathcal P_{\Phi}$ and $M(E) > m(E) \in \Phi$. \textit{Pre-$\Delta$-stability data} on $\mathcal A$ is a triple $\left(\Delta, \Phi, \mathcal P_{\Phi}\right)$ satisfying the following condition: each object $E \in \mathcal A$ has a Harder-Narasimhan (HN) sequence of the form
$$
E_0 \longrightarrow E_1 \longrightarrow \cdots \longrightarrow E_{n-1} \longrightarrow E_n = E
$$ 
such that $E_{j-1} \xrightarrow{i_{j-1}} E_j \xrightarrow{p_{j-1}} F_j \in \Delta$ with $F_{M(E)} =E_0 \in \mathcal P_{\phi_0}, \dots, F_{m(E)} = F_n \in \mathcal P_{\phi_n}$, and $\phi_i < \phi_j$ for all $i>j$. Moreover, the triple $\left(\Delta, \Phi, \{\mathcal P_{\phi}\}\right)$ is called a \textit{$\Delta$-stability data} of $\mathcal A$ if 
\begin{enumerate}
\item $h'\circ h$ is not an isomorphism for any sequence $F_{\phi} \xrightarrow{h} F''_{\phi'} \xrightarrow{h'} F_{\phi}$ with $\phi \neq \phi'$;
\item $p_{j-1}\circ i_{j-1} = h\circ p_{j-2}$ for some $h: F_{j-1} \rightarrow F_j$ for all $j>0$;
\item Each sequence $E_{j-1} \xrightarrow{i_{j-1}} E_j \xrightarrow{p_{j-1}} F_j$ is unique for all $j$ in $\Delta$, and given a commutative diagram of sequences of $\Delta$,
\begin{center}
\begin{tikzcd}[row sep=large, column sep=large]
E_{j-1} \arrow[r, "i_{j-1}"] \arrow[d, dashed, "g^0"]&E_j \arrow[d, "f"] \arrow[r, "p_{j-1}"] &F_j \arrow[d, dashed, "g^1"] \\
E'_{j'-1} \arrow[r, "i'_{j'-1}"]&E'_{j'} \arrow[r, "p'_{j'-1}"] &F'_{j'} 
\end{tikzcd}
\end{center}
and $f$ is an isomorphism, then $g^0, g^1$ are isomorphisms. 
\end{enumerate}
\end{definition}

Indeed, $\Delta$-stability data would be closely related to the algebraic structure of non-associative operations in categories. Let us first consider the following basic examples one could keep in mind.
\begin{exa}
Let the set of objects in the category $\mathcal A$ be the natural numbers $(\mathbb N, -)$, and if $n_1, n_2, n_3:=n_2-n_1 \in \mathbb N$ then we define the sequence $n_1 \rightarrow n_2 \rightarrow n_3 \in \Delta$. It is clear that subtraction is neither an operation in $\mathbb N$, but $\mathbb Z$, nor associative. Let $\mathcal P_{\Phi} = \{1\}$ and $\Phi:= "="$, then for each number $n \neq 0 \in \mathbb N$, there is a finite sequence $1 \rightarrow \cdots \rightarrow n-1 \rightarrow  n$ such that $n-(k+1) \rightarrow n-k \rightarrow 1 \in \Delta$ for $0\leq k \leq n-2$, and we call such sequence a Jordan-H\"older (JH) sequence of $n$. Moreover,  lengths of JH sequence of objects in $\mathcal A$ induces a new partial order $<$ different from the original one $"="$. In more general case, we could not expect the uniqueness of JH sequences of objects, but the length of objects is independent to the choice of different JH sequences,  for instance, JH filtration of each semistable sheaf in the category of coherent sheaves.  Therefore in this article we only consider HN sequences of objects in categories.  
\end{exa}

\begin{exa}\label{division}
Let the set of objects in the category $\mathcal A$ be positive integers $\mathbb N^*$, and if $n_1, n_2, n_3:=n_2/n_1 \in \mathbb N^*$ then we define the sequence $n_1 \rightarrow n_2 \rightarrow n_3 \in \Delta$. Obviously division is not an operation in $\mathbb N^*$ but $\mathbb Q^+$. Let $\mathcal P_{\Phi} = \{p^n\}$ for all prime numbers $p$ and $n \in \mathbb N$, and $p_1^{n_1} > p_2^{n_2}$ if $p_1 - p_2  > 0$. Each number $m \neq 1 \in \mathbb N^*$ has the HN sequence
\begin{center}
\begin{tikzcd}[row sep=small, column sep=tiny]
p_1^{n_1} \arrow{rr}& &p_1^{n_1}p_2^{n_2} \arrow{rr} \arrow{dl}& &p_1^{n_1}p_2^{n_2}p_3^{n_3} \arrow{r} \arrow{dl} &\cdots \arrow{r} &  p_1^{n_1}\cdots p_{k-1}^{n_{k-1}}\arrow{rr}\arrow{dl} && m = p_1^{n_1}p_2^{n_2}\cdots p_k^{n_k} \arrow{dl} \\
&p_2^{n_2} &&p_3^{n_3} &&p_{k-1}^{n_{k-1}}&& p_k^{n_k} 
\end{tikzcd}
\end{center}
with $p_1 > p_2 > \cdots >p_k$, and each factor $p_i^{n_i}$ has the JH sequence induced by $p_i$. Indeed, the triple $(\Delta, ">", \{p^n\})$ is a $\Delta$-stability data if we impose the commutative condition (ii) in Def.~\ref{sd}: $p_{j-1}\circ i_{j-1} = h\circ p_{j-2}$ for all $h$ on $\mathcal A$.
\end{exa}

\begin{exa}\label{vs}
Suppose $\mathcal A$ is the abelian category of all finite dimensional vector spaces over a division ring $R$ spanned by $\{v_i\}$ for all $i \in \mathbb N$. Given a triple of vector spaces $(V_1, V_2, V_3)$, we define a sequence $V_1 \rightarrow V_2 \rightarrow V_3$ if $V_1 = Rv_{n_1}\oplus\cdots\oplus Rv_{n_{l-1}}, V_2 = Rv_{n_1}\oplus\cdots\oplus Rv_{n_{l-1}}\oplus Rv_{n_l}, V_3 = Rv_{n_l}$. Here each arrow in sequences is a regular linear transformation of vector spaces unique up to isomorphism in $\mathcal A$. Thus every nontrivial finite dimensional vector space $V$ has the HN sequence
\begin{center}
\begin{tikzcd}[row sep=small, column sep=tiny]
Rv_{n_1} \arrow{rr} & &Rv_{n_{1}}\oplus Rv_{n_2} \arrow{r} \arrow{dl} &\cdots \arrow{r} &  Rv_{n_1}\oplus\cdots\oplus Rv_{n_{l-1}}\arrow{rr}\arrow{dl} && V = Rv_{n_{1}}\oplus \cdots \oplus Rv_{n_l} \arrow{dl} \\
&Rv_{n_2} &&Rv_{n_{l-1}}&& Rv_{n_l} 
\end{tikzcd}
\end{center}
where $Rv_{n_i} > Rv_{n_j}$ if $i<j$, and each factor $Rv_{n_i}$ has the only  trivial JH sequence itself. If $\mathcal A^s$ is the subcategory of $\mathcal A$ such that $\mathrm{Hom}_R(v_i, v_j) = 0$ for $i>j$, then $(\Delta, ">", \{Rv_i\})$ forms a $\Delta$-stability data on $\mathcal A^s$.
\end{exa}

To the best of the author's knowledge, there are some more examples. In homotopy theory, each simple space $X$ of homotopy type of CW complex has a Postnikov system, infinite HN sequence with inverse arrows, such that each factor space is an Eilenberg-Mac~Lane space $K(\pi_n(X), n)$ (see~\cite{MR1702278}). In quantum mechanics, each quantum state is a superposition of stable states in the relevant Hilbert space determined by the quantum mechanical system. 

If the class $\Delta$ is a set of all short exact sequences of an abelian category $\mathcal A$, then the $\Delta$-stability data defined by the triple is coincided with stability data presented by  A.~L.~Gorodentsev and S.~A.~Kuleshov and A.~N.~Rudakov~\cite{MR2084563} and the finite sequence is the regular Harder-Narasimhan filtration. Thus we call the triple $\Delta$-stability data of $\mathcal A$. Indeed, if the class $\Delta$ consists of all short complexes lying on distinguished triangles in a triangulated category $\mathcal T$, then a $\Delta$-stability data of $\mathcal T$ gives a t-stability in the sense of~\cite{MR2084563}.

If the stability data is induced by exceptional sequences in the associated derived category $\mathrm{D^b}(\mathcal A)$ of an abelian category $\mathcal A$~\cite{MR992977}, then we say that it is an  \textit{exceptional} stability data of $\mathcal A$ and induces a \textit{exceptional} t-stability of $\mathrm{D^b}(\mathcal A)$. For example, on $\mathbf P^1$ over a field $k$ there is a uniquely determined decreasing sequence of integers $a_1 \geq \cdots \geq a_r$ such that $E \simeq \mathcal O(a_1) \oplus \cdots \oplus \mathcal O(a_r)$ for each vector bundle $E$ of rank $r$ on $\mathbf P^1$ (see~\cite{MR2665168}).

Suppose a triangulated category $\mathcal T$ is \textit{strongly generated} by an extension-closed subcategory $\mathcal E \subset \mathcal T$, that is, for each object $E \in \mathcal T$ and a fixed integer $M_{\mathcal T}$ there exists a Harder-Narasimhan sequence such that each object $F_j$ belongs to $\langle\mathcal E\rangle$ with the length of HN sequence $\leq M_{\mathcal T}$. Then $\left(\Delta, \Phi:=\{0,1,\dots, M_{\mathcal T}\}, \{\mathcal P_{\phi} := \langle\mathcal E\rangle\}\right)$ is a pre-$\Delta$-stability data of $\mathcal T$. Indeed, for any smooth separated scheme $X$ the triangulated category of perfect complexes has a strong generator, thus a pre-$\Delta$-stability data (see A.~Bondal and M.~Van den Bergh~\cite{MR1996800}). Moreover, for some wrapped Fukaya category of Liouville manifolds one may find some (split) generators inducing a pre-$\Delta$-stability data (see M.~Abouzaid~\cite{MR2737980}).

\subsection{Stability filtration in categories}\label{stabilityf}

In general, different categories can have similar $\Delta$-stability data and HN sequences of objects even their sources are totally different, and thus share some categorical invariant or partial orders. So the natural question is that what sorts of conditions are sufficient or necessary for existence of $\Delta$-stability data in given categories. The first attempt to define a $\Delta$-stability data in a category $\mathcal A$ is to find a trivial morphism class and \textit{stability filtration} of $\mathcal A$ defined below.
\begin{definition}
Let $P$ be a partial order ``$\leq$'' and $\mathcal A$ is a category with a zero or trivial morphism class $0_c$ such that $f\circ 0_c = 0_c \circ f = 0_c$ if compositions are defined. Then the \textit{semistable} functor category $\mathcal{SS}_{\mathcal A^{P^{\mathrm{op}}}}$ is a subcategory of the functor category $\mathcal A^{P^{\mathrm{op}}}$ so that $\mathrm{Hom}_{\mathcal A}\left(F_i, F'_j\right) = 0_c$ if $j < i \in P$ for all functors $F, F' \in \mathcal A^{P^{\mathrm{op}}}$, and $\mathcal E^0$ denotes the smallest full subcategory containing $\mathcal{SS}_{\mathcal A^{P^{\mathrm{op}}}}(P)$ of $\mathcal A$. We defined the subcategory $\mathcal E^n$ of $\mathcal A$ by induction on $n$ such that all $E^{n} \in \mathcal E^n$ satisfy the following conditions: 
\begin{enumerate}
\item $\mathrm{Hom}_{\mathcal A}(E^{n}, \mathcal E^{0}) \neq 0_c$ or $\emptyset$, and $\mathrm{Hom}_{\mathcal A}(\mathcal E^{n-1},  E^{n}) \neq 0_c$ or $\emptyset$;
\item A nontrivial connected diagram in the comma category $(E^{n}\downarrow \mathcal E^0)$ has a weakly initial object $p: E^n \rightarrow E^0$ such that given another object $p': E^n \rightarrow E'^0$ with $p=h\circ p'$, the composite arrow $E^0 \xrightarrow{h'} E'^0 \xrightarrow{h} E^0$ is an identity. Here weakly initial means that for any nontrivial arrow $p': E^n \rightarrow E'^0$, it would fulfill $p' = h'\circ p$ for an arrow $h': E^0 \rightarrow E'^0$;  
\item A nontrivial connected diagram in the comma category $(\mathcal E^{n-1}\downarrow E^n)$ has a weakly terminal object $i: E^{n-1} \rightarrow E^n$ such that given another object $i': E'^{n-1} \rightarrow E^n$, the composite arrow $E^{n-1} \xrightarrow{h'} E'^{n-1} \xrightarrow{h} E^n$ is an identity for nontrivial morphisms $h, h'$;  
\item $m(E^n)  < m(E^{n-1})$, where $m(E^i) \in \mathcal E^0$ is the weakly initial object in $(E^i \downarrow \mathcal E^0)$.
\end{enumerate}
It turns out that $\bigcup_{m=0}^{\infty}\mathcal E^m \subset \mathcal A$ and if $E \in \mathcal E^n$ for all $E\in \mathcal A$ and $n \geq 0$, then $\mathcal E^0 \subset \mathcal E^0 \cup \mathcal E^1 \subset \cdots \subset \mathcal A$ is called a \textit{stability filtration} of $\mathcal A$.
\end{definition}

\begin{theorem}
If the category $\mathcal A$ has a $\Delta$-stability data, then there exists a subcategory $\mathcal A^s$ with the same objects of $\mathcal A$ and same $\Delta$-stability data such that $\mathcal A^s$ has a stability filtration. Conversely, if $\mathcal A$ has a stability filtration, then $\mathcal A$ has a $\Delta$-stability data.
\end{theorem}
\begin{proof}
Suppose there is a $\Delta$-stability data on $\mathcal A$, we first define the full subcategory $\mathcal E^0 := \mathcal P_{\Phi}$ with $\mathrm{Hom}_{\mathcal E^0}(F_{\phi}, F_{\phi'}) := 0_c$ if $\phi > \phi' \in P:= \Phi$, and $\mathcal E^n$ to be a subcategory of $\mathcal A$ with objects $\mathrm{ob}(\mathcal E^n):= \{ E\in \mathcal A \mid \mathrm{HN\ length}(E) = n\}$ and $\mathrm{Hom}_{\mathcal E^n}(E^n, E'^n) = \{ f\in \mathrm{Hom}_{\mathcal A}(E^n, E'^n) \mid (h_1, f, h_2): (E^{n-1} \rightarrow E^n \rightarrow E^0) \rightarrow  (E'^{n-1} \rightarrow E'^n \rightarrow E'^0) \in \mathrm{Hom}(\Delta), h_1 \in \mathrm{Hom}_{\mathcal E^{n-1}}(E^{n-1}, E'^{n-1}), h_2 \in \mathrm{Hom}_{\mathcal E^{0}}(E^{0}, E'^{0})\}$. For each object $E^n \in \mathcal E^n$, there exists a unique sequence $E^{n-1} \xrightarrow{i} E^n \xrightarrow{p} E^0 \in \Delta$ by assumption. Thus we can define morphisms in $\mathcal A^n := \bigcup_{m=0}^n\mathcal E^m$ as
\begin{enumerate}
\item $\mathrm{Hom}_{\mathcal A^n}(E^n, \mathcal E^0) := \mathrm{Hom}_{\mathcal A}(E^0, \mathcal E^0)\circ p$; 
\item $\mathrm{Hom}_{\mathcal A^n}(\mathcal E^{n-1}, E^n) := i \circ\mathrm{Hom}_{\mathcal A^{n-1}}(\mathcal E^{n-1}, E^{n-1})$;
\item $\mathrm{Hom}_{\mathcal A^n}(E^n, E'^n) := \mathrm{Hom}_{\mathcal E^n}(E^{n}, E'^n)$.
\end{enumerate}
One could see that the definition is well defined under composition of morphisms, and $\mathrm{Hom}_{\mathcal A^m}(E^n, \mathcal E^0) = \mathrm{Hom}_{\mathcal A^n}(E^n, \mathcal E^0)$ for all $m\geq n$ since $p\circ i = 0_c$. Also $\mathrm{Hom}_{\mathcal A^m}(E^n, E'^n) = \mathrm{Hom}_{\mathcal A^n}(E^n, E'^n)$, thus $\mathrm{Hom}_{\mathcal A^m}(\mathcal E^{n-1}, E^n)= \mathrm{Hom}_{\mathcal A^n}(\mathcal E^{n-1}, E^n)$. It turns out that $E^n \xrightarrow{p} E^0$ is a weakly initial object of the connected diagram containing $p\neq 0_c$ in $(E^n\downarrow \mathcal E^0)$ within $\mathcal A^s := \bigcup_{m=0}^{\infty}\mathcal E^m$. Indeed, since $p$ is unique up to isomorphisms in $\mathcal A$, $p = 0_c$ if and only if $p = 0_c\circ g$ for some $g\in \mathrm{Hom}_{\mathcal A^n}(E^n, F_{\phi})$ with $F_{\phi} > E^0$. But $g = h\circ p$ for $h \in \mathrm{Hom}_{\mathcal A}(E^0, F_{\phi})$, thus by the uniqueness $p = 0_c\circ h\circ p$ with an identity $0_c \circ h \in \mathrm{Hom}_{\mathcal A^s}(E^0, E^0)$. However, the assumption of $\Delta$-stability data implies $0_c\circ h$ is not an isomorphism as $\phi > m(E^n)$. Hence $p \neq 0_c$ in $\mathcal A^s$. 

Moreover, $p$ is a weak universal arrow in $\mathcal A^s$ by (i), and given any morphism $p': E^n \rightarrow E'^0$ with $p=h\circ p'$, one could form a commutative diagram 
\begin{center}
\begin{tikzcd}[row sep=large, column sep=large]
E^{n-1} \arrow[r, "i"] \arrow[d, "Id"]&E^n \arrow[d, "Id"] \arrow[r, "p"] &E^0 \arrow[d, "h\circ h'"] \\
E^{n-1} \arrow[r, "i"]&E^{n} \arrow[r, "p"] &E^0
\end{tikzcd}
\end{center}
and the left and middle vertical arrows are identities, thus $h\circ h'$ is an identity by the fact of $\Delta$-stability. Similarly, $E^{n-1} \xrightarrow{i} E^n$ becomes a weakly terminal object of the connected diagram containing $i\neq 0_c$ in $(\mathcal E^{n-1}\downarrow E^n)$ within $\mathcal A^s$. Hence $\{\mathcal E^0, \mathcal E^1, \dots\}$ is a stability filtration of $\mathcal A^s \subset \mathcal A$.

Conversely, assume that $\mathcal A$ has a stability filtration $\{\mathcal E^0, \mathcal E^1, \dots\}$ with a trivial morphism class $0_c$, and $\mathcal P_{\Phi}:= \mathcal E^0$ with $\Phi := P$. Thus $P_{\Phi}$ is weakly initial and terminal. For each objects $E^n \in \mathcal E^n$ there is an universal sequence $E^{n-1} \xrightarrow{i_{n-1}} E^n \xrightarrow{p_{n-1}} E^0$ with $p_{n-1}\circ i_{n-1} = 0_c$ and $m(E^{n-1}) > m(E^n)= E^0$. Let the $\Delta$-class being the collection of universal sequences of all $E^n \in \mathcal E^n$ for all $n$, then each object $E \in \mathcal A$ has the unique HN sequence up to isomorphism. Indeed, given a diagram of sequences of $\Delta$ with an isomorphism $f$,
\begin{center}
\begin{tikzcd}[row sep=large, column sep=large]
E^{n-1} \arrow[r] \arrow[d, dashed, "h'_1"]&E^n \arrow[d, "f"] \arrow{r} &E^0 \arrow[d, dashed, "h'_2"] \\
E'^{n-1} \arrow{r}&E'^{n} \arrow[r] &E'^0,
\end{tikzcd}
\end{center}
then there exists morphisms $h_1, h_2$ such that the both squares are commutative by the weak universality. Similarly we have a  morphism of sequences $(h_1, f^{-1}, h_2): (E'^{n-1}\rightarrow E'^n \rightarrow E'^0) \rightarrow (E^{n-1}\rightarrow E^n \rightarrow E^0)$, and $h_i\circ h'_i = \mathrm{Id}, h'_i \circ h_i = \mathrm{Id}$ by assumption. It turns out that $h_1$ and $h_2$ are isomorphisms. Therefore, $\left(\Delta, P, \mathcal E^0\right)$ is a $\Delta$-stability data of $\mathcal A$.
\end{proof}

\begin{remark}
Note that we may find some isomorphic objects $E^n \simeq E^m$ in $\mathcal A$ which are not isomorphic in $\mathcal A^s$ if there does not exist any commutative diagram comparable with this isomorphism. Indeed, $E^n$ is not isomorphic to $E^m$ for $n\neq m$ in $\mathcal A^s$.
\end{remark}

In other words, the existence of $\Delta$-stability data of a category $\mathcal A$ is equivalent to the existence of a trivial morphism $0_c$, a partial order on the objects and  a collection $\Delta$ containing some nontrivial universal sequences in $\mathcal A$.

\section{Exact $\Delta$-class and $\Delta$-exact functor}\label{nts}

\subsection{$\Delta$-stability structure} 

To define semistable subcategory it is expected to have such a partial order called \textit{$\Delta$-stability structure} on the objects of $\mathcal A$ if the following \textit{$\Delta$-seesaw property} holds. 
\begin{definition}
Assume each sequence in $\Delta$ is unique up to isomorphism, then a \textit{$\Delta$-stability structure} on an arbitrary category $\mathcal A$ is a partial order $\preceq$ on the objects of $\mathcal A$ such that for every sequence $A \rightarrow B \rightarrow C$ in $\Delta$, one of the three conditions is satisfied:
\begin{enumerate}
\item $A \prec B \Leftrightarrow A \prec C \Leftrightarrow B \prec C$;
\item $A \succ B \Leftrightarrow A \succ C \Leftrightarrow B \succ C$;
\item $A \asymp B \Leftrightarrow A \asymp C \Leftrightarrow B \asymp C$.
\end{enumerate}
This property is called the \textit{$\Delta$-seesaw property}. An object $A\in \mathcal A$ is called \textit{$\Delta$-semistable} if $B \preceq A$ for every sequence $B \rightarrow A \rightarrow C$ in $\Delta$. Note that $A$ is $\Delta$-semistable if no such sequence $B \rightarrow A \rightarrow C$ exists in $\Delta$.
\end{definition}

\begin{lem}\label{o}
Suppose $A$ is a category with a $\Delta$-stability structure with respect to the class $\Delta$ of sequences, such that each morphism $f: E_1 \rightarrow E_2$ of $\Delta$-semistable objects in $\mathcal A$ can either fit into one of two sequences of $\Delta$: $E_1 \xrightarrow{f} E_2 \rightarrow F$, $F \rightarrow E_1 \xrightarrow{f} E_2$, or $E_1 \xrightarrow{f_1} F \xrightarrow{f_2} E_2$ with $f = f_2 \circ f_1$. We define $0_c:= f$ if $E_1 \succ F$ or $F \succ E_2$, then $\mathrm{Hom}_{\mathcal A}(E_1, E_2) \equiv 0_c$ for any $\Delta$-semistable objects $E_1, E_2$ with $E_1 \succ E_2$.
\end{lem}
\begin{proof}
If $\mathrm{Hom}_{\mathcal A}\left(E_1, E_2\right) \neq 0_c$ for some $E_1 \succ E_2$ with $\Delta$-semistable objects $E_1$ and $E_2$, then there exists a non-trivial morphism $f \neq 0_c: E_1 \rightarrow E_2$ of $\mathrm{Hom}(E_1, E_2)$. By the assumption $f$ can not fit into neither $E_1 \xrightarrow{f} E_2 \rightarrow F$, nor $F \rightarrow E_1 \xrightarrow{f} E_2$, which implies $E_1 \preceq E_2$ since $E_1$ and $E_2$ are $\Delta$-semistable. Thus $f = f_2\circ f_1$ with $E_1 \preceq F$ and $F\preceq E_2$ but this contradicts to $E_1 \succ E_2$. Hence $\mathrm{Hom}_{\mathcal A}(E_1, E_2) \equiv 0_c$.
\end{proof}

\subsection{Exact $\Delta$-class} 

In general, given a $\Delta$-stability structure the existence of Harder-Narasimhan sequence of any object of $\mathcal A$ is still unclear. However, if the class $\Delta$ fulfills the \textit{exact condition} and the category $\mathcal A$ is weakly $\Delta$-Aritinian and weakly $\Delta$-Noetherian, then the existence of Harder-Narasimhan sequences is given by the theorem~\ref{tHN}, a generalization of the theorem in~\cite{MR1480783, MR2373143}.

\begin{definition}
The class $\Delta$ of sequences in a category $\mathcal A$ with a specific morphism class $S$ is called an \textit{exact} class of $\mathcal A$ if $i, p\neq S$ and are not isomorphisms for any $E' \xrightarrow{i} E \xrightarrow{p} F \in \Delta$, and the following conditions are fulfilled
\begin{enumerate}
\item For any two sequences $F_1 \rightarrow E_1 \xrightarrow{p_1} E_2, F_2 \rightarrow E_2 \xrightarrow{p_2} E_3 \in \Delta$, there exists such sequences $F_1' \rightarrow E_1 \xrightarrow{p_2 \circ p_1} E_3, F_1 \rightarrow F_1' \rightarrow F_2 \in \Delta$;
\item For any two sequences $E_1 \xrightarrow{i_1} E_2 \rightarrow F_1, E_2 \xrightarrow{i_2} E_3 \rightarrow F_2 \in \Delta$, there exists such sequences $E_1 \xrightarrow{i_2\circ i_1} E_3 \rightarrow F_1', F_1 \rightarrow F_1' \rightarrow F_2 \in \Delta$;
\item For any sequence $E_1 \xrightarrow{i} E_2 \xrightarrow{p} F_1 \in \Delta$ with morphisms $f: E_2 \xrightarrow{f} E'_2$ and $p': E'_2 \xrightarrow{p'} F'_1$, if $p'\circ f\circ i = S$ then there exists a morphism $g^1 \in \mathrm{Hom}(F_1, F'_1)$ such that the following diagram is commutative,
\begin{center}
\begin{tikzcd}[row sep=large, column sep=large]
E_1 \arrow[r, "i"] &E_2 \arrow[d, "f"] \arrow[r, "p"] &F_1 \arrow[d, dashed, "g^1"] \\
&E'_2 \arrow[r, "p'"] &F'_1.
\end{tikzcd}
\end{center}
Similarly, if $p\circ f'\circ i' = S$ for morphisms $f': E'_2 \xrightarrow{f'} E_2, i': E'_1 \xrightarrow{i'} E_1$, then there exists a morphism $g^0 \in \mathrm{Hom}(E'_1, E_1)$ so that $i\circ g^0 = f'\circ i'$. 
\item Given a commutative diagram of sequences in $\Delta$, 
\begin{center}
\begin{tikzcd}[row sep=large, column sep=large]
E_1 \arrow[r, "i"]\arrow[d, "g^0"] &E_2 \arrow[d, "f"] \arrow[r, "p"] &F_1 \arrow[d, "g^1"] \\
E'_1 \arrow[r, "i'"] &E'_2 \arrow[r, "p'"] &F'_1,
\end{tikzcd}
\end{center}
if two of morphisms $(g^0, f, g^1)$ are isomorphisms, then the other is also an isomorphism. Moreover, if $f$ is an identity, $p'=p$ and $i'=i$ then $g^0, g^1$ are identities.
\end{enumerate}
\end{definition}
 
Let us consider the following easy example first before further formal discussion. 
\begin{exa}\label{eN}
Suppose $\mathcal A$ is the category $\mathbb N^{*}$ with the $\Delta$ class induced by division as Example~\ref{division}. Then a trivial morphism $S$ could be defined as the composite $p\circ i$ for any sequence $n_1 \xrightarrow{i} n_2 \xrightarrow{p} n_2/n_1$ in $\Delta$, and any composite with a trivial morphism is itself a trivial morphism. One can easily check that condition (i) and (ii) of $\Delta$-exactness above are satisfied. We assume morphisms induced by condition (i) and (ii) are commutative, for instance, the squares below are all commutative
\begin{center}
\begin{tikzcd}[row sep=large, column sep=large]
n_1 \arrow[r, "i_1"] \arrow[d, "\mathrm{Id}"]&n_2 \arrow[d, "i_2"] \arrow[r, "p_1"] &n_2/n_1 \arrow[d, dashed, "\tilde{i}"] &&n_1/n_2 \arrow[r, "i_1"] \arrow[d, dashed, "\tilde{i}"] &n_1  \arrow[d, "\mathrm{Id}"]\arrow[r, "p_1"] &n_2 \arrow[d, "p_2"]\\
n_1 \arrow[r, "i_2\circ i_1"] \arrow[d, "i_1"]&n_3 \arrow[r, "p_3"]  \arrow[d, "\mathrm{Id}"] &n_3/n_1 \arrow[d, dashed, "\tilde{p}"]&& n_1/n_3 \arrow[r, "i_3"] \arrow[d, dashed, "\tilde{p}"]&n_1 \arrow[r, "p_2\circ p_1"]  \arrow[d, "p_1"] &n_3 \arrow[d, "\mathrm{Id}"] \\
n_2\arrow[r, "i_2"] &n_3 \arrow[r, "p_2"] &n_3/n_2, &&n_2/n_3\arrow[r, "i_2"] &n_2 \arrow[r, "p_2"] &n_3,
\end{tikzcd}
\end{center}
where each row and dashed column sequence is belonging to $\Delta$ class. Indeed, condition (i) and (ii) with commutativity are similar to the octahedron axiom in triangulated categories.

To prove condition (iii) of $\Delta$ exactness, we only have to assume $p'\circ f \neq S$, otherwise $g^1 = S$. Any nontrivial morphism can be decomposed as $p_m\circ i_{m-1} \circ \cdots \circ p_2 \circ i_1$ for $i_j, p_j$ fitting in some sequences of $\Delta$. Since $p'\circ f \circ i = S$, it turns out that $p'\circ f \circ i = \cdots \circ (\bar{p}\circ \bar{i}) \circ \cdots$ by commutativity, where $\bar{p}\circ \bar{i} = S$. Then $g^1$ is given by the composite of induced morphisms in condition (i) and (ii). Therefore, this $\Delta$ class with the trivial morphism $S$ and the law of commutativity forms an exact $\Delta$ class. Note that there is no isomorphism but identity in $\mathcal A$, thus condition (iv) is trivial.
\end{exa}

Weakly $\Delta$-Artinian means that there are no infinite chain of objects of $\mathcal A$
$$
\cdots \xrightarrow{i_3} E_{2} \xrightarrow{i_2} E_1 \xrightarrow{i_1} E
$$
with $E \prec E_1 \prec E_2 \prec \cdots$ and $E_{j} \xrightarrow{i_j} E_{j-1} \rightarrow F_j  \in \Delta$ for all $j >0$. Dually, weakly $\Delta$-Noetherian means that there are no infinite chain of objects in $\mathcal A$ 
$$
E \xrightarrow{p_1} E_1 \xrightarrow{p_2} E_2 \xrightarrow{p_3} \cdots
$$
with $E \succ E_1 \succ E_2 \succ \cdots$ and $F_j \rightarrow E_{j-1} \xrightarrow{p_j} E_{j} \in \Delta$ for all $j > 0$. Note that the definition of the weakly Noetherian is a bit stronger than Rudakov used in~\cite{MR1480783}. In this article it would be sufficient to consider such a stronger conditions proposed by Bridgeland~\cite{MR2373143}.

\begin{theorem}\label{tHN}
Given an exact $\Delta$-stability structure on a category $\mathcal A$ with $\mathrm{Hom}(A, B) = 0_c = S$ for $\Delta$-semistable objects $A \succ B$. If $\mathcal A$ is weakly $\Delta$-Artinian and weakly $\Delta$-Noetherian, then for every object $E \in \mathcal A$, there exists a unique Harder-Narasimhan sequence
$$
E_0 \xrightarrow{i_0} E_1 \xrightarrow{i_1} \cdots \xrightarrow{i_{n-1}} E_{n} \xrightarrow{i_{n}} E
$$ 
such that for $j > 0$, $E_{j-1} \xrightarrow{i_{j-1}} E_j \rightarrow F_{j-1} \in \Delta$, and $F_j$ are $\Delta$-semistable objects of $\mathcal A$ with 
$$
F_0 \succ F_1 \succ \cdots \succ F_n.
$$
\end{theorem}
\begin{proof}
Given an object $E$ of $\mathcal A$, if $E$ is $\Delta$-semistable then we are done and let $\mathcal E^0$ be the full subcategory consisting of all $\Delta$-semistable objects in $\mathcal A$. So we assume $E$ is not $\Delta$-semistable, thus there exists a sequence $E' \rightarrow E \rightarrow F \in \Delta$ with $E' \succ E$. Continuing the process we obtain a chain of objects, then the condition of weakly $\Delta$-Artinian and exactness of $\Delta$ imply that such a chain must terminate, i.e. we get a $\Delta$-semistable object $A$ such that $A \rightarrow E \rightarrow F' \in \Delta$ with $A \succ E$. Similarly, every object of $\mathcal A$ has a $\Delta$-semistable object $B$ such that $F'' \rightarrow E \rightarrow B$ with $E\succ B$ by the assumption of weakly $\Delta$-Noetherian and exactness of $\Delta$.

To construct the Harder-Narasimhan sequence of $E$, we first need the existence of a weakly initial object called maximally destabilising object (mdo) in $(E\downarrow \mathcal E^0)$. A mdo of an object $E \in \mathcal A$ is an object $B \in \mathcal A$ such that 
\begin{enumerate}
\item $F \rightarrow E \xrightarrow{p} B$ is a sequence in the exact class $\Delta$,
\item if $p': E \xrightarrow{p'} B'$ is a nontrivial morphism such that $E \succ B'$, then $B \preceq B'$, 
\item the equality in (ii) holds only if the morphism $E \xrightarrow{p'} B'$ factors through $E \xrightarrow{p} B \xrightarrow{h'} B'$ for a nontrivial morphism $h': B \xrightarrow{h'} B' \in \mathrm{Hom}(B, B')$. If $p = h\circ p'$ for a morphism $h: B' \xrightarrow{h} B$, then $h\circ h'$ is identity of $B$.
\end{enumerate}  
Note that a $\Delta$-semistable object is just its mdo.

Since $E$ is not a $\Delta$-semistable object of $\mathcal A$, we have such a sequence $A \rightarrow E \rightarrow E'$ with $\Delta$-semistable object $A$ and $A \succ E \succ E'$. Now we suppose $B$ is a mdo for $E'$. If $E \rightarrow B'$ with $\Delta$-semistable object $B'$ and $B' \preceq B$ then $B' \prec A$ and thus $\mathrm{Hom}(A, B') = 0_c$ by the assumption. Hence the exactness of the class $\Delta$ implies that the morphism $E \rightarrow B'$ factors through $E \rightarrow E' \rightarrow B'$, thus the object $B$ is also a mdo of $E $ since there is a sequence $A' \rightarrow E \rightarrow B \in \Delta$. Now repeating the argument for $E'$, and so on, the assumption of weakly $\Delta$-Noetherian implies that the process must terminate. It turns out that every object of $\mathcal A$ has a mdo.

Now we have a sequence $E' \rightarrow E \rightarrow B \in \Delta$ with $B$ a mdo of $E$ and $E' \succ E$ for a non-$\Delta$-semistable object $E$. Given a mdo $B'$ of $E'$ there are associated sequences $K \rightarrow E' \rightarrow B'$, $K \rightarrow E \rightarrow Q$ and $B' \rightarrow Q \rightarrow B$ in $\Delta$ by the assumption of exactness of $\Delta$. Since $B$ is a mdo of $E$, the second sequence and the condition (ii) imply that $B \prec Q$ and thus $B' \succ B$ by the $\Delta$-seesaw property. Repeating the process for $E'$ and so on, we obtain a chain of objects of $E$
$$
\cdots \rightarrow E^2 \rightarrow E^1 \rightarrow E^0 = E
$$
such that $E^0 \prec E^1 \prec E^2 \prec \cdots$ and with $\Delta$-semistable objects $F^j $ such that $E^{j+1} \rightarrow E^j  \rightarrow F^j \in \Delta$ for all $j \geq 0$. By the assumption of weakly $\Delta$-Artinian, this chain must terminate eventually and gives a Harder-Narasimhan sequence of $E$. 

Finally, to prove the uniqueness of HN sequence of $E$ we only need to prove the uniqueness of $E^{j+1}\xrightarrow{i} E^j \xrightarrow{p} F^i$. Indeed, given $E'^{j+1} \xrightarrow{i'}E^j \xrightarrow{p'} F'^j \in \Delta$ with a mdo $F'^j \simeq F^j$ such that the right square in following diagram is commutative
\begin{center}
\begin{tikzcd}[row sep=large, column sep=large]
E^{j+1} \arrow[r, "i"]\arrow[d, dashed, "h"]&E^j \arrow[d, "f"] \arrow[r, "p"] &F^j \arrow[d, "g"] \\
E'^{j+1} \arrow[r, "i'"] &E^j \arrow[r, "p'"] &F'^j,
\end{tikzcd}
\end{center}
and $f, g$ are isomorphisms, we have $\mathrm{Hom}(E^{j+1}, F'^j) = 0_c$ since $E^{j+1} \succ E^j \succ F'^j$ and $\mathrm{mdo}(E^{j+1}) \succ F'^j$. Thus there is a morphism $h: E^{j+1} \xrightarrow{h} E'^{j+1}$ so that the left square in the diagram above is commutative. It turns out that $h$ is an isomorphism as $f, g$ are isomorphisms. Indeed, $E^{j+1} \xrightarrow{i} E^j$ is a weakly terminal object. Therefore, each object $E$ in $\mathcal A$ has a Harder-Narasimhan sequence unique up to isomorphism.
\end{proof}

\begin{exa}
Let $\mathcal A$ be the category $\mathbb N^*$ with the exact $\Delta$ class as Example~\ref{eN}. We define a length function on $\mathcal A$ by $L(n) = l$ if $n= p_0p_1^{m_1}p_2^{m_2}\cdots p_l^{m_l}$ with $p_0 = 1$ and primes $p_j > p_{j+1}$ for $n \in \mathbb N^*$, and $n_1 \succ n_2 \succ n_3$ for each sequence $n_1 \rightarrow n_2 \rightarrow n_3 \in \Delta$ if there exists an integer $0 \leq \bar l \leq \min\{L(n_1), L(n_3)\}$ such that $p_{L(n_1)- \bar l} > p_{L(n_3)- \bar l}$ and $p_{L(n_1)- j} = p_{L(n_3)- j}$ for $0 \leq j < \bar l$. Then the partial order induces a $\Delta$-stability structure with $\Delta$-semistable objects $\{p_j^{m_j}\}$ for all primes $p_j$ and positive integers $m_j$ on $\mathcal A$ which is obviously weakly $\Delta$-Artinian and $\Delta$-Noetherian. It turns out that the HN sequence for each $n \in \mathbb N^*$ given by the exact $\Delta$ class with this $\Delta$-stability structure is the same as the one in Example~\ref{division}. 
\end{exa}

\begin{exa}[Wedderburn-Artin theory]
Suppose $R$ is a left semisimple ring with identity and $\mathcal A$ is the abelian category of finitely generated unitary left $R$-modules. Here we call $R$ left semisimple if each short exact sequence of left $R$-modules is split. Obviously the set of all nontrivial short exact sequences is an exact $\Delta$ class of $\mathcal A$. By split exactness and Zorn's lemma we have a decomposition $R = \bigoplus_{i\in I} S_i$ into simple left $R$-modules $S_i$ with $|I| < \infty$. Hence $A \in \mathcal A$ is a direct sum of a finite number of simple submodules and thus $\mathcal A$ is Artinian and Noetherian. Let $\{V_1, V_2, \dots, V_r\}$ be a full set of mutually nonisomorphisc simple left $R$-modules with division rings $D_j = \mathrm{End}_{R}(V_j)$, then $\{V_1^{\oplus n_1}, V_2^{\oplus n_2}, \cdots, V_r^{\oplus n_r}\}$ for all $n_1, n_2, \dots, n_r \in \mathbb N^*$ forms the full set of $\Delta$-semistable objects ordered as $V_1 > V_2 > \cdots > V_r$. Thus $R$ can be written as $V_1^{\oplus m_1} \oplus V_2^{\oplus m_2} \oplus \cdots \oplus V_r^{\oplus m_r}$ and $R \simeq \mathrm{End}_R(R) \simeq \mathrm{End}_R(V_1^{\oplus m_1}) \oplus \mathrm{End}_R(V_2^{\oplus m_2}) \oplus \cdots \oplus \mathrm{End}_R(V_r^{\oplus m_r}) \simeq \mathbb M_{m_1}(D_1) \oplus \mathbb M_{m_2}(D_2) \oplus \cdots \oplus \mathbb M_{m_r}(D_r)$, where $\mathbb M_{m}(D)$ is a $m \times m$ matrix over a division ring $D$.
\end{exa}

\begin{exa}
Any bounded t-structure of a triangulated category $\mathcal T$ induces a $\Delta$-stability data on $\mathcal T$. Indeed, let $\Delta$-class be the class of distinguished triangles associated to the bounded t-structure, i.e., for any $E \in \mathcal T, n\in \mathbb N$ there exists a distinguished triangle $A \rightarrow E \rightarrow B \rightarrow A[1]$ with $A\in \mathrm{ob} \mathcal D^{\leq n}, B \in \mathrm{ob} \mathcal D^{\geq n+1}$, and $\Delta$-stability structure is given by $A \succ B \succ C$ for $A \rightarrow B \rightarrow C \in \Delta$. Then $\mathcal T$ is weakly $\Delta$-Artinian and weakly $\Delta$-Noetherian by the boundedness of t-structure. It follows that $\mathcal T$ has a $\Delta$-stability data with $\Delta$-semistable objects in $\{\mathcal A[n]\}$, $n \in \mathbb N$. Here $\mathcal A$ is the heart of bounded t-structure.
\end{exa}

\begin{exa}
A semi-orthogonal sequence of full admissible triangulated subcategories $\mathcal T_1, \dots, \mathcal T_n \subset \mathcal T$ induces a $\Delta$-stability data on $\mathcal T$. Recall $\mathcal T_j \subset \mathcal T_i^{\bot}$ for $j < i$ and orthogonal complements $\mathcal T_n^{\bot}, \dots, \mathcal T_n^{\bot} \cap \cdots \cap \mathcal T_1^{\bot}$ are triangulated subcategories. Then for all object $E \in \mathcal T$ there exists a distinguished triangle $A \rightarrow E \rightarrow B \rightarrow A[1]$ with $A \in \mathcal T_i$ and $B \in \mathcal T_n^{\bot} \cap \cdots \cap \mathcal T_{n-j}^{\bot}$. The collection of these distinguished triangles forms a $\Delta$-class with the $\Delta$-stability structure $A \succ B \succ C$ for $A \rightarrow B \rightarrow C \in \Delta$. So it induces a $\Delta$-stability data on $\mathcal T$ with $\Delta$-semistable objects $\{\mathcal T_n, \dots \mathcal T_1, \mathcal T_n^{\bot} \cap \cdots \cap \mathcal T_{1}^{\bot}\}$.
\end{exa}

\subsection{$\Delta$-exact functor}\label{Dex}

Compared to $\Delta$-stability data in a category $\mathcal A$ studied in previous sections, the structure of the exact $\Delta$-class on $\mathcal A$ turns out to be much richer. Indeed, Quillen exact category is an additive category with the exact $\Delta$-class given by admissible short exact sequences which can be embedded into an abelian category. Thus Quillen exact class can be considered as a special exact $\Delta$-class with pull-back and push-out on $\mathcal A$. 

In the following we would discuss some properties of those functors preserving exact $\Delta$ structures, and relations of $\Delta$-stability data associated to relevant exact $\Delta$-class.
\begin{definition}
Suppose $(\mathcal A, \Delta_{\mathcal A})$ and $(\mathcal B, \Delta_{\mathcal B})$ are categories with exact $\Delta$ classes $\Delta_{\mathcal A}, \Delta_{\mathcal B}$, respectively. Then a functor $T: \mathcal A \rightarrow \mathcal B$ is called \textit{$\Delta$-exact} if $T(\Delta_{\mathcal A}) \subset \Delta_{\mathcal B}$ and $T(S_{\mathcal A}) = S_{\mathcal B}$.
\end{definition}

One could see that $T$ is a regular exact functor in case that $\mathcal A$ and $\mathcal B$ are abelian with all short exact sequences $\Delta_A$ and $\Delta_B$, respectively.  Indeed, if $\mathcal A$ is an extension closed full subcategory of an abelian or triangulated category, the inclusion functor is a $\Delta$-exact functor.
\begin{prop}\label{is}
Let $T: \mathcal A \rightarrow \mathcal B$ be a $\Delta$-exact functor, and $\mathcal A$ has a $\Delta$-stability data induced by the exact class $\Delta_{\mathcal A}$. Then $\mathcal B$ has a $\Delta$-stability data induced by $T(\Delta_{\mathcal A})$. 
\end{prop}
\begin{proof}
Since $T$ is $\Delta$-exact, $T(\Delta_{\mathcal A})$ forms an exact $\Delta$-class of $\mathcal B$. Hence $\mathcal B$ has a $\Delta$-stability data induced by $(T(\Delta_{\mathcal A}), \Phi, \mathcal P_{\Phi})$ with $\Delta$-semistable objects $\left(\mathcal B\setminus T(\Delta_{\mathcal A})\right) \bigcup \mathcal P_{\Phi}(T(\mathcal A))$.
\end{proof}
\begin{cor}
Suppose each functor $T^j: \mathcal A^j \rightarrow \mathcal B$ is $\Delta$-exact for $j \in J$, a small set, each $\mathcal A^j$ has a $\Delta$-stability data and $\mathrm{ob}(T^j(\mathcal A^j)) \cap \mathrm{ob}\left(T^{j'}\left(\mathcal A^{j'}\right)\right)= \emptyset$ for $j \neq j'$. Then $\mathcal B$ has a $\Delta$-stability data induced by $\bigcup_{j\in J}T^j(\Delta_{\mathcal A^j})$.
\end{cor}

It turns out that a $\Delta$-exact fully faithful functor $T$ would preserve Harder-Narasimhan sequences, or $\Delta$-stability data. So this argument leads to a natural question that whether we could find a preorder on the exact class $\Delta_{\mathcal B}$ containing $T(\Delta_{\mathcal A})$ such that there exists a $\Delta$-stability data induced by $\Delta_{\mathcal B}$ and each object in $T(\mathcal A)$ has a Harder-Narasimhan sequence over $T(\mathcal A)$. More precisely, given a family of full subcategories with $\Delta$-stability data $\{(\mathcal A^j, \Delta_{\mathcal A^j}) \}$, $j\in J$, we want to obtain a more general $\Delta$-stability data induced by $\Delta_{\mathcal B} \supset \bigcup_{j\in J}T^j(\Delta_{\mathcal A^j})$ in $\mathcal B$. Indeed, we have the following gluing theorem.
\begin{theorem}\label{glu}
Let $\{T^j: \mathcal A^j \rightarrow \mathcal B\}$ be a collection of $\Delta$-exact fully faithful functors, and each $\mathcal A^j$ has a $\Delta$-stability data induced by an exact class $\Delta_{\mathcal A^j}$ for $j \in J$, a small set. Assume $T^j(\mathcal A^j) \cap T^{j'}\left(\mathcal A^{j'}\right) = \emptyset$ for $j\neq j'$, then there exists a $\Delta$-stability data of $\mathcal B$ induced by an exact class $\Delta_{\mathcal B} \supset \bigcup_{j\in J}T^j(\Delta_{\mathcal A^j})$ with $\Delta_{\mathcal B}$-semistable objects $\bigcup_{j\in J}\mathcal P_{\Phi^j}(T^j(\mathcal A^j))$ and the partial order $(J, \{\Phi^J\})$ by lexicographic order if and only if  \begin{enumerate}
\item there exists a partial order $\leq$ on $J$ such that $\mathrm{Hom}\left(T^j(\mathcal A^j), T^{j'}\left(\mathcal A^{j'}\right)\right) = S$ if $j > j'$,
\item for each object $E$ of $\mathcal B \setminus\bigcup_{j\in J}T^j(\mathcal A^j)$ there is an unique sequence $E_1 \xrightarrow{i_m} E \xrightarrow{p_m} A^{j_m} \in \Delta_{\mathcal B}$ with $j_{m_1}:= j_m(E_1) > j_m:= j_m(E)$ for $A^{j_m} \in T^{j_m}(\mathcal A^{j_m})$,
\item $\mathrm{Hom}(E, T^j(\mathcal A^j)) = \mathrm{Hom}(A^{j_m}, T^j(\mathcal A^j))\circ p_m$ for $j \leq j_m$,
\item for a  large enough $M=M(E) \in \mathbb N$, $i_m(E)\circ i_m(E_1) \circ \cdots \circ i_m(E_M) = i_M$ with $i_M: E_M=A^{j_M} \xrightarrow{i_M} E$ for $A^{j_M} \in T^{j_M}(A^{j_M})$.
\end{enumerate}
\end{theorem}
\begin{proof}
Suppose $\left(\Delta_{\mathcal B}, \left(J, \Phi^J\right), \bigcup_{j\in J}\mathcal P_{\Phi^j}(T^j(\mathcal A^j))\right)$ forms a $\Delta$-stability data on $\mathcal B$ and the partial order restricted to $J$ induces a partial order on $J$ implying (i). Thus each object $E\in \mathcal B$ has a sequence $E^{n-1} \rightarrow E^n =E \rightarrow F^0 \in \mathcal B$ with $m(E^{n-1}) > m(E)=F^0$, and for a suitable $m\leq n$ we have $j(E^{n-m}) = j(E^{n-m+1}) = \cdots = j(E^n) =: j_m$. By the exactness of $\Delta_{\mathcal B}$ it turns out that there is a sequence $E^{n-m-1} \xrightarrow{i_m} E \xrightarrow{p_m} A^{j_m} \in \Delta_{\mathcal B}$ unique up to isomorphism and $A^{j_m} \in T^{j_m}(\mathcal A^{j_m})$. Similarly, there exists the unique sequence $A^{j_M} \xrightarrow{i_M} E \xrightarrow{p_M} E' \in \Delta_{\mathcal B}$ with $i_M=i_m(E)\circ i_m(E^{n-m-1})\circ \cdots$ and $A^{j_M} \in T^{j_M}(\mathcal A^{j_M})$.

On the other hand, the partial order on $J$ in (i) induces a partial order on $(J, \{\Phi^J\})$ by lexicographic order. For each $E\in T^j(\mathcal A^j)$ there is a natural HN sequence induced by $T^j(\Delta_{\mathcal A^j})$. So we only need to consider the objects $E$ in $\mathcal B \setminus\bigcup_{j\in J}T^j(\mathcal A^j)$. Since we have the sequence $E_1 \xrightarrow{i_m} E \xrightarrow{p_m} A^{j_m} \in \Delta_{\mathcal B}$ and $A^{j_m}$ has a HN sequence over $T^{j_m}(\mathcal A^{j_m})$, it leads to a sequence $E_1 \rightarrow E^{j_m}_n \rightarrow E^{j_m}_{n-1} \rightarrow E^{j_m}_{n-2} \rightarrow \cdots \rightarrow E^{j_m}_0=E$ so that $E^{j_m}_k \rightarrow E^{j_m}_{k-1} \rightarrow A^{j_m}_{k-1} \in \Delta_{\mathcal B}$, $k=0,\dots,n$, and $A^{j_m}_n > \cdots >A^{j_m}_0$. Note that $\left\{A^{j_m}_k \in \mathcal P_{\Phi^{j_m}}(T^{j_m}(\mathcal A^{j_m}))\right\}$ is given by the HN sequence of $A^{j_m}$. The assumption (iv) implies that there is a unique sequence of finite length of $E$, hence a HN sequence we expect over $\bigcup_{j\in J}\mathcal P_{\Phi^j}(T^j(\mathcal A^j))$.
\end{proof}

\begin{remark}
Note that if all $\mathcal A^j$ and $\mathcal B$ are additive categories with $S = 0$, then each null object in $T^j(\mathcal A^j)$ and $\mathcal B$ must be identified since any trivial morphism factors through the null object. In this case, we would assume for $j\neq j'$, $T^j(\mathcal A^j) \cap T^{j'}\left(\mathcal A^{j'}\right) = 0$, the null object of $\mathcal B$.
\end{remark}

For any category $\mathcal A$ with an exact class $\Delta_{\mathcal A}$ we can not expect each morphism could fit into a sequence in $\Delta_{\mathcal A}$ in general. However, if $\mathcal A$ is an additive category, then there exists an universal triangulated category $\mathcal B$ such that $\Delta_{\mathcal A}$ belongs to the class of distinguished triangles of $\mathcal B$ and each morphism in $\mathcal A$ can be completed to a distinguished triangle of $\mathcal B$. Note that the class of distinguished triangle of any triangulated category always forms an exact $\Delta$ class if we take out those distinguished triangles containing isomorphisms.
\begin{prop}
Let $\mathcal A$ be an extension-closed full subcategory of a triangulated category $\mathcal T$ with an exact class $\Delta$, a subset of distinguished triangles in $\mathcal T$. Let $\Delta_{\mathcal A}$ is the set of all distinguished triangles in $\Delta$ which terms lie on $\mathcal A$. Then $\Delta_{\mathcal A}$ is an exact class of $\mathcal A$.
\end{prop}
\begin{proof}
For any two complexes $F_1 \rightarrow E_1 \xrightarrow{p_1} E_2, F_2 \rightarrow E_2 \xrightarrow{p_2} E_3 \in \Delta_{\mathcal A}$ over $\mathcal A$, they are also distinguished triangles in $\mathcal T$, and the composite map $E_1\xrightarrow{p_2\circ p_1}E_3$ can be completed to be a distinguished triangle $F_1' \rightarrow E_1 \rightarrow E_3 \rightarrow F_1'[1]$. The octahedron axiom implies that $F_1 \rightarrow F_1' \rightarrow F_2 \rightarrow F_1[1]$ is also a distinguished triangle, and thus $F_1' \in \mathcal A$ since $\mathcal A$ is extension-closed. Hence $F_1' \rightarrow E_1 \xrightarrow{p_2 \circ p_1} E_3, F_1 \rightarrow F_1' \rightarrow F_2 \in \Delta_{\mathcal A}$. By the similar argument, for any two complexes $E_1 \xrightarrow{i_1} E_2 \rightarrow F_1, E_2 \xrightarrow{i_2} E_3 \rightarrow F_2 \in \Delta_{\mathcal A}$, there exists such complexes $E_1 \xrightarrow{i_2\circ i_1} E_3 \rightarrow F_1', F_1 \rightarrow F_1' \rightarrow F_2 \in \Delta_{\mathcal A}$.

For any commutative diagram of distinguished triangles
\begin{center}
\begin{tikzcd}[row sep=large, column sep=large]
E_1 \arrow[r, "i"] \arrow[d, dashed, "g^0"]&E_2 \arrow[d, "f"] \arrow{r} &F_1 \arrow[d, dashed, "g^1"] \arrow{r} &E_1[1] \arrow[d, dashed,  "{g^0[1]}"]\\
F_2 \arrow{r}&E_3 \arrow[r, "p"] &E_4 \arrow{r} & F_2[1]
\end{tikzcd}
\end{center}
the sufficient condition of the existence of morphisms $g^0$ and $g^1$ is $p\circ f \circ i = 0$ (See the book \cite{MR1950475}). Therefore, the class $\Delta_{\mathcal A}$ forms an exact class of complexes over $\mathcal A$.
\end{proof}

\begin{prop}
Given an additive category $\mathcal A$ with an exact $\Delta$-class $\Delta_{\mathcal A}$, then there exists an additive category $\mathcal B$ with an exact $\Delta$-class $\Delta_{\mathcal B}$, and a $\Delta$-exact fully faithful functor $U: \mathcal A \rightarrow \mathcal B$ such that each (non-isomorphic) morphism $f: E \rightarrow F$ in $\mathrm{Mor}(\mathcal A)$ can fit into a sequence in $\Delta_{\mathcal B}$.
\end{prop}
\begin{proof}
Let $\mathcal{K(A)}$ be the homotopy category of $\mathcal A$ which is a triangulated category with a class of distinguished triangles $E \xrightarrow{f} F \xrightarrow{\pi} C(f) \xrightarrow{[1]}\cdots$. For any sequence $E \xrightarrow{f} F \xrightarrow{g} G \in \Delta_{\mathcal A}$, we have distinguished triangles $E \xrightarrow{\bar{f}} C(g)[-1] \rightarrow C(\bar{f}) \rightarrow \cdots$ and $C(f) \xrightarrow{\bar g} G \rightarrow C(\bar g) \rightarrow \cdots.$ Then one can define a Null system $\mathcal N$, full saturated triangulated subcategory of $\mathcal{K(A)}$, such that $C(\bar f)[i] ,C(\bar g)[i] \in \mathcal N$ for all $f, g \in \Delta, i \in \mathbb N$, and obtain the localized triangulated category $\mathcal{B=K(A)/N}$ (see~\cite{MR2182076}).
\end{proof}

\begin{cor}\label{ts}
Suppose $\mathcal T$ is a triangulated category with a bounded t-structure. Then its heart $\mathcal A$ has a stability data if and only if $\mathcal T$ has a $\Delta$-stability data induced by the exact class $\Delta$ generated by the bounded t-structure and some short exact sequences of $\mathcal A$.
\end{cor}
\begin{proof}
Assume $\mathcal A$ has a stability data, and $T^j: \mathcal A[j] \rightarrow \mathcal T$ is the inclusion functor for all $j \in \mathbb Z$. The bounded t-structure. The truncation functors corresponding to the given bounded t-structure give the unique distinguished triangle $E_1 \xrightarrow{i_m} E \xrightarrow{p_m} A[j_m]$ as the one in the gluing theorem for all $E$. Let $\Delta$ be the collection of all short exact sequences of $\mathcal A[j]$ for all $j \in \mathbb Z$ and $E_1 \xrightarrow{i_m} E \xrightarrow{p_m} A[j_m]$ for all $E$. Thus $\Delta$ forms an exact class of $\mathcal T$ with the partial order $\mathbb Z\times \Phi$. Therefore, $\mathcal T$ has a $\Delta$-stability data induced by assumption of the gluing theorem above.\\
Conversely, the exact class $\Delta$ restricted to the heart $\mathcal A$ is collection of short exact sequences of $\mathcal A$ and thus the HN sequence of $E \in \mathcal A$ induced by $\Delta$ is a HN filtration of $E$ over $\mathcal A$. Hence $\mathcal A$ has a stability data.
\end{proof}

\begin{cor}
Suppose a triangulated category $\mathcal T$ admits a semi-orthogonal sequence $\mathcal T_n, \dots, \mathcal T_1$, then $\mathcal T$ has a $\Delta$-stability data induced by the exact class $\Delta$ generated by distinguished triangles associated to the semi-orthogonal sequence and each exact $\Delta_j$-class of $\mathcal T_j$ if there exists a $\Delta_j$-stability data on $\mathcal T_j$. 
\end{cor}

In the above cases $\mathcal A$ and $\mathcal B$ informally share the same $\Delta$-stability data $(\Delta, \Phi, \mathcal P_{\Phi})$. In geometric context, there may exist a complete and cocomplete small category $P$ so that $\Phi_{\mathcal A}$ and $\Phi_{\mathcal B}$ are  subcategories of $P$ with induced partial orders, and thus $\Phi_{\mathcal A} \neq \Phi_{\mathcal B}$ in general. However, if there is a $\Delta$-exact adjunction from $\mathcal A$ to $\mathcal B$, it may give us some constrains or relations between $\mathcal P_{\Phi_{\mathcal A}}$ and $\mathcal P_{\Phi_{\mathcal B}}$. The following proposition would be the key idea used later.
\begin{prop}\label{con}
Suppose $U \dashv T$ is a $\Delta$-exact adjunction from $(\mathcal A, \Delta_{\mathcal A})$ to $(\mathcal B, \Delta_{\mathcal B})$. Assume $U(\mathcal P_{\Phi_{\mathcal B}}) \subset \mathcal P_{\Phi_{\mathcal A}}$ and $G_{M(T(E))} \in \mathcal P_{\Phi_{\mathcal B}}$ is the maximal destabilizing factor object of $T(E)$ for $E \in \mathcal P_{\Phi_{\mathcal A}}$. Then $U(G_{M(T(E))}) \preceq E$. Similarly, if $T(\mathcal P_{\Phi_{\mathcal A}}) \subset \mathcal P_{\Phi_{\mathcal B}}$ and $E_{m(U(G))} \in \mathcal P_{\Phi_{\mathcal A}}$ is the minimal destabilizing factor object of $U(G)$ for $G \in \mathcal P_{\Phi_{\mathcal B}}$, then $T(E_{m(U(G))}) \succeq G$.
\end{prop}
\begin{proof}
Since $0_{\mathcal B} \neq \mathrm{Hom}(G_{M(T(E))}, T(E)) = T(\mathrm{Hom}(U(G_{M(T(E))}), E)) \circ \eta$, and $\eta: G_{M(T(E))} \rightarrow TU(G_{M(T(E))})$ is the unit induced by adjunction. By assumption $T(0_{\mathcal A}) = 0_{\mathcal B}$, thus we have $\mathrm{Hom}(U(G_{M(T(E))}), E) \neq 0_{\mathcal A}$. Furthermore, $U(G_{M(T(E))})$ is $\Delta_{\mathcal A}$-semistable as $G_{M(T(E))}$ is $\Delta_{\mathcal B}$-semistable. It turns out that $U(G_{M(T(E))}) \preceq E$, otherwise $\mathrm{Hom}(U(G_{M(T(E))}), E) = 0_{\mathcal A}$, a contradiction.
\end{proof}

Since there always exists a zero morphism induced by the exact class $\Delta_{\mathcal A}$ in any category $\mathcal A$, thus if $\mathcal A$ has finite products which are canonically isomorphic to finite coproduct, we can impose a commutative monoidal structure on $\mathrm{Hom}(\mathcal A)$ inducing additive structure on $\mathcal A$ (see~\cite{MR1712872}). Therefore, in the following we would like to focus on the study of $\Delta$-stability on additive categories.

\section{Weakly ample sequence in additive categories} \label{sec:was} 
In this section we introduce a notion of \textit{weakly ample sequence} in an additive category and its ambient triangulated category. \textit{Ample sequence} introduced by A.~Bondal and D.~Orlov~\cite{MR1818984} satisfies the condition of weakly ample sequence, but it is not so easy to find such ample sequence on the tilted heart of the triangulated category, even in the case of smooth projective varieties. Alternatively it may be possible to find a weakly ample sequence which allow us to build suitable slope polynomials on the tilted heart.

\subsection{Weakly ample sequence} 

Let us first recall the notions of the (left) Postnikov system (See the book \cite{MR1950475}): given a finite complex $A^{\bullet} = \left\{A^0\xrightarrow{d_0} A^1 \xrightarrow{d_1} A^2 \xrightarrow{d_2} \cdots \xrightarrow{d_{n-1}} A^n \right\}$ over $\mathcal T$, a diagram of the form
\begin{center}
\begin{tikzcd}[row sep=normal, column sep=small]
A^0 \arrow[rr, "d_0"] \arrow[dr, equal]& &A^1 \arrow[rr, "d_1"] \arrow{dr}&&A^2 \arrow{r} \arrow{dr} &\cdots&&\cdots\arrow{r} & A^{n-1} \arrow[rr, "d_{n-1}"] \arrow[dr]& & A^n  \arrow{dr} \\
&T^0 \arrow[ur] &&T^1\arrow[ll, "{[1]}"] \arrow{ur}&&T^2 \arrow[ll, "{[1]}"]&\cdots\arrow[l]&T^{n-2}\arrow[l] \arrow[ur]&& T^{n-1} \arrow[ur]\arrow[ll, "{[1]}"] &&T^n\arrow[ll, "{[1]}"]
\end{tikzcd}
\end{center}
is called a left Postnikov system subordinated to $A^{\bullet}$, and the object $T:= T^n[n] $ is called the convolution of the left Postnikov system. Here all triangles 
$$
T^{n-r} \rightarrow A^{n-r+1} \rightarrow T^{n-r+1} \xrightarrow{[1]} T^{n-r}[1]
$$
are distinguished triangles. The class of all convolutions of $A^{\bullet}$ is denoted by $\mathrm{Tot}(A^{\bullet})$. In general, $\mathrm{Tot}(A^{\bullet})$ may be empty, and sometimes it contains many non-isomorphic objects. If $\mathcal T = \mathrm{D^b}(\mathcal A)$ and $A^{\bullet}$ is a bounded complex over $\mathcal A$, then the unique convolution is just the complex $A^{\bullet}$ itself.

\begin{definition}\label{wa}
Let $\mathcal A$ be an extension-closed full additive subcategory of a $k$-linear triangulated category $\mathcal T$. A sequence of objects $L_i \in \mathcal A, i \in \mathbb Z$ is called \textit{weakly ample} if for some $m \in \mathbb N$ and each object $A \neq 0$ in $\mathcal A$ there exists an integer $i_0(A)$ such that for all $i' < i_0(A)$ the following conditions are satisfied:
\begin{enumerate}
\item A has weak $m$-resolution property: There exists $P(A) \in  \mathrm{Tot}\left( V_{i'}^{j=0, \dots, m}(A) \otimes_k L_{i'} [-2m]\right)$ for some $k$-vector spaces $V_{i'}^j(A)$, so that $K(A):= \mathrm C(P(A) \rightarrow A)[-1]$ has weak $m$-resolution property. 
\item If $j \neq 0, \dots, m$, then $\mathrm{Hom}(L_{i'}, A[j]) = 0$.
\end{enumerate}
\end{definition}

\begin{remark}
If $\mathcal A$ is trivial, then the weakly ample sequence is denoted by the null object $0$. Note that in the condition (i) of the definition above $P(A)$ is a convolution of the complex $\left\{V_{i'}^0(A) \otimes_k L_{i'}[-2m] \rightarrow V_{i'}^1(A) \otimes_k L_{i'}[-2m] \rightarrow \cdots \rightarrow V_{i'}^m(A) \otimes_k L_{i'}[-2m]\right\}$ over $\mathcal T$. By definition $K(A)$ has $m$-resolution property if $\exists\, P(K(A)) \in \mathrm{Tot}\left( V_{i'}^{j=0, \dots, m}(K(A)) \otimes_k L_{i'} [-2m]\right)$ for $i' < i'(K(A))$ so that $K_1(A):= \mathrm C(P(K(A)) \rightarrow K(A))[-1]$ has $m$-resolution property.
\end{remark}

\begin{definition} 
Given a weakly ample sequence $\{L_i\}$ in $\mathcal A$ as above, the object $L_0$ in $\{L_i\}$ is called a \textit{structure} object if for all $A$ in $\mathcal A$ and all $i < 0$ the following condition is satisfied: 
$$
\mathrm{Hom}(L_i, A[j]) \neq 0 \quad \Longrightarrow \quad d_{0, \min}(\mathcal A) \leq j \leq d_{0, \max}(\mathcal A).
$$ 
Here $d_{0, \max}(\mathcal A) := \max_{A \in \mathcal A}\{j \in~\mathbb Z \mid \mathrm{Hom}(L_0, A[j]) \neq 0 \}$ and $d_{0, \min}(\mathcal A) := \min_{A \in \mathcal A}\{j \in~\mathbb Z \mid \mathrm{Hom}(L_0, A[j]) \neq 0 \}$, and $d_0(\mathcal A) := d_{0, \max}(\mathcal A) - d_{0, \min}(\mathcal A)$ is called the dimension of $\mathcal A$ with respect to $L_0$. If $\mathcal A$ is trivial, then $d_0(\mathcal A) := -1$.
\end{definition}

For any ample sequence in $\mathcal A$, we always have $\mathrm{Hom}(L_{i'}, A[j]) = 0$ as $j \neq 0$ by definition, thus it is weakly ample with $m=0$. Indeed, each ample invertible sheave or ample locally free sheaves on the category of coherent sheaves of a projective variety $X$ spans a ample sequence, and the structure sheaf $\mathscr O_X$ is a structure object. Actually, the existence of structure objects in a weakly ample sequence is given by the following easy proposition.

\begin{prop}
Let $\{L_i\}$ be a weakly ample sequence in a $k$-linear additive category $\mathcal A$ and $\mathrm{Hom}(A, B[j]) = 0$ if $j>m_1$ or $j<m_2$ for some $m_1, m_2 \in \mathbb Z$ and all $A, B \in \mathcal A$, then there exists a weakly ample sequence $\{L'_i\}$ such that $L'_0$ is a structure object of $\mathcal A$. 
\end{prop} 
\begin{proof}
Assume $L_0$ is not a structure object, and there is an object $L_k$ in $\{L_i\}$ for some $k < 0$, and some $A$ in $\mathcal A$ such that Hom$(L_k, A[j]) \neq 0$ with $j < d_{0, \min}(\mathcal A)$, or $> d_{0, \max}(\mathcal A)$. Then we can replace $L_0$ in $\{L_i\}$ by an extension $L'_0$ of $L_0$ by $L_k$, i.e., $L_k \rightarrow L'_0 \rightarrow L_0 \rightarrow L_k[1]$, then $d'_{0,\min}(\mathcal A) < d_{0, \min}(\mathcal A)$, or $d'_{0, \max}(\mathcal A) > d_{0, \max}(\mathcal A)$, respectively. If $L'_0$ is not a structure object, we can proceed with this construction. Since $m_1$ and $m_2$ are finite numbers, the process must terminate and thus $L'_0$ is a structure object.
\end{proof}

Given an exact functor between triangulated categories, to decide whether it is fully faithful or even an equivalence, the existence of bijection between their spanning classes may provide a sufficient condition, see~\cite{MR2244106, MR1651025, MR1465519}. Indeed, a weakly ample sequence in an abelian category also may form a spanning class in it's ambient triangulated category by the following observations.
\begin{prop}\label{sc}
Given a weakly ample sequence $\{L_i\}$ in a heart $\mathcal A$ given by a bounded t-structure of a $k$-linear triangulated category $\mathcal T$ and $\mathrm{Hom}(A, B[j]) = 0$ for all $A, B \in \mathcal A$ and $j > d$. Suppose for each $A \in \mathcal A$ and $m, l \geq 0$ there exists a nontrivial morphism 
$$
P(K_l(A)[-1]) \in \mathrm{Tot}\left( V_{i'}^{j=0,\dots, m}(K_l(A)) \otimes_k L_{i'} [-2l]\right) \longrightarrow K_l(A)[-1]
$$
for some $k$-vector spaces $V_{i'}^j(K_l(A))$, $i' = i'(K_l(A))<< 0$ and 
$$K_l(A):= \mathrm C(P(K_{l-1}(A)[-1]) \rightarrow K_{l-1}(A)[-1])$$
with $K_{-1}(A)[-1] := A$ and $H^i(K_l(A))=0$ if $i \neq -1, \dots, m-1$, then $\{L_i\}$ is a spanning class of the ambient category $\mathcal T$, that is, for all objects $E$ in $\mathcal T$ the following conditions are satisfied:
\begin{enumerate}
\item $\mathrm{Hom}(L_i, E[j]) = 0$ for all $i, j \in \mathbb Z$, then $E \simeq 0$.
\item $\mathrm{Hom}(E[j], L_i) = 0$ for all $i, j \in \mathbb Z$, then $E \simeq 0$.
\end{enumerate}
\end{prop}
\begin{proof}
(i): Let $E$ in $\mathcal T$ be a non-trivial object such that $\mathrm{Hom}(L_i, E[j]) = 0$ for all $i, j$. We may assume there is a distinguished triangle: 
$$
H^n(E)[-n] \rightarrow E \rightarrow E' \rightarrow H^n(E)[1-n]
$$ 
with $E^n:= H^{n}(E) \neq 0$ and $H^{j}(E) = 0$ for $j < n$ induced by the bounded t-structure. Since $E^{n} \in \mathcal A$ one could find a convolution $P(E^n)$ with a nontrivial distinguished triangle  
$$
K_0(A)[-1] \rightarrow P(E^n) \rightarrow E^n \rightarrow K_0(A).
$$ Thus $0\neq \mathrm{Hom}(E^n, E[n])  =\mathrm{Hom}(K_0(A), E[n]) $. By assumption, there is also a distinguished triangle
$$
K_1(A)[-1] \rightarrow P(K_0(A)[-1]) \rightarrow K_0(A)[-1] \rightarrow K_1(A)
$$ 
which implies $\mathrm{Hom}(E^n, E[n])  = \mathrm{Hom}(K_0(A), E[n]) = \mathrm{Hom}(K_1(A)[1], E[n])$. Thus, recursively we have isomorphisms
$$
0\neq \mathrm{Hom}(E^n, E[n])  = \mathrm{Hom}(K_0(A), E[n]) = \mathrm{Hom}(K_1(A)[1], E[n]) = \cdots.
$$
However, for a fixed object $E \in \mathcal T$ there exists $i(E)$ such that $\mathrm{Hom}(K_l(A)[i'], E) = 0$ for $i' > i(E)$ and all $K_l(A)$, since $E$ is a convolution of a bounded complex over $\mathcal A$ given by the bounded t-structure. Indeed, one can easily see this by the spectral sequence of convolutions~\cite{MR1950475}. Note that $K_l(A)$ is also a convolution of a bounded complex over $\mathcal A$ by assumption. It turns out that $\mathrm{Hom}(K_l(A)[i'], E) = 0$ for all $i'>>0$, a contradiction.  Hence $E \simeq 0$.\\
(ii): We assume $E$ has a distinguished triangle 
$$
H^n(E)[-n-1] \rightarrow E' \rightarrow E \rightarrow H^n(E)[-n]
$$  
with $H^n(E) \neq 0$ and $H^{j}(E) = 0$ for $j \geq n$. As (i) we have isomorphisms
$$
0\neq \mathrm{Hom}(E[n], E^n)  = \mathrm{Hom}(E[n], K_0(A)) = \mathrm{Hom}(E[n], K_1(A)[1]) = \cdots.
$$
Similarly, for a fixed object $E \in \mathcal T$ there exists $i(E)$ such that $\mathrm{Hom}(E, K_l(A)[i']) = 0$ for $i' > i(E)$ and all $K_l(A)$. Thus it leads to $E$ has to be isomorphic to 0, so we are done.
\end{proof}

\begin{remark}
If $\{L_i\}$ is weakly ample with $m=0$ and the nontrivial morphism in condition (i) of weak ample sequences is epi in $\mathcal A$ of finite homological dimension, then $\{L_i\}$ forms a spanning class in $\mathcal T$. Indeed, this case is the same as the one in the ample sequence. 
\end{remark}

However, in general for arbitrary ambient triangulated category $\mathcal T$, it is unnecessary for a weakly ample sequence of $\mathcal A$ to be a spanning class of $\mathcal T$. Instead, we have to quotient out some objects which are (left and right) orthogonal to the weakly ample sequence of $\mathcal A$. Here, the right orthogonal complement of a subcategory $\mathcal{T'} \in \mathcal T$ is a full subcategory $\mathcal{T'}^{\bot}$ of all objects $A' \in \mathcal T$ such that $\mathrm{Hom}(B, A') = 0$ for all $B \in \mathcal{T'}$, and the left orthogonal complement $^{\bot}\mathcal{T'}$ is defined in the similar way. 

\begin{prop}
Let $\{L_i\}, i \in \mathbb Z$ be a weakly ample sequence of an extension closed full additive subcategory $\mathcal A$ in a $k$-linear triangulated category $\mathcal T$. Then $\{L_i\}^{\bot} := \{A \in \mathcal T \mid \mathrm{Hom}(L_i, A[j]) = 0, \mathrm{for\ all}\ i, j \in \mathbb Z\}$ is a triangulated subcategory of $\mathcal T$. If $\langle \{ L_i\}\rangle$ denotes the smallest triangulated subcategory containing $\{L_i \}$, then $\{L_i\}^{\bot} \simeq \langle\{L_i\}\rangle^{\bot}$.
\end{prop}
\begin{proof}
Clearly $0 \in \{L_i\}^{\bot}$ which is invariant under the shift functor $[1]$. For any distinguished triangle $A \rightarrow B \rightarrow C \rightarrow A[1]$ in $\mathcal T$ with $A, B \in \{L_i\}^{\bot}$, by the long exact sequence induced by $\mathrm{Hom}(L_i, -)$ one easily see that $\mathrm{Hom}(L_i, C[j]) = 0$ for all $i, j \in \mathbb Z$. Hence $C \in \{L_i\}^{\bot}$ and thus $\{L_i\}^{\bot}$ is a triangulated subcategory of $\mathcal T$.\\
For the second statement, since $\langle \{ L_i\} \rangle$ is a triangulated subcategory, $\langle \{L_i\} \rangle^{\bot}$ is also a triangulated subcategory of $\mathcal T$ and $\langle \{L_i\} \rangle^{\bot} \subset \{L_i\}^{\bot}$. For any $B \in \{L_i\}^{\bot}$, $^{\bot}B:= \{A \in \mathcal T \mid \mathrm{Hom}(A[j], B) = 0, \mathrm{for\ all}\ j \in \mathbb Z\}$ is a triangulated subcategory of $\mathcal T$ by the similar argument above and $\langle \{ L_i\} \rangle \subset {^{\bot}B}$ since $\langle \{ L_i\} \rangle$ is the smallest triangulated subcategory containing $\{L_i\}$. Hence $\mathrm{Hom}(\langle \{ L_i\} \rangle, B) = 0$, thus $B \in \langle \{ L_i\} \rangle^{\bot}$, and it turns out that $\{L_i\}^{\bot} \simeq \langle\{L_i\}\rangle^{\bot}$.
\end{proof}

Let us show that invariants similar to the Euler characteristic induced by a weakly ample sequence in an additive category $\mathcal A$ embedded in a triangulated category $\mathcal T$ can be computed explicitly for any convolution $T\in \mathrm{Tot}(A^{\bullet})$ over $\mathcal A$ even we know nothing about the differentials $d_i$'s.

\begin{lem}\label{tot}
Given a finite complex $A^{\bullet} = \left\{A^0\xrightarrow{d_0} A^1 \xrightarrow{d_1} A^2 \xrightarrow{d_2} \cdots \xrightarrow{d_{n-1}} A^n \right\}$ over $\mathcal A$ and $\mathrm{Tot}(A^{\bullet})$ is not empty, assume there is an object $L \in \mathcal A$ such that $\mathrm{Hom}\left(L, A^i[j]\right) = 0$ for all $i = 0, \dots, n$ and $j > m$, or $< 0$, then for $T \in \mathrm{Tot}(A^{\bullet})$  we have 
$$
\sum_{j=-2n}^{m-n} (-1)^j\mathrm{hom}(L, T[j]) = \sum_{i=0}^n\sum_{j=0}^m (-1)^{n-i+j} \mathrm{hom}(L, A^i[j]),
$$
where $\mathrm{hom} := \dim_k\mathrm{Hom}$.
\end{lem}
\begin{proof}
We proceed by induction on the length $n$ of the complex $A^{\bullet}$. For $n = 0$ it is trivial. Let $n>0$. The distinguished triangle
$$
T^{n-1} \rightarrow A^{n} \rightarrow T[-n]:=T^{n} \xrightarrow{[1]} T^{n-1}[1]
$$
implies the long exact sequence
$$
0 \rightarrow \mathrm{Hom}(L, T^n[-1]) \rightarrow \mathrm{Hom}(L, T^{n-1}) \rightarrow \mathrm{Hom}(L, A^n) \rightarrow \cdots \rightarrow \mathrm{Hom}(L, T^n[m]) \rightarrow 0
$$
and $\mathrm{Hom}(L, T^n[-j]) = \mathrm{Hom}(L, T^{n-1}[-j+1])$ for $j = 2, \dots, n$. Hence, 
$$
\sum_{j=-n}^m (-1)^j \mathrm{hom}(L, T^n[j]) = \sum_{j=0}^m (-1)^j \mathrm{hom}(L, A^n[j]) - \sum_{j=-n+1}^m (-1)^j \mathrm{hom}(L, T^{n-1}[j]) .
$$
Furthermore, $T^{n-1}[n-1]$ is a convolution of the complex $\left\{A^0\xrightarrow{d_0} A^1 \xrightarrow{d_1} A^2 \xrightarrow{d_2} \cdots \xrightarrow{d_{n-2}} A^{n-1} \right\}$, so $\sum_{j=-n+1}^m (-1)^j\mathrm{hom}(L, T^{n-1}[j]) = \sum_{i=0}^{n-1}\sum_{j=0}^m (-1)^{n-1-i+j}\mathrm{hom}(L, A^i[j])$ by assumption. It turns out that 
$$
\sum_{j=-n}^m (-1)^j\mathrm{hom}(L, T^n[j]) = \sum_{j=0}^m (-1)^j\mathrm{hom}(L, A^n[j]) + \sum_{i=0}^{n-1}\sum_{j=0}^m (-1)^{n-i+j}\mathrm{hom}(L, A^i[j]),
$$
so we are done.
\end{proof}

\begin{prop}
Let $\mathcal T'$ be a full saturated subcategory of a $k$-linear triangulated category $\mathcal T$, and for each $X \in \mathcal T'$ there exists a series of objects $\{L_i'\}, i' < i_0(X)$, belong to a weakly ample sequence $\{L_i\}$ of an extension closed full additive subcategory $\mathcal A$ of $\mathcal T$ such that $\left|\chi(L_{i'}, X)\right| = \left|\sum_j (-1)^j\chi(L_{i'}, A^j)\right| < \infty$ for some $A^j \in \mathcal A$. Then $\mathcal T'$ is a full saturated triangulated subcategory containing all convolutions $\mathrm{Tot}(A^{\bullet})$ of each complex $A^{\bullet}$ over $\mathcal A$.
\end{prop}
\begin{proof}
$0 \in \mathcal T'$, and $X \in \mathcal T'$ if and only if $X[1] \in \mathcal T'$. It is enough to check is that given a distinguished triangle $X \rightarrow Y \rightarrow Z \rightarrow X[1]$ and $X, Z \in \mathcal T'$, then $Y \in \mathcal T'$. Indeed we can choose $i' < \min\{i_0(X), i_0(Z)\}$ and the long exact sequence induced by the cohomological functor $\mathrm{Hom}(L_{i'}, -)$ implies $\chi(L_{i'}, Y) = \chi(L_{i'}, X) + \chi(L_{i'}, Z) = \sum_j (-1)^j\chi(L_{i'}, A^j) + \sum_j (-1)^j\chi(L_{i'}, B^j) = \sum_j (-1)^j\chi(L_{i'}, A^j \oplus B^j)$. Since $\mathcal A$ is extension closed full additive subcategory of $\mathcal T$, $A^j \oplus B^j \in \mathcal A$ for all $j$. Hence $Y$ is belong to $\mathcal T'$. Then the restriction of the shift functor $[1]$ and the collection of distinguished triangles $X \rightarrow Y \rightarrow Z \rightarrow X[1]$ with $X, Y, Z \in \mathcal T'$ make $\mathcal T'$ a full saturated triangulated subcategory of $\mathcal T$.
\end{proof}

\subsection{Chain of supported weakly ample sequence} 

\begin{definition}
We say that a morphism $f \in \mathrm{Hom}(E, F)$ for some $E, F\in \mathcal A$ fulfills the \textit{support} condition if the cone $\mathrm C(f):= \mathrm C\left(E \xrightarrow{f} F\right) \simeq \mathrm C(Q_f \rightarrow K_f)[-1]$ for some $Q_f, K_f \in \mathcal T$ such that there exists $k$-linear additive categories $\mathcal A^L$ and $\mathcal A^R$ with weakly ample sequences $\left(\{ L_i^L \}, L_0^L\right)$ and $(\{L_i^R\}, L_0^R)$, respectively, and for all objects $A$ in $\mathcal A$, $j \in \mathbb Z$, the following condition are satisfied: 
\begin{enumerate}
\item $\mathrm{Hom}_{\mathcal T}\left(K_f, A[j]\right) \simeq \mathrm{Hom}_{\mathcal T^L}\left(L_{i(K_f)}^L, T^L(A)[n(K_f)+j]\right)$  for some $T^L(A) \in \mathrm{Tot}(A^{L, \bullet})$ over $\mathcal A^L$.
\item $\mathrm{Hom}_{\mathcal T}\left(Q_f, A[j]\right) \simeq \mathrm{Hom}_{\mathcal T^R}\left(L_{i(Q_f)}^R, T^R(A)[n(Q_f)+j]\right)$ for some $T^R(A) \in \mathrm{Tot}(A^{R, \bullet})$ over $\mathcal A^R$.
\end{enumerate}
Here, $n(K_f), n(Q_f) \in \mathbb Z$, $d_0(\mathcal A^{L/R}) < d_0(\mathcal A)$ and $\mathcal T$, $\mathcal T^L$, and $\mathcal T^R$ denote ambient triangulated categories of $\mathcal A$, $\mathcal A^L$, and $\mathcal A^R$, respectively.
\end{definition}

\begin{theorem}\label{sp}
Let $(\{ L_i \}, L_0)$ be a weakly ample sequence with a structure object $L_0$ in a $k$-linear additive category $\mathcal A$ of dimension $d_0(\mathcal A) < \infty$. $(\{ L_i \}, L_0)$ is called a supported weakly ample sequence in $\mathcal A$ if there exists nontrivial morphisms $f_i : L_i \rightarrow L_{i+1}$ for all $i < 0$ fulfilling the support condition such that $\left(\{ L_i^L \}, L_0^L\right)$ and $(\{L_i^R\}, L_0^R)$ are supported weakly ample sequences of $\mathcal A^L$ and $\mathcal A^R$, respectively. Then for any $A$ in $\mathcal A$ there is an numerical polynomial $P_t(A) \in \mathbb Q[t]$ of degree $r \leq d_0(\mathcal A)$ such that $\chi_{\mathcal A}(L_{i}, A) = P_{-i}(A)$ for all $i < 0$.
\end{theorem}
\begin{proof}
We use induction on the dimension of $\mathcal A$ with respect to $L_0$. Suppose $d_0(\mathcal A) = 0$, 
by assumption of support conditions we have the following distinguished triangles
$$
L_i \longrightarrow L_{i+1} \longrightarrow \mathrm{C}(f_i) \longrightarrow L_i[1],
$$
$$
\mathrm C(f_i) \longrightarrow Q_{f_i} \longrightarrow K_{f_i} \longrightarrow \mathrm C(f_i)[1],
$$
which implies
$$
\chi_{\mathcal A}(L_i, A) - \chi_{\mathcal A}(L_{i+1}, A) = \chi_{\mathcal A}\left(K_{f_i}, A\right) - \chi_{\mathcal A}\left(Q_{f_i}, A\right).
$$
Since $d_{0}(\mathcal A^L) < d_0(\mathcal A)$ and $d_{0}(\mathcal A^R) < d_0(\mathcal A)$, which implies $\mathrm{Hom}_{\mathcal A}(K_{f_i}, A[j]) = \mathrm{Hom}_{\mathcal A}(Q_{f_i}, A[j]) = 0$, one finds $\chi_{\mathcal A}(K_{f_i}, A) = \chi_{\mathcal A}(Q_{f_i}, A) = 0$, thus $\chi_{\mathcal A}(L_0, A)= \chi_{\mathcal A}(L_1, A)= \cdots$ is a constant polynomial.\\
Assume $d_0(\mathcal A) > 0$. For each $i$ there exists bijections $\mathrm{Hom}_{\mathcal T}(K_{f_i}, A[j]) \simeq \mathrm{Hom}_{\mathcal T^L}(L_i^L, T^L[n+j])$, and thus $\chi_{\mathcal A}(K_{f_i}, A) = (-1)^{n}\chi_{\mathcal A^L}(L_i^L, T^L)$ with $d_0(\mathcal A^L) < d_0(\mathcal A)$ by assumption. Similarly, we have $\chi_{\mathcal A}(Q_{f_i}, A) = (-1)^{m}\chi_{\mathcal A^R}(L_{i+1}^R, T^R)$ with $d_0(\mathcal A^R) < d_0(\mathcal A)$. By induction $\chi_{\mathcal A}(K_{f_i}, A) $ and $\chi_{\mathcal A}(Q_{f_i}, A)$ are both polynomials of degree $< d_0(\mathcal A)$ since $\chi_{\mathcal A^L}(L_i^L, T^L)$ and $\chi_{\mathcal A^R}(L_i^R, T^R)$ do by the Lemma~\ref{tot}, so that $\chi_{\mathcal A}(L_i, A) = P_{-i}(A)$ is a numerical polynomial of degree $r \leq d_0(\mathcal A)$ (see, for example,~\cite{MR0463157, MR1771925}). Indeed, there exists $c_0, c_1, \dots, c_r \in \mathrm{Hom}(K(\mathcal A), \mathbb Z)$ such that $$
P_t(A) = c_{0}(A)\binom{t}{r} + c_1(A)\binom{t}{r-1} + \cdots + c_{r}(A).
$$
Here $K(\mathcal A)$ is Grothendieck group of $\mathcal A$.
\end{proof}

\begin{definition}
Given a supported weakly ample sequence $\{ L_i\}$ in an additive category $\mathcal A$, then we say that $\mathcal A$ has a chain of supported weakly ample sequences if for all $i < 0$,
\begin{enumerate}
\item $\left\{L^{(-1)}_i := L^L_i\right\}$ in $\mathcal A^{(-1)} := \mathcal A^L$, and $\left\{K_{f_i} \simeq L^{(-1)}_{i+1}\right\}$ in $\mathcal A[2]$,
\item $\left\{ L^{(1)}_i :=L^R_i\right\}$ in $\mathcal A^{(1)}:=\mathcal A^R$, and $\left\{Q_{f_i} \simeq L^{(1)}_{i+1} \right\}$ in $\mathcal A$, 
\item the cones $\mathrm C_{\mathcal A}\left(f^{(-1)}_i\right) \simeq \mathrm C_{\mathcal A^{(-1)}}\left(f^{(-1)}_i\right)$ and $\mathrm C_{\mathcal A}\left(f^{(1)}_i\right) \simeq \mathrm C_{\mathcal A^{(1)}}\left(f^{(1)}_i\right)$ as isomorphisms of distinguished triangles in $\mathcal T$,
\end{enumerate}
and $\left\{L^{(1)}_i\right\}$ and $\left\{L^{(-1)}_i\right\}$ form chains of supported weakly ample sequences of $\mathcal A^{(1)}$ and $\mathcal A^{(-1)}$, respectively. $\{L_i\}$ is called an \textit{ample chain} of $\mathcal A$ if $\mathcal A^{(-1)}$ is trivial and $\left\{L^{(1)}_i\right\}$ is an ample chain of $\mathcal A^{(1)}$.
\end{definition}

\begin{lem}\label{ef}
Suppose a $k$-linear additive category $\mathcal A$ has an ample chain of supported weakly ample sequence $(\{L_i\}, L_0)$ with $d_0(\mathcal A) < \infty$. Let $P_{t:= -i}(A) = \chi_{\mathcal A}(L_i, A)$ be the numerical polynomial of degree $r$ for some $A$ in $\mathcal A$. Then there exists additive categories $\left\{\left(\mathcal A^{(l)}, \left\{ L^{(l)}_j \right\}\right)\right\}$ with $d_0\left(\mathcal A^{(l)}\right) < d_0\left(\mathcal A^{(l-1)}\right)$, $l = 1,\dots, r$, such that 
$$
P_t(A) =\sum^r_{d=0}(-1)^d\chi_{\mathcal A^{(0)}}\left(L_0^{(d)}, A\right)\binom{t}{d}.
$$
Here $\mathcal A^{(0)}$ denotes $\mathcal A$, and $\mathcal A^{(l)}$ denotes $\mathcal A^{(1), \dots, (1)}$.
\end{lem}
\begin{proof}
We proceed by induction. If $r=0$ it is trivial. Assume $r > 0$ and by the induction hypothesis we have 
$$
\chi_{\mathcal A^{(0)}}(L_{i-1}, A) - \chi_{\mathcal A^{(0)}}(L_i, A) = - \chi_{\mathcal A^{(0)}}\left(L^{(1)}_i, A\right) = \sum^{r-1}_{d=0}(-1)^{d+1}\chi_{\mathcal A^{(1)}}\left(L_0^{(d+1)}, A\right)\binom{t}{d},
$$
where $L^{(1)}_i \simeq Q_{f_{i-1}}, \chi_{\mathcal A^{(1)}}\left(L_i^{(d+1)}, A\right) := \chi_{\mathcal A^{(1)}}\left(L^{(d+1)}_i,  T^{(1)}[m]\right)$ for some $T^{(1)} \in \mathrm{Tot}(A^{(1) \bullet})$ over $\mathcal A^{(1)}$. We claim that $\chi_{\mathcal A^{(0)}}\left(L_i^{(d)}, A\right) = \chi_{\mathcal A^{(1)}}\left(L_i^{(d)}, A\right)$. If $f(i)$ denotes the difference of $\chi_{\mathcal A}(L_i, A)$ and the term on the right hand side in the lemma, then it turns out that $f(i) - f(i-1) = 0$. Also $f(0)= 0$, thus $f$ vanishes identically.

To prove the claim, $d=1$ the assertion is trivial. Suppose $d > 1$ and consider the distinguished triangle $L_{i-1}^{(d-1)} \rightarrow L_{i}^{(d-1)} \rightarrow L_i^{(d)} \rightarrow L_{i-1}^{(d-1)}[1]$, one can find $\chi_{\mathcal A^{(0)}}\left(L_i^{(d)}, A\right) = \chi_{\mathcal A^{(0)}}\left(L_i^{(d-1)}, A\right) - \chi_{\mathcal A^{(0)}}\left(L_{i-1}^{(d-1)}, A\right) = \chi_{\mathcal A^{(1)}}\left(L_i^{(d-1)},  T^{(1)}[m]\right) - \chi_{\mathcal A^{(1)}}\left(L_{i-1}^{(d-1)}, T^{(1)}[m]\right) = \chi_{\mathcal A^{(1)}}\left(L_i^{(d)},  T^{(1)}[m]\right).$ The middle equality is given by the induction hypothesis and the last one is by the fact that the triangle we use is a distinguished triangle in $\mathcal T^{(n)}$ for all $n = 0, \dots, r$ by the assumption of ample chain.
\end{proof}

\begin{remark}
We may consider another chain of supported weakly ample sequence $(\{L_i\}, L_0)$ with $\{Q_{f_i} = 0\}$. In this case, we can drop the factor $(-1)^d$ in the lemma.
\end{remark}

In many geometric situations, the natural numerical polynomial of degree $n$ induced by some ample sequences in projective scheme of dimension $n$ over $k$ is the well-known Hilbert polynomial. The isomorphisms of Hom's needed in the theorem is given by Grothendieck-Verdier duality (see~\cite{MR0222093, MR1804902}).

\section{Slope polynomial and $\Delta$-stability in additive categories} \label{sec:nsp} 

To define a suitable numerical ordering on objects of an additive category which can be used to give a stability condition on this category, we first introduce the following definition.

\subsection{Slope polynomial and slope sequence} 

Let $\mathcal A$ be an additive category with an exact class $\Delta_{\mathcal A}$ and $K_{\Delta}(\mathcal A)$ be its Grothendieck group associated to $\Delta_{\mathcal A}$  in the way $K_{\Delta}(\mathcal A) = \mathcal F/ \mathcal R$, where $\mathcal F$ is the free abelian group on the object of $\mathcal A$ and $\mathcal R$ is the subgroup generated by classes $[A] - [A'] - [A'']$ for each sequence $A' \rightarrow A \rightarrow A'' \in \Delta$. Here we always assume $[0] = 0$, and if $A \simeq A'$ then $[A] = [A']$. A $\Delta$-\textit{additive function} from $\mathcal A$ to an abelian group $\Gamma$ is a function $f$ from the objects of $\mathcal A$ to $\Gamma$ so that $f(A) = f(A') + f(A'')$ for each sequence $A' \rightarrow A \rightarrow A''$ in $\Delta$. The function $[-]: A \mapsto [A]$ from $\mathcal A$ to $K_{\Delta}(\mathcal A)$ defines a $\Delta$-additive function which has the universal property as following
\begin{center}
\begin{tikzcd}[row sep= huge, column sep=large]
\mathcal A  \arrow{rd}[swap]{f} \arrow{r}{[-]} & K_{\Delta}(\mathcal A) \arrow{d}{f'} \\
& \Gamma
\end{tikzcd}
\end{center}
It means that any $\Delta$-additive function $f$ from $\mathcal A$ to $\Gamma$ induces a unique group homomorphism $f': K_{\Delta}(\mathcal A) \rightarrow \Gamma$ such that $f'([A]) = f(A)$ for each $A \in \mathcal A$. If $\mathcal A$ is an abelian or exact category and $\Delta$ is a class of all short exact sequences in $\mathcal A$, then $K_{\Delta}(\mathcal A)$ is regular Grothendieck group $K(\mathcal A)$ (See the book~\cite{MR3076731}). 

\begin{definition}{\cite{MR2084563}}\label{ps}
A linear independent system $(x_0, x_1, \dots, x_r) : K_{\Delta}(\mathcal A) \rightarrow \mathbb Z$ of $\Delta$-additive functions is called \textit{positive} if for all $A \in \mathcal A$ the conditions below hold:
\begin{align*}
x_0(A) &\geq 0,\\
x_0(A) &= 0 \Rightarrow x_1(A) \geq 0, \\
x_0(A) &= x_1(A) = 0 \Rightarrow x_2(A) \geq 0, \\
\vdots & \\
x_0(A) &= \cdots =x_{r-1}(A) = 0 \Rightarrow x_r(A) > 0.
\end{align*}
If $x_i(A) = 0$ for all $i$ implies $A = 0$, then the positive system is called \textit{exhaustive}. Given a positive system $(x_0, \dots, x_r)$ on $\mathcal A$, one can define the \textit{standard vector slope} of an object $A \in \mathcal A$ with respect to this system is the vector
$$
\phi(A) = \left( \underbrace{1,\dots, 1}_s, \frac{x_{s+1}}{x_s}, \frac{x_{s+2}}{x_s}, \dots, \frac{x_r}{x_s}\right),
$$
where $s = \min_i\{ x_i(A) \neq 0\}$. We also can define the \textit{scalar slope} of $A$ by
$$
\mu(A) = \left( \underbrace{1,\dots, 1}_s, \frac{x_{s+1}}{x_s}, \underbrace{1, \dots, 1}_{r-s-1}\right).
$$
\end{definition}
The slope ordering on $\mathcal A$ is given by the prescription: $A \preceq B \Leftrightarrow \phi(A) \leq \phi(B)$ by the lexicographical order. One can check that this slope ordering actually fulfills the $\Delta$-seesaw property by comparing the slope component, and thus may construct $\Delta$-stability data (Sec.~\ref{sec:stab}). 

In the following we would provide a way to construct a positive system giving a vector slope ordering on an additive category $\mathcal A$. We first make a definition below.
\begin{definition}\label{sp}
A \textit{slope polynomial} $P_t(A) \in \mathbb Q[t]$ on an additive category $\mathcal A$ with a class $\Delta$ of sequences is a $\Delta$-additive polynomial with a finite degree for an indeterminate $t$ such that the non-zero highest coefficient for every $A \neq 0 \in \mathcal A$ is always positive. Precisely,
$$
P_t(A) = \sum_{i=0}^n a_i(A) t^i
$$
with the non-zero highest coefficient $a_{n_h} > 0$ for all $A\in \mathcal A$. Note that $n_h$ and $a_i$ depend on $A$. Then the exhaustive positive system is given by the vector of coefficients $(a_n, a_{n-1},\dots,a_0)$.
\end{definition}

\begin{remark}
Indeed the slope polynomial over $\mathbb Z$ is a characteristic function for an abelian category introduced by A.~N.~Rudakov~\cite{MR1480783}.
\end{remark}

\begin{definition}
Let $\mathcal A$ be an extension-closed full subcategory of a $k$-linear triangulated category $\mathcal T$ with an induced exact class $\Delta$. A sequence $L^s_i \in \mathcal T$, $i\in \mathbb Z$, is called a \textit{slope sequence} with respect to $\mathcal A$ if for all $A \in \mathcal A$, $i' < i'(A) \in \mathbb Z$, we have
\begin{enumerate}
\item $\chi(L^s_{i'}, A) := \sum_j(-1)^j\dim_k\left(\mathrm{Hom}(L^s_{i'}, A[j])\right) \neq 0$,
\item $\chi(L^s_{i'}, A) = P_{t= -i'}(A) =\sum^r_{d=0}c_d(A)\binom{t}{d}$ with $c_d \in \mathbb Z$ for a slope polynomial $P_t$ on $\mathcal A$.
\end{enumerate}
\end{definition}

The traditional and natural ordering tool in additive categories is a vector slope induced by some positive systems. The $\Delta$-stability structure given by a vector slope associated to some slope polynomials is called the \textit{numerical} $\Delta$-stability structure.
\begin{prop}\label{hnf}
Given a slope sequence $\{L^s_i\}$ for an additive category $\mathcal A$ with an exact class $\Delta$ of sequences ordered by the vector slope induced from $P_t$, then $\mathcal A$ is weakly $\Delta$-Artinian. Furthermore, if $\mathcal A$ is a full subcategory of a Noetherian abelian category, then $\mathcal A$ is weakly $\Delta$-Noetherian. Therefore, if $\mathrm{Hom}(A, B) = 0$ for any $\Delta$-semistable objects $A \succ B$, then $\mathcal A$ has the Harder-Narasimhan property, and thus a $\Delta$-stability data.
\end{prop}
\begin{proof}
Suppose there exists an infinite chain of objects in $\mathcal A$
$$
\cdots \xrightarrow{i_3} E_{3} \xrightarrow{i_2} E_2 \xrightarrow{i_1} E_1
$$
with $E_1 \prec E_2 \prec E_3 \prec \cdots$ and $E_{j+1} \xrightarrow{i_j} E_j \rightarrow F_j \in \Delta$ for all $j >0$. Since $P_t$ is additive, it leads to $\deg(P_t(E_j)) \geq \deg(P_t(E_{j+1}))$ for $E_{j+1} \xrightarrow{i_j} E_j$. Thus for large enough $j$ we have $\deg(P_t(E_j)) = \deg(P_t(E_{j+1})) = \cdots = d$ and $c_d(E_j) \geq c_d(E_{j+1}) \geq \cdots$. By assumption, $c_d(E_j) \in \mathbb N\setminus \{0\}$, this chain can not decrease infinitely, i.e. 
$$
c_d(E_s) = c_d(E_{s+1}) = \cdots = c
$$
for some large enough $s$. Again by the additivity of $P_t$, we have the following inequality
$$
(c, c_{d-1}(E_s), \dots, c_0(E_s)) > (c, c_{d-1}(E_{s+1}), \dots, c_0(E_{s+1}))
$$
which thus implies $E_s \succ E_{s+1}$, a contradiction.

For the second statement, given any infinite chain of objects of $E = E_1$
$$
E_1 \xrightarrow{p_1} E_2 \xrightarrow{p_2} E_3 \xrightarrow{p_3} \cdots, 
$$
with $F_j \rightarrow E_j \xrightarrow{p_j} E_{j+1} \in \Delta$ for all $j$, and by exactness of $\Delta$ each composite maps $E \xrightarrow{p_j'} E_{j+1}$ with $p_j':= p_{j-1} \circ \cdots \circ p_1$ can fit into a short exact sequence in the ambient abelian category
$$
0\rightarrow L_j \rightarrow E \xrightarrow{p_j'} E_j \rightarrow 0
$$
which leads to the following commutative diagram of exact sequences
\begin{center}
\begin{tikzcd}[row sep=large, column sep=large]
0\arrow{r}&L_{j-1}\arrow{r} \arrow[d, "f_j"]&E \arrow[d, "\mathrm{Id}"] \arrow[r, "p'_{j-1}"] &E_j \arrow[d, "p_j"] \arrow[r]& 0 \\
0\arrow{r} &L_{j} \arrow{r}&E \arrow[r, "p_j'"] &E_{j+1} \arrow[r]& 0
\end{tikzcd}
\end{center}
where $L_{j-1} := \ker(p'_{j-1})$ and $f_j$ is an inclusion by the snake lemma for all $j \geq 0$. Then we have a chain of subobjects of $E$
$$
0 \neq L_1 \subset L_2 \subset \cdots \subset E
$$
and this chain must terminate since $\mathcal A$ is Noetherian. Thus $\mathcal A$ is weakly $\Delta$-Noetherian, and each object in $\mathcal A$ has a Harder-Narasimhan sequence.
\end{proof}

\begin{cor}
Each extension closed full subcategory $\mathcal A$ of the category of coherent sheaves on a projective scheme over a field $k$ has a Harder-Narasimhan filtration unique up to isomorphism. 
\end{cor}
\begin{proof}
Let $\Delta$ be the collection of all short exact sequence lying over $\mathcal A$ and the vector slope induced by Hilbert polynomial. Hence $\Delta$ is an exact class by Snake lemma and universality, and the vector slope is a $\Delta$-stability structure. Since the category of coherent sheaves is Noetherian, $\mathcal A$ is weakly $\Delta$-Noetherian. Thus the property~\ref{hnf} implies that $\mathcal A$ has the Harder-Narasimhan property.
\end{proof}

If a $\Delta$-stability data is induced by any slope polynomial of degree $n$ with the vector slope ordering on $\mathcal A$, it is called a $\Delta$-stability data of degree $n$ on $\mathcal A$. Indeed, once one can find a $\Delta$-stability data associated to a slope polynomial on $\mathcal A$, there exists a family of polynomials giving the same $\Delta$-stability data by the following proposition.
\begin{prop}\label{deform}
Suppose there is a $\Delta$-stability data $\left(\Delta, \Phi, \{ \mathcal P_{\phi}\} \right)$ associated to a slope polynomial on an additive category $\mathcal A$. Then there exists an universal family of slope polynomials inducing the same $\Delta$-stability data on $\mathcal A$.
\end{prop}
\begin{proof}
Let $P_t$ be the slope polynomial of degree $=n$ inducing this $\Delta$-stability data on $\mathcal A$, thus $P_t(E)=\sum^r_{i=0}a_i(E) t^i$ with $a_r(E) \neq 0$ for some $E \in \mathcal A$. Given any polynomial $Q_t$ of degree $=r < n$ with constant coefficients over $\mathbb Z$, then $P_t^Q(E) = P_t(E) + a^r(E)\cdot Q_t$, $a^r(E) = \sum_{i=r+1}^n c_ia_i(A)$ for $c_i \in \mathbb Z$, is also a slope polynomial for any $E$ in $\mathcal A$ with the same $\Delta$-stability data $\left(\Phi^Q, \{ \mathcal P_{\phi^Q}\} \right) = \left(\Phi, \{ \mathcal P_{\phi}\} \right)$. Indeed, the vector slopes induced by $\phi^Q$ and $\phi$ are equivalent, that is, $\phi(A) \leq \phi(B)$ if and only if $\phi^Q(A) \leq \phi^Q(B)$ for all $A, B \in \mathcal A$ by definition, thus $\mathcal P_{\phi^Q} = \mathcal P_{\phi}$. Hence the family $\left\{P^Q_t\right\}$ indexed by all polynomials $Q_t$ of degree $< r$ over $\mathbb Z$ leads to the $\Delta$-stability data $\left(\Delta, \Phi, \{ \mathcal P_{\phi}\} \right)$ on $\mathcal A$.
\end{proof}

Given any sequence $\{L_i\}$ in $\mathcal A$, there may exist a numerical polynomial $P_{t=-i}(A) = \chi(L_i, A)$ for all $A$, but it is unnecessary to be a slope polynomial, or define an exhaustive positive system automatically. For instance, the sequence generated by structure sheafs $\mathcal O_{kH}$ of closed subschemes $Y_k$ associated to ideal sheaf $\mathcal O(-kH)$ in a projective scheme $X$ of finite type over a field $k$ with ample invertible sheaf $\mathcal O(H)$ is not a slope sequence, but $ \chi(\mathcal O_{kH}, A)$ induces a numerical polynomial $P_k(A)$. 
\begin{theorem}\label{asp}
Suppose the numerical polynomial $P_{t:= -i}(A) = \chi_{\mathcal A}(L_i, A)$ for any $A\in \mathcal A$ is induced by an ample chain of supported weakly ample sequence $(\{L_i\}, L_0)$ with $d_0(\mathcal A) < \infty$. If we assume 
\begin{enumerate}
\item $\chi_{\mathcal A}\left(L_i, A\right) \neq 0$ for $A \neq 0 \in \mathcal A$, $i << 0$, 
\item $\deg P_t(L_i) \geq \deg P_t(A)$ for all $A \in \mathcal A$,
\item $\deg P_t\left(L^{(d)}_i\right) = \deg P_t\left(L^{(d+1)}_i\right) + 1$ for all $d = 0, \dots, \deg P_t(L_i)$, 
\end{enumerate}
then the vector of coefficients of the numerical polynomial $P_t(A)$ gives an exhaustive positive system, that is, the leading coefficient of $P_t(A)$ is either positive or negative simultaneously for all $A \in \mathcal A$ with the same degree. Indeed, if $P_t(L_i)$ is a slope polynomial, then $P_t(A)$ is also a slope polynomial for each $A \in \mathcal A$.
\end{theorem}
\begin{proof}
$\chi_{\mathcal A}(L_i, A) \neq 0$ for $i << 0$ by the definition of weak ampleness, thus $P_t(B) > 0$, or $<0$ for all $B$ with $\deg P_t(B) = 0$. Indeed, if $P_t(B_1) > 0$ and $P_t(B_2) < 0$ for some $B_i$ with $\deg P_t(B_i) = 0, i = 1,2$, then one can find suitable integers $a_1$ and $a_2$ such that $P_t(a_1B_1 \oplus a_2B_2) = 0$, a contradiction.\\
Assume $\deg P_t(L_i) = r$ and consider such a triangle
$$
L_i^{(r-1)} \longrightarrow L_{i+1}^{(r-1)} \longrightarrow L_i^{(r)} \longrightarrow L_i^{(r-1)}[1],
$$
which is a distinguished triangles in $\mathcal T$ by assumption, and thus $\deg P_t\left(L_i^{(r)}\right) = 0$ and $\deg P_t\left(L_i^{(r-1)}\right) = \deg P_t\left(L_{i+1}^{(r-1)}\right) = 1$. It leads to the polynomial $P_t\left(L_i^{(r-1)}\right)$ is of the form $a_1t + a_0 + i b_0$ where $P_t\left(L_0^{(r-1)}\right) = a_1t + a_0$ and $0 \neq b_0 = P_t\left(L_0^{(r)}\right) = P_t\left(L_{-1}^{(r)}\right) = \cdots$. If there exists an object $B$ with $\deg P_t(B) = 1$ such that the leading coefficient of $P_t(B)$ is $-a_1$, then for $i <<0$, one can see that $P_t\left(L_i^{(r-1)}\oplus B\right) = ib_0 + O(i^0)$ implying $L_i^{(r-1)}\oplus B$ is an object with $\deg P_t\left(L_i^{(r-1)}\oplus B\right) = 0$, but the plus-minus sign of $P_t\left(L_i^{(r-1)}\oplus B\right)$ is different from the one of $P_t\left(L_i^{(r)}\right)$, a contradiction.\\
By the same argument it turns out that the sign of the leading coefficient of $P_t(A)$ for all $A \in \mathcal A$ with the same degree $n$ is determined by the one of $P_t\left(L_0^{(r-n)}\right)$. Hence, the vector of coefficients of $P_t(A)$ can be used to define a vector slope on $\mathcal A$. Moreover, if the numerical polynomial $P_t(A)$ is given by an ample chain $\{L_i\}$ of $\mathcal A$ such that $P_t(L_i)$ is a slope polynomial, then so do $P_t(A)$ for any $A \in \mathcal A$. 
\end{proof}

\begin{remark}
Suppose $ \mathcal A$ is the heart of a bounded t-structure of a triangulated category with a weakly ample sequence $\{L_i\}$ as the one in Prop.~\ref{sc}. If $\{L_i\}$ forms an ample chain of supported weakly ample sequence then we may have $\deg P_t(L_i) \geq \deg P_t(A)$ for all $A \in \mathcal A$. However,  in general it can not be expected to have the condition of $\deg P_t\left( L^{(n)}_i \right) = \deg P_t\left( L^{(n+1)}_i \right) + 1$. For the derived category $\mathrm{D^b}(X)$ of bounded coherent sheaves of a projective scheme $X$ over $k$, $\mathcal A \subset \mathrm{D^b}(X)$ is the abelian category of coherent sheaves with a ample invertible sheaf $\mathcal O(H)$, the existence of dual and tensor product of ample chain leads to the sufficient conditions in the theorem. 
\end{remark}

Note that if a chain which is not ample induces a numerical polynomial, it can happen that the coefficients of this polynomial do not form a exhaustive positive system even the chain satisfies the condition in theorem~\ref{asp}. However, for a chain $\{L_i\}$ in $\mathcal A$ which is tilted ample, i.e., $\left\{ L_i^{(-1)}\right\}$ in $\mathcal A^{(-1)}$ is an ample chain and $\mathcal A^{(1)}$ is trivial, one may form a slope sequence on $\mathcal A$ by some objects $L_i^s \in \mathcal T $. Indeed, such slope sequences always exist in a normal projective scheme over a perfect field $k$ which can be used to construct $\Delta$-stability of degree $n-1$.
\begin{prop}\label{bound}
Suppose $\{L_i\}$ is a tilted ample chain in $\mathcal A$ such that $\left\{ L_i^{(-1)}\right\}$ is an ample chain in $\mathcal A^{(-1)}$, and for all $A \in \mathcal A$ the induced tilted numerical polynomial $\widetilde P_{t=-i}(A) = -\chi(L_i, A) = \sum^r_{d=0}c_d(A)\binom{t}{d}$, $i << 0$, fulfills boundedness conditions:
\begin{enumerate}
\item If $c_r(A) \neq 0$, then $c_{r-1}(A) \geq c_r(A) \cdot C_1$ for a fixed $C_1 \in \mathbb Z$;
\item If $c_r(A) \neq 0, c_{r-1}(A) = c_r(A) \cdot C_1$, then $c_{r-2}(A) > c_r(A)\cdot C_2$ for a fixed $C_2 \in \mathbb Z$;
\item If $c_r(A) = 0$, then $\widetilde P_t(A)$ is a slope polynomial.
\end{enumerate}
Then there exists a slope sequence $\{L_i^s\}$ corresponding to a slope polynomial $P^s_t(A) = \chi(L_i^s, A)$ of degree $< r$.
\end{prop}
\begin{proof}
Since $\{L_i\}$ is tilted ample, the argument of Lemma~\ref{ef} implies 
$$
\widetilde P_t(A) = -\chi(L_i, A) = \sum^{r}_{d=0}(-1)^d\chi_{\mathcal A^{(0)}}\left(L_0^{(d)}, A\right)\binom{t}{d},
$$
with $c_{d}(A) =  (-1)^{d}\chi_{\mathcal A^{(0)}}\left(L_0^{(d)}, A\right)$ for $d=0, \dots, r$. Let $n = \binom{t}{r} + C_1\binom{t}{r-1} + C_2\binom{t}{r-2} > 0$ as $t >>0$, and we define $L_i^s = \mathrm C\left(L_0^{(r), \oplus n}[-r+1] \rightarrow L_i \right)$ for $i << 0$ and $n > 0$, thus the associated numerical polynomial would be
$$
P_t^s(A) =\left(c_{r-1}(A) - c_r(A)C_1\right)\binom{t}{r-1} + \left(c_{r-2}(A) - c_r(A)C_2\right)\binom{t}{r-2} + \sum^{r-3}_{d=0}c_{d}(A)\binom{t}{d}.
$$
By assumption $c_{r-1}(A) - c_r(A)C_1 \geq 0$ and if equality holds, we have $c_{r-2}(A) - c_r(A)C_2 > 0$. Therefore, $P^s_t(A) = \chi(L_i^s, A)$ is a slope polynomial on $\mathcal A$.
\end{proof}

\begin{remark}
Indeed this construction of the property also holds for any polynomial fulfilling the same boundedness conditions above.
\end{remark}

In algebraic geometry, those boundedness conditions are closely related to boundedness of pure sheaves, moduli spaces of semistable sheaves, Hodge Index theorem, Bogomolov inequality, \dots, of basic importance to many recent developments. In following sections we would like to turn our attention to $\Delta$-stability data on normal projective schemes over a perfect field $k$.

\subsection{$\Delta$-stability on nonsingular curves} \label{zero}

To close this section we  finally consider cases of nonsingular proper curves $X$ over an algebraically closed field $k$, i.e., integral regular proper scheme of Krull dimension $1$ over $k$. Such a curve is necessarily projective, thus by Cor.~\ref{ts} and Lemma~\ref{ef} each ample invertible sheaf $\mathcal O(H)$ forms an ample chain with the $\Delta$-stability of degree $1$ in $\mathrm{D^b}(X) := \mathrm{D^b(Coh}(X))$ induced by the slope polynomial 
$$
P_t(E) = \chi(\mathcal O(-tH), E) = -\chi(\mathcal O_H, E)t + \chi(\mathcal O_X, E)
$$
for $E \in \mathrm{Coh}(X)$ and $\mathcal O_H$ is the structure sheaf of zero subscheme $H$ in $X$. Indeed, this $\Delta$-stability of degree $1$ is equivalent to Mumford-Takemoto ($\mu$-) stability, Gieseker stability and Bridgeland stability on curves.

Assume $X$ is $\mathbf P^1$ with the ample chain $\mathcal O(-t), t \in \mathbb Z$, and thus semistable subcategory $\mathcal P_{\Phi} = $ \{torsion sheaves$[i]$\} $\cup \{\mathcal O(t)[j]\}_{t\in \mathbb Z}$, $i, j \in \mathbb Z$. Let $\widetilde{\mathcal A}$ be the extension-closed full subcategory of $\mathrm{D^b}(X)$ spanned by $\{ \mathcal O(-t)[1], \mathrm{torsion}, \mathcal O(t-1) \}_{t >0}$, that is, for $\widetilde E \in \widetilde{\mathcal A}$ each factor of HN$\left(\widetilde E\right)$ belongs to $\{ \mathcal O(-t)[1], \mathrm{torsion}, \mathcal O(t-1) \}$. Then $\widetilde{\mathcal A}$ is the heart of $\mathrm{D^b}(X)$ given by a bounded t-structure~\cite{MR1327209}. 

Obviously, $P_t(E) = \chi(\mathcal O(-t), E) = -\chi(\mathcal O_p, E)t + \chi(\mathcal O_X, E)$ is not a slope polynomial on $\widetilde{\mathcal A}$, and $\{\mathcal O(-t)[1]\}_{t > 0}$ is a tilted ample chain with the tilted numerical polynomial  $\widetilde P_t(-) = -\chi(\mathcal O(-t)[1], -) = P_t(-)$ of $\widetilde{\mathcal A}$. Note that $\mathcal O(-1)[1]$ is the only stable object in $\widetilde{\mathcal A}$ so that $\chi(\mathcal O_X, \mathcal O(-1)[1]) =0$. Let $\{L_{-t} := \mathcal O(-t)[1] \oplus \mathcal O(-t)[1]\}_{t > 0}$ be a tilted ample chain of $\widetilde{\mathcal A}$, we can construct a slope sequence $L^s_{-t} = \mathrm C(\mathcal O_p^{\oplus 2t + 1} \rightarrow L_{-t})$ by the following short exact sequence in $\mathrm{Coh}(\mathbf P^1)$:
$$
0 \longrightarrow \mathcal O(-t) \oplus \mathcal O(-t) \longrightarrow L^s_{-t} = \mathcal O \oplus \mathcal O(1) \longrightarrow \mathcal O_p^{\oplus t} \oplus \mathcal O_p^{\oplus t+1} \longrightarrow 0,
$$
such that the numerical polynomial $P^s_t(A)$ is a positive integer for all $A \neq 0 \in \widetilde{\mathcal A}$ inducing a $\Delta$-stability of degree $0$, i.e., each object in $\widetilde{\mathcal A}$ is semistable. 

One can easily check that $L_{-t}^s$ is also a weakly ample sequence with $m=0$ on $\widetilde{\mathcal A}$, but not an ample sequence. Indeed, $\widetilde{\mathcal A} = \langle\mathcal O, \mathcal O(-1)[1] \rangle$ and $\mathrm{End}(L_{-t}^s)$ is the algebra of paths of Kronecker quiver $P_2:= "\cdot \rightrightarrows \cdot"$, the functor $R\mathrm{Hom}(L^s_{-t}, - )$ gives a equivalence of abelian categories between $\widetilde{\mathcal A}$ and representation of $P_2$ since $R\mathrm{Hom}(L^s_{-t}, \mathcal O)$ and $R\mathrm{Hom}(L^s_{-t}, \mathcal O(-1)[1])$ correspond to the irreducible representations associated to vertices in $\mathrm{Rep}(P_2)$. Note that $R\mathrm{Hom}(L^s_{-t}, \mathcal O(1))$ corresponds to the projective submodule of $\mathrm{End}(L_{-t}^s)$ associated to the codomain of arrows in $P_2$, and the canonical short exact sequence in $\widetilde{\mathcal A}$
$$
0 \longrightarrow \mathcal O \oplus \mathcal O \longrightarrow \mathcal O(1) \longrightarrow \mathcal O(-1)[1]  \longrightarrow 0
$$
induces the short exact sequence associated to projective module and irreducible representation in $\mathrm{Rep}(P_2)$~\cite{MR992977}. Therefore we already prove the following property.
\begin{prop}
On $\mathrm{D^b}(\mathbf P^1)$ up to $\mathrm{Aut}(\mathrm{D^b}(\mathbf P^1))$, there exist a $\Delta$-stability of degree $1$ given by Hilbert polynomial on $\mathrm{Coh}(\mathbf P^1)$ with $m = 1$ ample chain $\{\mathcal O(-t)\}_{t \in \mathbb Z}$, and a $\Delta$-stability of degree $0$ given by slope polynomial $P^s_t(-)$ on Kronecker heart $\mathrm{Rep}(P_2)$ with $m = 0$ ample chain $\mathcal O \oplus \mathcal O(1)$.
\end{prop}
However, for a curve $X$ of positive genus such a construction of $\Delta$-stability of degree $0$ does not work because moduli spaces of stable sheaves of rank $> 1$ are not empty. Indeed, those numerical polynomial $P^s_t(-)$ of degree $0$ as above restricted to some moduli spaces of stable sheaves are trivial, thus can not imply $\Delta$-stability on $\mathrm{D^b}(X)$, but $\mathrm{D^b}(X)/\{E^{\bullet}: P^s_t(E^{\bullet}) = 0\}$.

Thus on $\mathrm{D^b}(X)$ most numerical stability conditions are equivalent to each other, that is, have the same $\Delta$-stability data. Conversely, if $\mathrm{D^b}(X)$ is derived equivalent to $\mathrm{D^b}(Y)$ for some nonsingular integral  projective scheme $Y$ over $k$, the induced $\Delta$-stability data would implies an isomorphism $X \simeq Y$.
\begin{lem}{\cite{MR2084563}}\label{cs}
Suppose $E \in \mathrm{Coh}(X)$ on a nonsingular curve $X$ with positive genus can fit into a distinguished triangle $E' \rightarrow E \rightarrow E'' \rightarrow E'[1]$ with $\mathrm{Hom}(E', E''[j]) = 0$ for $j \leq 0$. Then we have $E', E'' \in \mathrm{Coh}(X)$.
\end{lem}
\begin{prop}
Let $X$ be a nonsingular projective curve over $k$ and $Y$ be a nonsingular integral projective scheme over $k$. Then $\mathrm{D^b}(X)$ is derived equivalent to $\mathrm{D^b}(Y)$ if and only if $X$ is isomorphic to $Y$.
\end{prop}
\begin{proof}
"if" part is trivial, thus we only consider "only if" part. Given an exact equivalence $T: \mathrm{D^b}(X) \rightarrow \mathrm{D^b}(Y)$, we have $\dim_k X = \dim_k Y$ (for example, see~\cite{MR2244106}). An ample invertible sheaf $\mathcal O(H)$ induces Hilbert polynomial 
$$
P_t(E) = \chi(\mathcal O(-tH), E) = -\chi(\mathcal O_H, E)t + \chi(\mathcal O_X, E)
$$
for $E \in \mathrm{Coh}(X)$. Since $H$ is a zero subscheme of $X$, one can choose a closed point $x$ in $X$ such that $\chi(\mathcal O_H, E) = \chi\left(k(x)^{\oplus d}, E\right)$, and $k(x)$ is simple, thus up to shift we may assume $T\left(k(x)\right) \in \mathrm{Coh}(Y)$. It implies that $T(L)$ belongs to either $\mathrm{Coh}(Y)$ or $\mathrm{Coh}(Y)[-1]$ for $L \in \mathrm{Pic}(X)$ because $L$ is also simple and $\mathrm{Hom}(L, k(x)) \neq 0$.

Suppose $X$ is rational, $k(y)$ is $\bar{\mu}$-stable with respect to the induced slope polynomial 
$$
\bar P_t(k(y)) = -\chi\left(T(k(x))^{\oplus d}, k(y)\right)t + \chi(T(\mathcal O_X), k(y))
$$ 
for each closed point $y \in Y$ because of simpleness of $k(y)$. Note that $k(y) \asymp k(y')$ for all $y, y'  \in Y$, i.e. $\bar{\mu}(k(y)) = \bar{\mu}(k(y'))$, thus $T|_{\{k(x)\}}$ gives an equivalence of categories between $\{k(x)\}_{x \in X}$ and $\{k(y)\}_{y \in Y}$. Hence we have $X \simeq \mathrm{Hilb}^1(X) \simeq  \mathrm{Hilb}^1(Y) \simeq Y$ where $\mathrm{Hilb}^1(X)$ is Hilbert shceme of $X$ with Hilbert polynomial $P = 1$ by the universal property of moduli functors.

Assume the genus $g(X) > 0$, $k(y)$ is $\bar{\mu}$-stable for $y \in Y$ and $L'$ is $\bar{\mu}$-stable for $L' \in \mathrm{Pic}(Y)$ by Lemma~\ref{cs}.  Indeed, if $L'$ is not $\bar{\mu}$-semistable, then $\bar{\mu}(L') > \bar{\mu}(k(y))$ for some $y \in Y$ implies an infinite chain $\bar{\mu}(L') < \bar{\mu}(L' \otimes L_y) < \bar{\mu}\left(L' \otimes L_y^{\otimes 2}\right) < \cdots$ for $\cdots \hookrightarrow L' \otimes L_y^{\otimes 2} \hookrightarrow L' \otimes L_y \hookrightarrow L$ with ideal sheaf $L_y$ of $y$ in $T(\mathrm{Coh}(X)[l])$ for $l \in \mathbb Z$, a contradiction. Further, if $\bar{\mu}(L') = \bar{\mu}(k(y))$ then $L'$ has an infinite Jordan-H\"older filtration, a contraction. Also we have $\bar{\mu}(k(y)) = \bar{\mu}(k(y'))$ for all $y, y' \in Y$ and $\{k(y)\}_{y\in Y}$ forms a full $\bar{\mu}$-stable subcategory of $\mathrm{D^b}(Y)$. Thus the equivalence $T$ induces an isomorphism between $\mathrm{Hilb}^1(X)$ or $\mathcal M(r, d)$ and $\mathrm{Hilb}^1(Y)$ where $\mathcal M(r, d)$ is the moduli space of stable bundles on $X$ with rank $r$ and degree $d$.

If $g(X) > 1$ then $\dim \mathcal M(r, d) = r^2(g(X) - 1) + 1 > 1$ (see~\cite{MR2665168}). It turns out that $X \simeq \mathrm{Hilb}^1(X) \simeq  \mathrm{Hilb}^1(Y) \simeq Y$. If $g(X) = 1$, $\dim \mathcal M(r, d) = 1$ for $(r, d) = 1$ and $\mathcal M(r, d) \simeq X$ by Atiyah's theorem~\cite{MR131423}. Thus we have $X \simeq Y$.
\end{proof}

\begin{remark} (i) This property also implies that we only have one $\Delta$-stability data given by the slope polynomial of degree $1$ associated to each ample chain. Indeed, such a $\Delta$-stability of degree $1$ is equivalent to Bridgeland's stability on curves up to the action $\widetilde{\mathrm{GL}}^+_2(\mathbb R)$, for example, see~\cite{MR2376815, MR2709619, MR2335991}. For the finest stability conditions (t-stability) on rational curve and elliptic curve, see~\cite{MR2084563}.

(ii) If $X$ is not an elliptic curve, then this property is a special case of Bondal and Orlov theorem~\cite{MR1818984}, which said that if the (anti-)canonical bundle of $X$ is ample, then $\mathrm{D^b}(X) \simeq \mathrm{D^b}(Y) \Longleftrightarrow X \simeq Y$. For derived categories of elliptic curves or abelian varieties, a detailed overview of complete understanding can be found in the book~\cite{MR2244106}. 
\end{remark}

\section{$\Delta$-stability on normal projective surfaces}\label{sec:nv}

In the following sections $k$ is assumed to be a perfect field of arbitrary characteristic. Let $X$ be a normal integral projective scheme of $\dim X = n$ over $k$ with a very ample invertible sheaf $\mathcal O(H)$, and $\mathcal O_{H^d}$ be the structure sheaf $\mathcal O_{H_1\cap\cdots\cap H_d}$ of the closed scheme of locally complete intersection $H_1\cap\cdots\cap H_d$ for some effective divisors $H_1,\dots,H_d \in |\mathcal O(H)|$ for $1 \leq d \leq n$. Then $\{\mathcal O(-tH)\}_{t \in \mathbb Z}$ forms an ample chain and Hilbert polynomial $P_t(E)$ for $E \in \mathrm{Coh}(X)$ is $\chi(\mathcal O(-tH), E)$.
\begin{lem}\label{Hilbert}
Let $E$ be a coherent sheaf of dimension $d$, i.e., $\dim(\mathrm{Supp(E)}) = d$. Then Hilbert polynomial of $E$ is
$$
P_t(E) = \chi(\mathcal O(-tH), E) = \sum^d_{c=0}(-1)^c\chi\left(\mathcal O_{H^c}, E\right)\binom{t}{c},
$$
where $\mathcal O_{H^0} = \mathcal O_X$ and $c$ is the codimension of $H^c := H_1\cap\cdots\cap H_c$.
\end{lem}
\begin{proof}
Since $\{\mathcal O(-tH)\}_{t \in \mathbb Z}$ forms an ample chain, we have $\mathcal O_{H^c} = L_0^{(c)}$ by assumption of ample chain. Let $i: H^c \hookrightarrow H^{c'}$ for $c > c'$ and $H^c$ is a complete intersection, thus the exact functor $i_*: \mathrm{D^b}(H^c) \rightarrow \mathrm{D^b}\left(H^{c'}\right)$ admits the right adjoint functor $i^!: \mathrm{D^b}\left(H^{c'}\right) \rightarrow \mathrm{D^b}(H^c)$. Indeed, let $Li^*$ be the left adjoint functor of $i_*$, we have
$$
\mathrm{Hom}_{H^{c'}}(i_*(E^{\bullet}), F^{\bullet}) \simeq \mathrm{Hom}_{H^c}\left(E^{\bullet}, i^!(F^{\bullet})\right) \simeq \mathrm{Hom}_{H^c}(E^{\bullet}, Li^*(F^{\bullet})\otimes \ox_{c/c'}[c'-c]),
$$
where $E^{\bullet} \in \mathrm{D^b}(H^c), F^{\bullet} \in \mathrm{D^b}\left(H^{c'}\right)$ and $\ox_{c/c'}$ is a locally free sheaf of rank one (see~\cite{MR0222093}). Thus by lemma~\ref{ef} we are done.
\end{proof}
\begin{remark}
If the closed subscheme $H^c$ is not a locally complete intersection, then the formula does not hold in general. However, for each coherent sheaf $F$ on $X$ there exists an $F$-regular sequence $H_1,\dots,H_d \in |\mathcal O(H)|$ such that $Li^*(F)$ is quasi-isomorphic to a coherent sheaf on $H^d$ (e.g., see~\cite{MR2665168}).
\end{remark}
\begin{definition}\label{rd}
If $E$ is a coherent sheaf of dimension $n$, then the rank of $E$ is 
$$
\mathrm{rk}(E) := \frac{\chi(\mathcal O_{H^n}, E)}{\chi(\mathcal O_{H^n}, \mathcal O_X)}
$$
with $\deg_H X := (-1)^n\chi(\mathcal O_{H^n}, \mathcal O_X)$, and the degree of $E$ is defined by
$$
\deg_H(E) := (-1)^{n-1}\left(\chi(\mathcal O_{H^{n-1}}, E) - \mathrm{rk}(E) \chi(\mathcal O_{H^{n-1}}, \mathcal O_X)\right).
$$
\end{definition}
\begin{remark}
Since $X$ is integral, there exists an open dense subset $U \subset X$ such that $E|_U$ is a locally free sheaf on $U$ and $\mathrm{rk}(E)$ is the rank of the vector bundle $E|_U$. If $X$ is nonsingualr, the Hirzebruch-Riemann-Roch formula shows $\deg_H(E) = c_1(E).H^{n-1}$ and $\deg_H(E) = \deg_H(\det(E))$. Indeed, this definition has all the properties one expects of the degree.
\end{remark}

Hilbert polynomial is a slope numerical polynomial because of ampleness of $H$ or theorem~\ref{asp}, thus the coefficients of Hilbert  polynomial give an exhaustive positive system. The vector slope induces a $\Delta$-stability of degree $n$ or Gieseker stability on $\mathrm{D^b}(X)$ and the scalar slope induces a coarse $\Delta$-stability of degree $n$ or Mumford-Takemoto $\mu$-stability. In the additive category of coherent sheaves of dimension $n$ on $X$, there are two natural scalar slopes giving the same $\Delta$-stability data defined as following.
\begin{definition}
Let $E$ is a coherent sheaf of dimension $n$, then its $\mu$-slope is defined by
$$
\mu(E) := \frac{\deg(E)}{\mathrm{rk}(E)}
$$
where $\deg(E) = \deg_H(E)$, and its $\hat{\mu}$-slope by
$$
\hat{\mu}(E) := -\frac{\chi(\mathcal O_{H^{n-1}}, E)}{\chi(\mathcal O_{H^n}, E) } = \frac{\mu(E)}{\deg X} - \frac{\chi(\mathcal O_{H^{n-1}}, \mathcal O_X)}{\deg X}
$$
where $\deg X = \deg_H X$.
\end{definition}
\begin{remark}
One has the following chain of stabilities: $E$ is $\mu$-stable $\Rightarrow$ $E$ is stable $\Rightarrow$ $E$ is semistable $\Rightarrow$ $E$ is $\mu$-semistable.
\end{remark}
In algebraic geometry based on the study of moduli spaces, intersection theory and relevant subjects, the most interesting slope of the object $E$ in the bounded derived category of coherent sheaves is $\mu(E)$. However, in category theory on the study of numerical stability condition the slope $\hat{\mu}(E)$ would be more functorial under pushforward.

\subsection{Finite cover of projective spaces}\label{cover}

In this subsection we will study the coarse $\Dx$-stability of degree $n$ or $\hat{\mu}$-stability between normal integral projective scheme $X$ of dimension $n$ over $k$ and the projective space $\mathbf P^n := \mathbf P^n_k$. This in turn allows us to find the relation of $\hat{\mu}$-stable torsion free coherent sheaves between $X$ and $\mathbf P^n$ as has been seen in Prop.~\ref{con}, and paves the way for boundedness conditions in Prop.~\ref{bound} and thus Bridgeland stability on normal surfaces.

Since $X$ is a projective scheme of dimension $n$ over a perfect field $k$, there exists a finite surjective separable morphism from $X$ to $\mathbf P^n$ by Noether normalization lemma (e.g., see~\cite{MR1322960, MR1011461}). In other words, $\mathbf P^n$ is a weakly terminal object in the category of projective schemes of dimension $n$ over $k$. When $k$ is infinite, such a finite covering map can be constructed by taking a general projection from a projective space containing $X$ to a $n$-dimensional linear subspace.

Let $f: X \rightarrow \mathbf P^n$ be a finite separable morphism of degree $d$ of normal integral projective scheme of dimension $n$ over $k$  induced by a very ample invertible sheaf $\mathcal O_X(H)$. Let $\mathcal O_{\mathbf P^n}(1)$ be a vary ample invertible sheaf on $\mathbf P^n$, then $\mathcal O_X(H) = f^*(\mathcal O_{\mathbf P^n}(1))$ and $d = \deg_H X$. Since $f$ is proper and affine, $f_*: \mathrm{Coh}(X) \rightarrow \mathrm{Coh}(\mathbf P^n)$ is an exact functor with a left adjoint functor $f^* \dashv f_*$ and thus an exact functor from $\mathrm{D}(X)$ to $\mathrm{D}(\mathbf P^n)$ with a natural functorial isomorphism
$$
f_*R\mathcal Hom_{\mathrm{D}(X)}(Lf^*F^{\bullet}, E^{\bullet}) \simeq R\mathcal Hom_{\mathrm{D}(\mathbf P^n)}(F^{\bullet}, f_*E^{\bullet})
$$
for $F^{\bullet} \in \mathrm{D^-}(\mathbf P^n)$ and $E^{\bullet} \in \mathrm{D^+}(X)$. Here, the left derived functor $Lf^*$ takes $\mathrm{D}(\mathbf P^n)$ into $\mathrm{D}(X)$. Moreover, by Grothendieck-Verdier duality there is a functor $f^!: \mathrm{D}(\mathbf P^n) \rightarrow \mathrm D(X)$ with a natural functorial isomorphism
$$
f_*R\mathcal Hom_{\mathrm{D}(X)}\left(E^{\bullet}, f^!F^{\bullet}\right) \simeq R\mathcal Hom_{\mathrm{D}(\mathbf P^n)}(f
_*E^{\bullet}, F^{\bullet})
$$
for $E^{\bullet} \in \mathrm{D^-}(X)$ and $F^{\bullet} \in \mathrm{D^+}(\mathbf P^n)$. For further details we refer to~\cite{MR0222093}.
\begin{remark}
If $X$ is nonsingular, the functors $Lf^*, f^!: \mathrm{D^b}(\mathbf P^n) \rightarrow \mathrm{D^b}(X)$ are left, right adjoint to $f_*: \mathrm{D^b}(X) \rightarrow \mathrm{D^b}(\mathbf P^n)$, respectively. That is, $Lf^* \dashv f_* \dashv f^!$.
\end{remark}

Because of exactness of $f_*$ and $\mathcal O_X(H) = f^*(\mathcal O_{\mathbf P^n}(1))$, for $E^{\bullet} \in \mathrm{D^b}(X), j \in \mathbb Z$, we have 
$$
\mathrm{Hom}_{\mathrm{D}(X)}(\mathcal O_X(H), E^{\bullet}[j]) \simeq \mathrm{Hom}_{\mathrm{D}(\mathbf P^n)}(\mathcal O_{\mathbf P^n}(1), f_*E^{\bullet}[j])
$$
and thus $P_t(E^{\bullet}) = \chi(\mathcal O_X(-tH), E^{\bullet}) = \chi(\mathcal O_{\mathbf P^n}(-t), f_*E^{\bullet}) = P_t(f_*E^{\bullet})$. Moreover, $f$ is dominant and $X$ is integral, thus $f_*$ preserve purity, that is, if $E \in \mathrm{Coh}(X)$ is torsion free then $f_*(E) \in \mathrm{Coh}(\mathbf P^n)$ is also torsion free (see~\cite{MR217083}). On the other hand, since $X$ is regular in codimension $1$ (R1), the singular locus $Z:= \mathrm{Sing}(X)$ is a closed subset of $X$ of codimension $\geq 2$. Let $U := X\setminus Z$ be the nonsingular locus of $X$. Then the restricted morphism $f|_U$ is flat. Indeed, if $X$ is Cohen-Macaulay, then $f$ is flat (see~\cite{MR1011461}). Thus for each normal projective scheme of dimension $2$ this finite surjective morphism to $\mathbf P^2$ is flat. Note that a normal, locally Noetherian scheme of dimension $2$ is Cohen-Macaulay (e.g., see~\cite{MR1917232}). In general, $f$ is flat in codimension $1$, i.e., $f^*$ is exact modulo coherent sheaves of dimension $\leq n - 2$.

\begin{lem}\label{order}
Let $F$ be a torsion free coherent sheaf on $\mathbf P^n$. Then $\mathrm{rk}(f^*F) = \mathrm{rk}(F)$ and $\mu(f^*F) = d\cdot \mu(F)$. Let $E$ be a torsion free coherent sheaf on $X$. Then $\mathrm{rk}(f_*E) = d \cdot \mathrm{rk}(E)$ and $\hat{\mu}(f_*E) = \hat{\mu}(E)$.
\end{lem}
\begin{proof}
We have the commutative diagram below
\begin{center}
\begin{tikzcd}[row sep=large, column sep=large]
H^n \arrow[hookrightarrow]{r}\arrow[d, "f|_{H^n}"] &H^{n-1} \arrow[d, "\bar f:= f|_{H^{n-1}}"] \arrow[hookrightarrow]{r} &X \arrow[d, "f"] \\
\{p\} \arrow[hookrightarrow]{r} &\mathbf P^1 \arrow[hookrightarrow]{r} &\mathbf P^n.
\end{tikzcd}
\end{center}
$E \in \mathrm{Coh}(X)$ and $f_*$ is exact, then
$$
\mathrm{rk}(f_*E) = \frac{\chi(k(p), f_*E)}{\chi(k(p), \mathcal O_{\mathbf P^n})} = (-1)^n\chi(\mathcal O_{H^n}, \mathcal O_X)\frac{\chi(\mathcal O_{H^n}, E)}{\chi(\mathcal O_{H^n}, \mathcal O_X)} = d\cdot \mathrm{rk}(E),
$$
and
$$
\hat{\mu}(f_*E) = -\frac{\chi(\mathcal O_{\mathbf P^1}, f_*E)}{\chi(k(p), f_*E) } = -\frac{\chi(\mathcal O_{H^{n-1}}, E)}{\chi(\mathcal O_{H^n}, E)} = \hat{\mu}(E).
$$
$F$ is a torsion free coherent sheaf on $\mathbf P^n$, one has 
$$
\chi_{H^n}(\mathcal O_{H^n}, (f^*F)|_{H^n}) = \chi_{H^n}(\mathcal O_{H^n}, f|_{H^n}^*(F|_p)) = d\cdot \chi_p(k(p), F|_p),
$$
thus it leads to
$$
\frac{\chi(\mathcal O_{H^n}, f^*F)}{\chi(\mathcal O_{H^n}, \mathcal O_X)} = \frac{\chi_{H^n}(\mathcal O_{H^n}, (f^*F)|_{H^n})}{\chi(\mathcal O_{H^n}, \mathcal O_X)} = \frac{d\cdot \chi_p(k(p), F|_p)}{d\cdot \chi(k(p), \mathcal O_{\mathbf P^n})} = \frac{ \chi(k(p), F)}{\chi(k(p), \mathcal O_{\mathbf P^n})}
$$
and we have $\mathrm{rk}(f^*F)  = \mathrm{rk}(F)$. On the other hand, we may assume $f^*F$ is $i^*$-acyclic and $\ox_{n-1}$ is the relative dualizing locally free sheaf associated to $i : H^{n-1} \hookrightarrow X$, then
\begin{align*}
\chi_{H^{n-1}}(\mathcal O_{H^{n-1}}, f^*F|_{H^{n-1}}\otimes \ox_{n-1}) &= \chi_{H^{n-1}}\left(\mathcal O_{H^{n-1}}, \bar f^*(F|_{\mathbf P^1})\otimes \ox_{n-1}\right)\\ &= \chi_{\mathbf P^1}\left(\mathcal O_{\mathbf P^1}, \bar f_*\left(\bar f^*(F|_{\mathbf P^1})\otimes \ox_{n-1}\right)\right) \\ &= \chi_{\mathbf P^1}\left(\mathcal O_{\mathbf P^1},  F|_{\mathbf P^1} \otimes \bar f_*\ox_{n-1}\right).
\end{align*}
In $\mathbf P^1$, $\chi_{\mathbf P^1}(\mathcal O_{\mathbf P^1}, V \otimes W) = \mathrm{rk}(W)\deg(V) + \mathrm{rk}(V)\chi_{\mathbf P^1}(\mathcal O_{\mathbf P^1}, W)$ for locally free sheaves $V, W$. It turns out that 
\begin{align*}
\chi_{H^{n-1}}\left(\mathcal O_{H^{n-1}}, f^*F|_{H^{n-1}} \otimes \ox_{n-1}\right) &= d\cdot\deg_{\mathbf P^1}(F|_{\mathbf P^1}) + \mathrm{rk}(F)\chi_{\mathbf P^1}\left(\mathcal O_{\mathbf P^1}, \bar f_*\ox_{n-1}\right)\\ &= d\cdot\deg(F) + \mathrm{rk}(F)\chi_{H^{n-1}}(\mathcal O_{H^{n-1}}, \ox_{n-1})
\end{align*}
and 
\begin{align*}
\deg(f^*F) &= (-1)^{n-1}\left(\chi(\mathcal O_{H^{n-1}}, f^*F) - \mathrm{rk}(f^*F) \chi(\mathcal O_{H^{n-1}}, \mathcal O_X)\right)\\ &= \chi_{H^{n-1}}\left(\mathcal O_{H^{n-1}}, f^*F|_{H^{n-1}} \otimes \ox_{n-1}\right) - \mathrm{rk}(F)\chi_{H^{n-1}}(\mathcal O_{H^{n-1}}, \ox_{n-1}) \\ &= d\cdot \deg(F)
\end{align*}
which implies 
$$
\mu(f^*F) = \frac{\deg(f^*F)}{\mathrm{rk}(f^*F)} = d\cdot \frac{\deg(F)}{\mathrm{rk}(F)} = d\cdot \mu(F).
$$
\end{proof}

\begin{lem}\label{Galois}
Let $F$ be a torsion free coherent sheaf on $\mathbf P^n$. Then $F$ is $\mu$-semistable if and only if the torsion free part of $f^*F$ is $\mu$-semistable.
\end{lem}
\begin{proof}
$F$ is torsion free, and $f^*F$ has no torsion in codimension $1$. Let $U = X\setminus Z$ be the nonsingular locus of $X$ with the singular locus $Z$ of codimension $\geq 2$. We may assume $f^*F|_U$ is torsion free on $U$ and $\bar f := f|_U$ is flat. Let $j: U \hookrightarrow X$ be the open immersion and one has the short exact sequence
$$
0 \longrightarrow \mathcal H_Z^0(F) \longrightarrow f^*F \longrightarrow j_*(f^*F|_U) \longrightarrow \mathcal H_Z^1(F) \longrightarrow 0,
$$
where $\mathcal H_Z^0(F), \mathcal H_Z^1(F) \in \mathrm{Coh}(X)$ are local cohomology sheaves (see~\cite{MR0224620}) and $j_*(f^*F|_U) \in \mathrm{Coh}(X)$ is torsion free. This gives $\mu(f^*F) = \mu(j_*(f^*F|_U)) = \mu\left(j_*\left(\bar f^*\left(F|_{f(U)}\right)\right)\right)$ since $\mathrm{Supp}(\mathcal H_Z^i(F)) \subset Z$ and the torsion free part of $f^*F$ is $\mu$-semistable if and only if  $j_*\left(\bar f^*\left(F|_{f(U)}\right)\right)$ is $\mu$-semistable. 

If $F' \subset F$ is a proper subsheaf with $\mu(F') > \mu(F)$, then 
$
\mu\left(j_*\left(\bar f^*\left(F'|_{f(U)}\right)\right)\right) = \mu(f^*F') > \mu(f^*F) = \mu\left(j_*\left(\bar f^*\left(F|_{f(U)}\right)\right)\right)$ 
and $j_*\left(\bar f^*\left(F'|_{f(U)}\right)\right)$ is a proper subsheaf of $j_*\left(\bar f^*\left(F|_{f(U)}\right)\right)$. Conversely, let $K$ be a splitting field of the function field $K(X)$ over $K(\mathbf P^n)$. Since the base field $k$ is perfect, $K$ is separable over $K(\mathbf P^n)$. Let $Y$ be the normalization of $X$ in $K$, we have finite separable morphisms $Y \rightarrow X \rightarrow \mathbf P^n$. Hence we may assume $K(X)$ is a Galois extension of $K(\mathbf P^n)$ with Galois group $G$. 

If $F$ is $\mu$-semistable on $\mathbf P^n$ so that $j_*\left(\bar f^*\left(F|_{f(U)}\right)\right)$ is not $\mu$-semistable on $X$. Let $E'$ be its maximal destabilizing subsheaf and $E'$ is unique. In other words, $E'$ is invariant under the action of $G$. By Galois decent theory (for example, see~\cite{MR1600388}), there exists a proper subsheaf $F'$ of $F$ such that $f^*F'|_{X\setminus f^{-1}(S(F))}$ is isomorphic to $E'|_{X\setminus f^{-1}(S(F))}$ where $S(F)$ is the singular support of $F$ of codimension $\geq 2$. Thus $\mu(F') > \mu(F)$, a contradiction.
\end{proof}
\begin{remark}
In this argument one needs the base field is perfect in order to imply Galois decent theory. Indeed, if $k$ is imperfect, $K(X)$ may not be a Galois extension. In this case, this lemma holds if we consider a \textit{geometrically} normal projective scheme $X$ instead of a normal projective scheme. Note that if $X$ is not normal, $\mathcal H_Z^1(F)$ is quasi-coherent in general, thus $j_*(f^*F|_U) $ is not coherent.
\end{remark}

The scalar slopes defined in these way realize the same coarse $\Dx$-stability data associated to Hilbert polynomials of coherent sheaves on $X$ as appearing in Def.~\ref{ps}. Whenever the normal integral scheme $X$ has a finite separable morphism $f$ to the projective space $\mathbf P^n$, the fibers of ample subvarieties induced by the ample chain on $\mathbf P^n$ develop the corresponding ample chain structure extending over the entire scheme $X$. The proposed construction can be viewed as a relative dimension $=0$ version of relative schemes between a $n$-dimensional integral scheme and $n$-dimensional projective space in the category of integral schemes of dimension $n$. In the context of finite covers between normal projective varieties over an algebraic closed field of characteristic $0$ a detailed comparison of scalar slopes can been seen in the book~\cite{MR2665168}.

While the ample chain on $X$ complies with the fibration of an ample chain on $\mathbf P^N$, we should stress that the relation of the order on semistable sheaves on $X$ and $\mathbf P^n$ encountered in this work arises from scalar slopes of coherent sheaves on $X$ and $\mathbf P^n$, where the scalar slopes of coherent sheaves reside in the coefficients of Hilbert polynomial. Thus, compared to Prop.~\ref{con} the constraints of scalar slopes of coherent sheaves emerges from vanishing of morphisms between semistable sheaves.

\begin{prop}\label{max}
Let $E$ be a $\mu$-semistable sheaf on a normal integral projective scheme $X$ of dimension $n$ over a perfect field and $f: X \rightarrow \mathbf P^n$ be a finite separable morphism. Then $f_*E$ is a torsion free sheaf on $\mathbf P^n$ and $$\mu_{\max}(f_*E) \leq \frac{\mu(E)}{d},$$ where $d = \deg_H X$ with respect to the vary ample invertible sheaf $\mathcal O_X(H) = f^*\mathcal O_{\mathbf P^n}(1)$.
\end{prop}
\begin{proof}
$f$ is finite, thus $f_*E$ is torsion free. Let $G$ be the maximal destabilizing subsheaf of $f_*E$ on $\mathbf P^n$. By the adjunction of $f^* \dashv f_*$ one has $0 \neq \mathrm{Hom}_{\mathbf P^n}(G, f_*E) \simeq \mathrm{Hom}_X(f^*G, E)$ and the previous lemma~\ref{Galois} implies the torsion free part $G'$ of $f^*G$ is $\mu$-semistable on $X$ with $\mu(G') = \mu(f^*G)$. Therefore, $0 \neq \mathrm{Hom}_X(f^*G, E) \simeq \mathrm{Hom}_X(G', E)$ and $G', E$ are $\mu$-semistable, one must have $\mu(G') \leq \mu(E)$. On the other hand, $\mu(G') = \mu(f^*G) = d\cdot \mu(G) = d\cdot \mu_{\max}(f_*E)$ by lemma~\ref{order}. It turns out that $\mu_{\max}(f_*E) \leq \frac{\mu(E)}{d}$.
\end{proof}

Let $U$ be the nonsingular locus of $X$ and we have the following diagram
\begin{center}
\begin{tikzcd}[row sep=large, column sep=large]
U \arrow[d, "\bar f:= f|_U"] \arrow[hookrightarrow]{r}{j} &X \arrow[d, "f"] \\
U_f := f(U) \arrow[hookrightarrow]{r} &\mathbf P^n.
\end{tikzcd}
\end{center}
Since $X$ is noetherian scheme of finite type over a field, there exists a dualizing complex in $\mathrm{D}^+(X)$ and the restriction of the dualizing complex to $U$ is a dualizing complex of $U$. Indeed. for any residually stable morphism $f$, the pullback functor $f^*$ takes a residue complex to a residue complex. For the projective scheme $X$ the dualizing complex $\ox_X^{\bullet}$ induced by the dualizing sheaf of the ambient projective space of X defines the dualizing sheaf $\ox_X: = H^{-n}(\ox_X^{\bullet})$ on $X$. Since the restriction of open immersion to $U$ is an exact functor, one also has $\ox_X|_U \simeq \ox_U$. Here, $\ox_U$ is the dualizing sheaf on $U$ and is isomorphic to the canonical sheaf of $U$. Recall that $X$ is normal, the $R_1$ and $S_2$ conditions combined with reflexive sheaf $\ox_X$ imply that that $\ox_X \simeq j_*\ox_U$. For more details one refers to~\cite{MR0222093, MR1804902, MR597077}.
\begin{prop}\label{min}
Let $E$ be a $\mu$-semistable sheaf on $X$ as Prop.~\ref{max}. Then $$\mu_{\min}(f_*E) \geq \frac{\mu(E)} d - \frac{\mu(\ox_X)} d -(n+1),$$ where $\ox_X$ is the canonical sheaf of $X$.
\end{prop}
\begin{proof}
Let $Q$ be the minimal destabilizing quotient sheaf of $f_*E$ on $\mathbf P^n$ such that $0 \neq \mathrm{Hom}_{\mathbf P^n}(f_*E, Q)$. Thus $0 \neq \mathrm{Hom}_{U_f}\left(f_*E|_{U_f}, Q|_{U_f}\right) \simeq \mathrm{Hom}_{U_f}\left(\bar f_*(E|_U), Q|_{U_f}\right)$ by flat base change. Also by the adjunction of $\bar f_* \dashv \bar f^!$ we have natural isomorphisms 
\begin{align*}
\mathrm{Hom}_{U_f}\left(\bar f_*(E|_U), Q|_{U_f}\right) &\simeq \mathrm{Hom}_U\left(E|_U, \bar f^!\left(Q|_{U_f}\right)\right) \\ &\simeq \mathrm{Hom}_U\left(E|_U \otimes \ox_U^{-1}(-n-1), \bar f^*\left(Q|_{U_f}\right)\right) \\ &\simeq \mathrm{Hom}_X\left(E\otimes j_*\left(\ox_U^{-1}\right)(-n-1), j_*\left(f^*Q|_U\right)\right),
\end{align*}
where $\bar f^!\left(Q|_{U_f}\right) \simeq \bar f^*\left(Q|_{U_f}\right) \otimes \ox_U(n+1) \simeq f^*Q|_U \otimes \ox_X|_U(n+1)$. Thus the lemma~\ref{Galois} implies $j_*(f^*Q|_U)$ is $\mu$-semistable and in turn the torsion free part of $E\otimes \ox_X^{-1}(-n-1)$ is also $\mu$-semistable with the reflexive sheaf $\ox_X^{-1} := j_*\left(\ox_U^{-1}\right)$ of rank one. Indeed, the torsion free part of $E' \otimes \ox_X^{-1}$ is $\mu$-semistable for any $\mu$-semistable sheaf $E'$ on $X$. If $G \subset j_*\left(E'|_U\otimes \ox_U^{-1}\right)$ is the maximal destabilizing subsheaf, then $G|_U\otimes \ox_U \subset E'|_U$ is a subsheaf since $\ox_U^{-1} = \ox_X^{-1}|_U$ is an invertible sheaf on $U$. One has $j_*(G|_U \otimes \ox_U) \subset j_*(E'|_U)$ with $\mu(j_*(G|_U \otimes \ox_U)) = \mu(j_*G|_U) + \mu(\ox_X) > \mu(E') = \mu(j_*(E'|_U))$, a contradiction. 

Here, we combine with the fact that $\mu(V \otimes W) = \mu(V) + \mu(W)$ for torsion free coherent sheaves $V, W$ on $X$. Since $X$ is normal projective scheme over a perfect field, we may assume $H^{n-1}$ is a nonsingular curve and $V|_{H^{n-1}}, W|_{H^{n-1}}$ are locally free. On a nonsingular curve $C$, one has $\deg(V'\otimes W') = \mathrm{rk}(W')\deg(V') + \mathrm{rk}(V')\deg(W')$ for locally free coherent sheaves $V', W'$ on $C$ and by elementary arithmetic we obtain $\mu(V \otimes W) = \mu(V) + \mu(W)$.

Therefore, the torsion free parts of $E \otimes \ox_X^{-1}(-n-1)$ and $f^*Q$ are $\mu$-semistable with nontrivial $\mathrm{Hom}_X\left(E \otimes \ox_X^{-1}(-n-1), f^*Q \right)$. The vanishing property between $\mu$-semistable sheaves implies 
$$
\mu_{\min}(f_*E) = \mu(Q) = \frac{\mu(f^*Q)} d  \geq \frac{\mu(E) } d - \frac{\mu(\ox_X)} d - (n+1).
$$
Note that $\mu(\mathcal O_X(H)) = \deg(\mathcal O_X(H)) = d$.
\end{proof}
\begin{remark}
If $\dim X = 2$, then $X$ is Cohen-Macaulay such that the finite map $f: X \rightarrow \mathbf P^2$ is flat and thus the dualizing complex $\ox_X^{\bullet}$ is quasi-isomorphic to the canonical sheaf $\ox_X$ shifted by $2$. Indeed, for any Cohen-Macaulay scheme of finite type of dimension $n$ over a field, we have $\ox_X^{\bullet} \simeq \ox_X[n]$.
\end{remark}
\begin{cor}
Let $\ox_X$ be the canonical sheaf of a normal integral projective scheme $X$ of dimension $n$ over a perfect field with a very ample invertible sheaf $\mathcal O_X(H)$ and $\deg_H X = d$. Then $$\deg_H(\ox_X) \geq -d\cdot (n+1). $$ 
\end{cor}
\begin{proof}
$\ox_X$ is a reflexive sheaf of rank one, thus is stable. By Prop.~\ref{max} and Prop~\ref{min}, we have $$\deg_H(\ox_X) \geq d \cdot \mu_{\max}(f_*\ox_X) \geq d\cdot \mu_{\min}(f_*\ox_X) \geq -d\cdot (n+1).$$
\end{proof}

This corollary implies that the degree of the canonical sheaf of $X$ of degree $d$ with respect to a very ample invertible sheaf  has a lower bound. In other words, the intersection number $\ox_X.H^{n-1}$ can not be less than $-d\cdot (\dim X +1)$.

\subsection{$\Delta$-stability of degree $1$ and Bridgeland stability on normal surfaces}\label{surface}

Let us now focus on a normal projective surface $X$ over a perfect field $k$, i.e., a normal integral projective scheme of dimension $2$ over $k$. According to Prop.~\ref{hnf} this amounts to constructing a slope sequence $\{L_i^s\}$ with the slope polynomial $P_{t = -i}(-) = \chi(L_i^s, -)$ on the noetherian heart $\tilde A$ of $\mathrm{D^b}(X)$ with a bounded t-structure. This can be achieved with the proposal by Prop.~\ref{bound}, generalizing the construction of $\Dx$-stability of degree zero on rational curves outlined in Sec.~\ref{zero}. In a particular example the possibility to realize $\Dx$-stability of degree $n-1$ for normal integral projective schemes of dimension $n > 2$ appeared in Sec.~\ref{higher}.

In the projective space $\mathbf P^2$ with the ample invertible sheaf $\mathcal O_{\mathbf P^2}(1)$, for any $\mu$-semistable torsion free sheaf $F$ with $s -1 \leq \mu(F) < s$ for some $s \in \mathbb Z$, we have 
$$
0 = \mathrm{Hom}(\mathcal O_{\mathbf P^2}(s), F) = \mathrm{Hom}(F, \mathcal O_{\mathbf P^2}(s-3)) \simeq \mathrm{Hom}(\mathcal O_{\mathbf P^2}(s), F[2])^*,
$$
where the last isomorphism is given by Serre duality with the dualizing sheaf $\ox_{\mathbf P^2} = \mathcal O_{\mathbf P^2}(-3)$. It turns out that $\chi(\mathcal O_{\mathbf P^2}(s), F) \leq - \dim_k\mathrm{Hom}(\mathcal O_{\mathbf P^2}(s), F[1]) \leq 0$. Therefore, by Lemma~\ref{Hilbert} the Euler characteristic of $F$ has a polynomial upper bound as
\begin{align*}
\chi(\mathcal O_{\mathbf P^2}, F) &\leq -\chi(k(p), F)\binom{-s}{2} + \chi(\mathcal O_{\mathbf P^1}, F)(-s) \\ &=-\chi(k(p), F)\left(\binom{-s}{2} + (-s)\hat{\mu}(F) \right) \\ &= \mathrm{rk}(F)\cdot (-s)\left(\frac s 2  - \frac 3 2 - \mu(F)\right)
\end{align*}
with $\hat{\mu}(F) = \mu(F) - \chi(\mathcal O_{\mathbf P^1}, \mathcal O_{\mathbf P^2}) = \mu(F) + 2$. Here, $k(p)$ is the skyscraper sheaf at a closed point $p$ and $\mathcal O_{\mathbf P^1}$ is the structure sheaf of a hyperplane in $\mathbf P^2$. Meanwhile, since $\mu(F) + 1 \geq s > \mu(F)$, by elementary calculus for each $\mu(F)$, one has $$f(s) := (-s)\left(\frac s 2  - \frac 3 2 - \mu(F)\right) \leq (\mu(F) + 1)\left(\frac{\mu(F)} 2 + 1\right) = \binom{\mu(F) + 2}{2}$$ at $s = \mu(F) + 1$. Indeed, the critical point of the polynomial $f(s)$ is at $s = \mu(F) + \frac 3 2$, thus $\sup\{f(s) : \mu(F) + 1\geq s >\mu(F)\} = \sup\{f(s) : \mu(F) + 3 > s \geq \mu(F) + 2\} = \binom{\hat{\mu}(F)}{2}$. Therefore, we prove the following lemma.
\begin{lem}\label{p2}
Let $F$ be a $\mu$-semistable torsion free sheaf on $\mathbf P^2$. Then the Euler characteristic of $F$ has a polynomial upper bound 
$$
\chi(\mathcal O_{\mathbf P^2}, F) \leq \chi(k(p), F) \cdot \binom{\hat{\mu}(F)}{2}
$$
with $\mathrm{rk}(F) = \chi(k(p), F)$ and $\hat{\mu}(F) = -\frac{\chi(\mathcal O_{\mathbf P^1}, F)}{\chi(k(p), F)}$.
\end{lem}

For the normal projective surface $X$ we pick again a finite separable morphism $f$ from $X$ to a projective space $\mathbf P^2$ with the ample invertible sheaf $\mathcal O_X(H) := f^*\mathcal O_{\mathbf P^2}(1)$. However, instead of choosing a HN filtration $f_*E$ in $\mathbf P^2$ for torsion free sheaf $E \in \mathrm{Coh}(X)$, we assume that the Euler characteristic $\chi(\mathcal O_X, E)$ factors into a sum
$$
\chi(\mathcal O_X, E) =\sum_{i=0}^m \chi(\mathcal O_X, E_i) =\sum_{i=0}^m \chi(\mathcal O_{\mathbf P^2}, f_*E_i)
$$ 
such that $E_i$ are $\mu$-semistable torsion free factors of HN filtration of $E$ with $\hat{\mu}(E_i) > \hat{\mu}(E_{i+1})$ for $i = 0, \dots, m-1$. As a consequence for each $E_i$ the Euler characteristic of $f_*E_i$ decomposes into
$$
\chi(\mathcal O_{\mathbf P^2}, f_*E_i) = \sum_{j=0}^{m_i} \chi\left(\mathcal O_{\mathbf P^2}, F_j^i\right)
$$
where $F_j^i$ are $\mu$-semistable torsion free sheaves of HN filtration of $f_*E_i$ for $j = 0, \dots, m_i$. Following Prop.~\ref{max} and Prop.~\ref{min}, we construct the upper and lower bound of $\mu_{\max}(f_*E_i), \mu_{\min}(f_*E_i)$, respectively, associated to HN filtration of $f_*E_i$ for each $\mu$-semistable torsion free sheaf $E_i$ by the slope $\mu(E_i)$, $\mu(\ox_X)$ and $d := \deg_H X$ according to
$$
\frac{\mu(E_i)}{d} \geq \mu_{\max}(f_*E_i) \geq \mu_{\min}(f_*E_i) \geq \frac{\mu(E_i)} d - \frac{\mu(\ox_X)} d - 3.
$$
Since the rank $\mathrm{rk}(E_i)$ and the slope $\hat{\mu}(E_i)$ are constants for each $i$, the rank $\mathrm{rk}(f_*E_i)$ and the slope $\hat{\mu}(f_*E_i)$ are constants as well. As Lemma~\ref{order}, the $\mu$-semistable $E_i$ is torsion free with $\mathrm{rk}(f_*E_i) = d\cdot \mathrm{rk}(E_i)$ and $\hat{\mu}(f_*E_i) = \hat{\mu}(E_i)$ for each $i$. By the decomposition of Euler characteristic of $f_*E_i$ on the bounded interval of the slope $\left[\frac{\mu(E_i)} d -\frac{\mu(\ox_X)} d - 3, \frac{\mu(E_i)}{d}\right]$ and Lemma~\ref{p2} the resulting upper bound of Euler characteristic of $f_*E_i$ is then given by
$$
\chi(\mathcal O_{\mathbf P^2}, f_*E_i) \leq \chi(\mathcal O_{H^2}, E_i) \hat P_1(\mu(E_i), \ox_X, d),
$$
where $\mathcal O_{H^2}$ is the structure sheaf of the zero scheme $H^2$ in $X$ and
$$
\hat P_1(\mu(E_i), \ox_X, d) := \sup\left\{\binom{\mu\left(F^i_j\right) + 2}{2} : \mu\left(F^i_j\right) \in \left[\frac{\mu(E_i)} d -\frac{\mu(\ox_X)} d - 3, \frac{\mu(E_i)}{d}\right], j=0, \dots, m_i\right\}.
$$
Indeed, $\hat P_1(\mu(E_i), \ox_X, d)$ is either $\binom{\frac{\mu(E_i)} d + 2}{2}$ or $\binom{\frac{\mu(E_i)} d -\frac{\mu(\ox_X)} d - 1}{2}$. Furthermore, the upper bound of Euler characteristic of $E$ for any torsion free coherent sheaf $E$ becomes
$$
\chi(\mathcal O_X, E) \leq \chi(\mathcal O_{H^2}, E)\hat P_2(\mu_{\max}(E), \mu_{\min}(E), \ox_X, d)
$$
in terms of the maximal and minimal slope $\mu_{\max}(E), \mu_{\min}(E)$, respectively, of HN filtration of $E$ and the degree of $\ox_X$ of the normal projective surface $X$ with respect to $\mathcal O_X(H)$. As before, the universal polynomial $\hat P_2$ is readily computed by the formula 
$$
\hat P_2(\mu_{\max}(E), \mu_{\min}(E), \ox_X, d) := \sup\left\{\hat P_1(\mu(E_i), \ox_X, d) : \mu(E_i) \in \left[\mu_{\min}(E), \mu_{\max}E \right], i=0, \dots, m\right\}.
$$
Thus we prove the following boundedness condition.
\begin{theorem}\label{ebound}
Let $E$ be a torsion free coherent sheaf on a normal projective surface $X$ over $k$ with a fixed very ample invertible sheaf $\mathcal O_X(H)$. Then there exists a polynomial $\hat P(\mu_{\max}(E), \mu_{\min}(E), \ox_X, d = \deg_H X)$ depending on the maximal, minimal slope of HN filtration of $E$, the degree of the canonical sheaf $\ox_X$ of $X$ and the degree of $X$ with respect to $\mathcal O_X(H)$ such that
$$
\chi(\mathcal O_X, E) \leq \chi(\mathcal O_{H^2}, E)\hat P(\mu_{\max}(E), \mu_{\min}(E), \ox_X, d)
$$
with $\chi(\mathcal O_{H^2}, E) = d\cdot \mathrm{rk}(E)$ and $\mu(E) = \frac{\deg_H E}{\mathrm{rk}(E)}$. Furthermore, if $E$ is $\mu$-semistable, then the polynomial $\hat P(\mu(E), \ox_X, d)$ is either $\binom{\frac{\mu(E)} d + 2}{2}$ or $\binom{\frac{\mu(E)} d -\frac{\mu(\ox_X)} d - 1}{2}$.
\end{theorem}
\begin{remark}
In Sec.~\ref{higher} we will present an alternative approach to the universal polynomial and give a more precise formula for higher dimensional $(n > 2)$ normal integral projective schemes.
\end{remark}

Assume $\mathcal A$ is the heart of a bounded t-structure on a triangulated category $\mathcal T$. D.~Happel, I.~ Reiten and S.~Smal{\o} show in~\cite{MR1327209} that with such a torsion pair $\left(\mathcal F^{\bot}, \mathcal F\right)$ of an abelian category $\mathcal A$ so that $\mathrm{Hom}_{\mathcal A}\left(\mathcal F^{\bot}, \mathcal F\right) = 0$ and given $E \in \mathcal A$ there exists a short exact sequence $0 \rightarrow T \rightarrow E \rightarrow F \rightarrow 0$ for some pairs $T \in \mathcal F^{\bot}$ and $F \in \mathcal F$, the full subcategory $\widetilde{\mathcal A}$ is the heart of a bounded t-structure on $\mathcal T$ for $$\widetilde{\mathcal A} := \left\{E \in \mathcal T : H^0(E) \in \mathcal F^{\bot}, H^{-1}(E) \in \mathcal F\right\} = \left\langle \mathcal F[1], \mathcal F^{\bot}\right\rangle.$$ On $\mathrm{D^b}(X)$ this gluing recipe can still be carried out in the same way. In particular, these torsion pairs from the abelian category of coherent sheave $\mathrm{Coh}(X)$ on a noetherian scheme $X$ always exist if the category $\mathrm{Coh}(X)$ has a stability filtration in the sense of Sec.~\ref{stabilityf}.

Thus, we observe that from the $\mu$-stability on $X$ often several torsion pairs can be constructed depending on HN filtration and the vanishing properties of $\mu$-semistable sheaves. Namely, for any given rational number $q \in \mathbb Q$ we obtain a torsion pair $\left(\mathcal F_q^{\bot}, \mathcal F_q\right)$ with the tilted heart $\mathcal  A_q$, whereas for a object $E$ belonging to $\mathcal  A_q$ we arrive with the pair of objects $(F, T)$ at a distinguished triangle $F \rightarrow E \rightarrow T \rightarrow F[1]$ with $F[-1] \in \mathcal F_q$ and $T \in \mathcal F_q^{\bot}$. The former object $F[-1]$ is a torsion free coherent sheaf with the slope $\mu(F[-1]) \leq q$, while the latter object $T$ is either a torsion coherent sheaf or torsion free coherent sheaf with the slope $\mu(T) > q$. In the sequel we argue that these different possibilities realize distinct branches of the $\Dx$-stability of degree $1$ on normal projective surfaces.

Recall from Prop.~\ref{bound} that on a projective scheme $X$ of dimension $n$ over a field $k$ with an ample invertible sheaf $\mathcal O_X(H)$, $\{L_{-t} := \mathcal O_X(-tH)[1]\}_{\hat{\mu}(\mathcal O_X(-tH)) \leq q}$ forms a tilted ample chain fulfilling Prop.~\ref{sc} in the tilted heart $\mathcal A_q$ for some $q = \frac{m_1}{m_2}\in \mathbb Q$ with $m_2 >0$, and induces the tilted numerical polynomial $\widetilde P_t(-) = \sum^{n}_{c=0}(-1)^c\chi\left(\mathcal O_{H^c}, -\right)\binom{t}{c} = \sum^{n}_{c=0}a_c(-)\binom{t}{c}$. Then we can pick $\check L_{-t} = \mathrm C\left( \mathcal O_{H^n}^{\oplus m_2\binom t n + m_1\binom t{n-1}}[-n+1] \rightarrow L_{-t}^{\oplus m_2}\right)$ for $t >> 0$ with the associated numerical polynomial $\check P_{-t}(-) = \chi\left(\check L_{-t}, -\right) = \sum^{n-1}_{c=0}\check a_c(-)\binom{t}{c}$ where $\check a_{n-1}(-) = m_2\cdot a_{n-1}(-) - m_1\cdot a_n(-)$ and $\check a_c(-) = m_2\cdot a_c(-)$ for $c= 0, \dots, n-2$. By construction one could see that  $\check a_{n-1}(E) \geq 0$ for $E \in \mathcal A_q$ and if the equality holds then $E \in \mathrm{Ext}^2(T, F)$ for some $F \in \mathcal F_q$ and torsion coherent sheaf $T \in \mathcal F_q^{\bot}$ of codimension $\geq 2$.
\begin{prop}\label{noetherian}
Suppose $X$ is a normal integral projective scheme over a (unnecessarily perfect) field $k$ with a vary ample invertible sheaf $\mathcal O_X(H)$. Let $\left(\mathcal F_q^{\bot}, \mathcal F_q\right)$ be a torsion pair of $\mathrm{Coh}(X)$ for some $q \in \mathbb Q$ such that $\hat{\mu}(F) \leq q$ for $F \in \mathcal F_q$ and $\hat{\mu}(T) > q$ for torsion free $T \in \mathcal F_q^{\bot}$. Then the tilted heart $\mathcal A_q$ is Noetherian.
\end{prop}
\begin{proof}
Suppose we have a chain of epis for a nontrivial object $E$ in $\mathcal A_q$
$$
E=E_0 \twoheadrightarrow E_1 \twoheadrightarrow E_2 \twoheadrightarrow \cdots
$$
with $\check a_{n-1}(E_i) \geq \check a_{n-1}(E_{i+1})$. Since the value of $\check a_{n-1}(-)$ is discrete, there exists a minimal value of $\check a_{n-1}(E_j)$ for some $j$, we may assume that $\check a_{n-1}(E_j) = \check a_{n-1}(E_{j+1})$ for all $j$. Moreover, there are epimorphisms of cohomolgiy sheaves  by taking long exact sequences in cohomology
$$
H^0(E_0) \twoheadrightarrow H^0(E_1) \twoheadrightarrow H^0(E_2) \twoheadrightarrow \cdots.
$$
This chain has to be stationary as the category of coherent sheaves is Noetherian, so $H^0(E_i) \cong H^0(E_{i+1}) \cong \cdots$ for large enough $i$. Thus we may assume $H^0(E) \cong H^0(E_i)$ for all $i$. 

Now consider these short exact sequences $0 \rightarrow L_i \rightarrow E \rightarrow E_i \rightarrow 0$ which induce a chain of monos
$$
 L_1 \hookrightarrow L_2 \hookrightarrow \cdots \hookrightarrow E
$$
with each $\check a_{n-1}(L_i) = 0$, thus $\check a_{n-1}(H^0(L_i)) = 0$. If $L_j =L_{j+1}=\cdots$ for large enough $j$, then we are done. By taking cohomology sheaves there are monos of coherent sheaves 
$$
H^{-1}(L_1) \hookrightarrow H^{-1}(L_2) \hookrightarrow \cdots \hookrightarrow H^{-1}(E).
$$
This chain also terminates thus we may assume $H^{-1}(L_i) \cong H^{-1}(L_{i+1})$ for all $i$. On the other hand, since $\check a_{n-1}(H^0(L_i)) = 0$, $H^0(L_i)$ is a torsion coherent sheaf supported in at least codimension 2 and thus $H^0(L_i) \hookrightarrow H^0(L_{i+1})$. 

Again the short exact sequences $0 \rightarrow L_i \rightarrow E \rightarrow E_i \rightarrow 0$ give us the long exact sequences of coherent sheaves
$$
0 \rightarrow H^{-1}(L_i) \rightarrow H^{-1}(E) \rightarrow H^{-1}(E_i) \rightarrow H^0(L_i) \rightarrow 0.
$$
Since $H^{-1}(L_i) \cong H^{-1}(L_{i+1})$ for all $i$, the images of the middle morphisms in $H^{-1}(E_i)$ for all $i$ are the same denoted by $Q$. Thus there is a short exact sequence of coherent sheaves
$$
0 \rightarrow Q \rightarrow H^{-1}(E_i) \rightarrow H^0(L_i) \rightarrow 0
$$
for all $i$. Here, $Q$ and $H^{-1}(E_i)$ are torsion free coherent sheaves, and $H^0(L_i)$ is a torsion coherent sheaf supported in at least codimension 2. Using the local cohomology vanishing criterion for depth~\cite{MR0224620} we have $H^0(L_i)^{\vee} :=\mathcal Ext^0(H^0(L_i), \mathcal O_X) = \mathcal Ext^1\left(H^0(L_i), \mathcal O_X\right) = 0$ because of $\mathrm{depth}_{\mathrm{Supp}(H^0(L_i))}(\mathcal O_X) \geq 2$ by Serre's criterion $S_2$. It turns out that $Q^{\vee} \simeq H^0(E_i)^{\vee}$ for all $i$ and the chain of monos of coherent sheaves
$$
H^{-1}(E_1) \hookrightarrow H^{-1}(E_2) \hookrightarrow \cdots \hookrightarrow Q^{\vee\vee}
$$
also has to be stationary and thus $H^{-1}(E_j) \cong H^{-1}(E_{j+1})$ for large enough $j$. This proves that $H^0(L_j) \cong H^0(L_{j+1})$ and hence $L_j = L_{j+1}$ for large enough $j$. 
\end{proof}

To arrive at this $\Dx$-stability of degree $1$ on the  normal projective surface $X$, let us consider a torsion pair $\left(\mathcal F_q^{\bot}, \mathcal F_q\right)$ for some $q = \frac{m_1}{m_2} \in \mathbb Q$ and $m_2 > 0$ with the tilted heart $\mathcal A_q$ of a bounded t-structure on $\mathrm{D^b}(X)$ and the tilted ample chain $\left\{\mathcal O_X(-tH)[1]\right\}_{\hat{\mu}(\mathcal O_X(-tH)) \leq q}$ constructed as above. Performing a cone of a morphism from $\mathcal O_{H^2}^{\oplus m(t)}[-1]$ to $\mathcal O_X(-tH)^{\oplus m_2}[1]$ with $m(t) := m_2\binom{t}{2} + m_1\cdot t + m_0 > 0$ for $t >> 0$ yields the sequence $\left\{\check L_{-t}\right\}$ as
$$
\check L_{-t} =\mathrm C\left( \mathcal O_{H^2}^{\oplus m_2\binom t 2 + m_1\cdot t + m_0}[-1] \longrightarrow \mathcal O_X(-tH)^{\oplus m_2}[1]\right).
$$
In the vicinity of the integer $m_0 > m_2 \cdot \hat P(\hat{\mu}(F) = q, \ox_X, d = \deg_H X)$ for the universal polynomial $\hat P(q, \ox_X, d)$ appearing in Theorem~\ref{ebound}, the sequence $\left\{\check L_{-t}\right\}$ becomes a slope sequence $\left\{L_{-t}^s\right\}$ in $\mathcal A_q$ because in the patch of the slope $\hat{\mu}(F) = q$ for $\mu$-semistable sheaf $F \in \mathcal F_q$ we get
$$
m_2\cdot \chi(\mathcal O_X, F) \leq m_2\cdot \chi(\mathcal O_{H^2}, F)\hat P\left(\frac{m_1}{m_2}, \ox_X, d\right).
$$
Thus, assuming the shift $[1]$ among the subcategory $\mathcal F_q$, there exists the minimal integer $m_{\min}$ at the tilted category $\mathcal A_q$ such that $m_2\cdot\chi(\mathcal O_X, E) > m_{\min}\cdot \chi(\mathcal O_{H^2}, E)$ for $E \in \mathcal A_q$ with $m_2\cdot \chi(\mathcal O_H, E) = m_1\cdot \chi(\mathcal O_{H^2}, E) \neq 0$. This minimal integer is given by
$$
m_{\min} = \min\left\{m \in \mathbb Z : m > m_2 \cdot \hat P(q, \ox_X, d)\right\}
$$
in terms of the denominator $m_2$ of $q$ and the universal polynomial $\hat P(q, \ox_X, d)$ within the normal surface $X$. This boundedness condition prevails in the bounded derived categories of normal projective surfaces over a perfect field $k$ and it thus leads to a slope polynomial $P^s_t(-) = \chi\left(L_{-t}^s, -\right)$ of degree $1$ on the tilted heart $\mathcal A_q$ as we have shown in Prop.~\ref{bound}.

We are now ready to construct the $\Dx$-stability of degree $1$ on $X$. The strategy presented in Prop.~\ref{hnf} finally explains \textit{numerical} $\Dx$-stability data on an additive category $\mathcal A$ as a full subcategory of a Noetherian abelian category that embedded in a $k$-linear triangulated category $\mathcal T$. Since the tilted heart $\mathcal A_q$, via the gluing recipe in Prop.~\ref{noetherian},  of $\mathrm{D^b}(X)$ is a Noetherian abelian category with a slope sequence $\left\{L^s_{-t}\right\}$ and the associated slope polynomial $P^s_t(-)$, the last step, namely to construct Harder-Narasimhan sequences of objects, is easy to achieve. We conclude the construction in the theorem below.       
\begin{theorem}\label{stability1}
Suppose $X$ is a normal projective surface over a perfect field $k$ with a very ample invertible sheaf $\mathcal O_X(H)$. Given a rational number $q = \frac{m_1}{m_2} \in \mathbb Q$, there exists the tilted heart $\mathcal A_q$ of a bounded t-structure on the bounded derived category of coherent sheaves $\mathrm{D^b}(X)$ with a one parameter family of slope sequences $\left\{L_{-t}^s(m_0) : m_0 > m_{\min}\right\}$ and $L_{-t}^s(m_0)[1] \in \mathrm{Ext}^2\left(\mathcal O_{H^2}^{\oplus m_2\binom t 2 + m_1\cdot t + m_0} , \mathcal O_X(-tH)^{\oplus m_2}\right)$ such that the associated slope polynomials $P^s_{t, m_0}(-) = \chi\left(L^s_{-t}(m_0), -\right)$ induce $\Dx$-stabilities of degree $1$, namely Bridgeland stability conditions, on $X$.
\end{theorem}
\begin{cor}
Suppose $X$ is a geometrically normal projective surface over an arbitrary field $k$ with $H^0(X, \mathcal O_X) = k$, i.e., geometrically connected. Then there exists a $\Dx$-stability of degree $1$ or Bridgeland stability induced by the slope polynomial $P^s_{t}(-) = \chi\left(L^s_{-t}, -\right)$ for a slope sequence $\left\{L_{-t}^s\right\}$ constructed as Theorem~\ref{stability1}.
\end{cor}
\begin{proof}
Since $X$ is geometrically normal and geometrically connected, $X_{\bar k}$ is normal integral under base extension $k \subset \bar k$. Using that fact that $E \in \mathrm{Coh}(X)$ is semistable if and only if $E_{\bar k} := E\otimes_k \bar k$ is semistable (see~\cite{MR2665168}) and the construction in Theorem~\ref{stability1}, we are done.
\end{proof}

To realize the full parameter space $\mathcal M^s$ of slope sequences $\left\{L_{-t}^s\right\}$ we first analyze the \textit{numerical equivalence class} of slope sequences $\left\{L_{-t}^s\right\} \equiv \left\{L_{-t'}^s\right\}$ if the associated slope polynomials are isomorphic in the sense of $\chi\left(L_{-t}^s, -\right) = \chi\left(L_{-t'}^s, -\right) $ for $t, t' >> 0$, thus give the same $\Dx$-stability data. In terms of its cone construction, we could replace the object $O_X(-tH)^{\oplus m_2}$ with any nontrivial extension between $\mathcal O_X(-tH)$, i.e., objects of length $m_2$ in the extension-closed full subcategory $\mathcal A_{H} := \langle\mathcal O_X(-tH) \rangle$. These object of finite length forms a finite set and if $H^1(X, \mathcal O_X) = 0$ there is the only one object, $O_X(-tH)^{\oplus m_2}$. On the other hand, $H^2$ is a zero subscheme of $X$ and the space of zero subscheme of length $l = d\cdot \left(m_2\binom t 2 + m_1\cdot t + m_0\right)$ is parametrized by Hilbert scheme $\mathrm{Hilb}^l(X)$ with Hilbert polynomial $P = l(t, d, m_0, m_1, m_2)$. Similarly, $H$ is an effective divisor in the complete linear system $|\mathcal O_X(H)|$. Furthermore, more general nontrivial extensions of skyscraper sheaves on closed points or structure sheaves of $H$ establish the Zariski tangent spaces on these closed points or normal sheaves on $H$, respectively. In sum, each equivalence class is contributed from $\mathrm{Hilb}^l(X)$, $|\mathcal O_X(H)|$, Zariski tangent spaces on closed points, normal sheaves on $H$ and a finite set associated to $\mathcal O_X(-tH)$. 

Therefore, for each $q= \frac{m_1}{m_2}$ the described one parameter family of slope sequences in Theorem~\ref{stability1} modulo numerical equivalence develops a family of $\Dx$-stabilities of degree~$1$ with the slope numerical polynomials $P_{t, m_0}^s(-) = \chi\left(L_{-t}^s(m_0), -\right)$ on $\mathcal A_q$. Furthermore, according to~\cite{MR2373143} each slope sequence class $\left\{L_{-t, q}^s(m_0)\right\} \in \coprod_{m_0 > m_{\min}, m_1, m_2>0} \mathrm{Hilb}^{d\cdot \left(m_2\binom t 2 + m_1\cdot t + m_0\right)}(X)$ with the associated $\Dx$-stability data contributes an \textit{algebraic} Bridgeland stability condition $\left( Z, \mathcal P_{\phi}\right)$ in $\mathrm{Stab}(X)$ via the Euler map $\chi(-, -): \left\{L_{-t, q}^s(m_0) \right\}\mapsto Z(-) = \chi\left(L_{-t, q}^s(m_0), -\right)$ with $t = \sqrt{-1}$. By the deformation of Bridgeland stability conditions the center charge $Z(-) \in \mathbb Z\left[\sqrt{-1}\right]$ thus can be extended to $Z_{\mathbb R}(-) \in\mathbb R\left[\sqrt{-1}\right]$.
\begin{prop}
Let $X$ be a normal projective surface over a perfect field with an ample invertible sheaf $\mathcal O_X(H)$ of degree $d$. There is a natural map $$\chi : \coprod_{m_0 > m_{\min}, m_1, m_2>0} \mathrm{Hilb}^{d\cdot \left(m_2\binom t 2 + m_1\cdot t + m_0\right)}(X) \longrightarrow \mathrm{Stab}(X)$$ where $\mathrm{Stab}(X)$ is the space of Bridgeland stability conditions, a complex manifold.
\end{prop}
\begin{remark}
From Prop.~\ref{deform} given a slope sequence class $\left\{L_{-t}^s(m_0)\right\}$ as above there exists a family of slope sequence classes with different slope polynomials inducing the same $\Dx$-stability data. Indeed, any slope sequence class $\left\{L_{-t}^s\right\}$ can be deformed to this \textit{skeleton} slope sequence class $\left\{L_{-t}^s(m_0)\right\}$ with the same $\Dx$-stability data on $X$.
\end{remark}

\subsection{Hodge Index Theorem and Bogomolov Inequality} \label{hb}

To close this section we finally present the relation of the boundedness theorem~\ref{ebound} to Hodge Index theorem and Bogomolov Inequality in the intersection theory of algebraic geometry. Let us recall the basic definitions and properties of intersection numbers that we use. For the modern comprehensive approach to intersection theory, we refer to~\cite{MR1644323, MR3617981}.

Let $X$ be a proper scheme over a general field $k$ of dimension $n$. The Chern character defines a ring homomorphism from the Grothendieck group $K^0(X)$ of locally free sheaves on $X$ to the \textit{cohomology} group $A^*(X)$ as
$$
\mathrm{ch} : K^0(X) \longrightarrow A^*(X)_{\mathbb Q} := \bigoplus_{r=0}^n A^r(X)_{\mathbb Q}.
$$
The cohomology group has a cup product structure $$A^r(X) \otimes A^s(X) \xrightarrow{\ \cup \ } A^{r+s}(X),$$ and a cap product action of $A^*(X)$ on $A_*(X)$ $$A^r(X) \otimes A_s(X) \xrightarrow{\ \cap \ } A_{r-s}(X),$$ where $A_*(X) := \bigoplus_{r=0}^n A_r(X)$ is the Chow group of algebraic cycles modulo rational equivalence on $X$. If $X$ is nonsingular then $A^*(X) \simeq A_*(X)$. There is a natural homomorphism, the \textit{degree}, from zero cycles $A_0(X)$ to $\mathbb Z$ defined by
$$
\deg(\ax) = \int_X \ax = \sum_x n_x[K(x) : k]
$$
for each zero cycle $\ax = \sum_x n_x[x]$ with the function field $K(x)$. Indeed, $\deg(\ax) = p_*(\ax)$ for the structure morphism $p$ of $X$ to $\mathrm{Spec}(k)$. If $\ax \in A^r(X)$ and $\bx \in A^s(X)$ we denote the intersection cocycle class by $\ax.\bx:= \ax \cup \bx \in A^{r+s}(X)$ satisfying some axioms. 

According to Riemann-Roch theorem for proper schemes of finite type over $k$, for all scheme $X$ of dimension $n$ there is a homomorphism $$\tau = \tau_X : K_0(X) \longrightarrow A_*(X) = \bigoplus_{r=0}^n A_r(X)$$ from the Grothendieck group $K_0(X)$ of coherent sheaves on $X$ to the Chow group $A_*(X)$, and $\tau([\mathcal O_Y]) = [Y] + \mathrm{cycles\ of\ dimension} < \dim Y$ for a closed subvariety $Y$ of $X$.  Again if $X$ is nonsingular, one has $K^0(X) \simeq K_0(X)$. Then for a locally free sheaf $E$ on $X$ its Euler characteristicis is given by the formula
$$
\chi(\mathcal O_X, E) = \deg(\tau(E)) = \deg(\mathrm{ch}(E)\cap \mathrm{Td}(X))= \int_X \mathrm{ch}(E)\cap \mathrm{Td}(X),
$$
where the \textit{Todd class} of $X$ is denoted by $\mathrm{Td}(X) = \tau([\mathcal O_X])= [X] + \ax$ for $\ax \in \bigoplus_{r=0}^{n-1} A_r(X)$ and $\tau(E \otimes F) = \mathrm{ch}(E) \cap \tau(F)$ for $E \in K^0(X), F\in K_0(X)$.

Suppose $C$ is a proper (integral) curve embedded in $X$, for a locally free sheaf $E$ on $X$ we define the degree of $E$ with respect to $C$ as
$$
\deg_C(E) = -\chi\left(\mathcal O_C, E\otimes \ox_X^{\bullet}\right) + \mathrm{rk}(E)\chi(\mathcal O_C, \ox_X^{\bullet}).
$$
Here, $\ox_X^{\bullet}$ is the dualizing complex of $X$ induced by the structure morphism $X \xrightarrow{f} \mathrm{Spec}(k)$, i.e., $\ox_X^{\bullet} = f^!\left(\mathcal O_{\mathrm{Spec}(k)}\right)$. Note that if $C$ is not a locally complete intersection then the degree defined in Def.~\ref{rd} may not work since the homological dimension of $X$ could not be finite. 
\begin{lem}
Let $X$ be a proper scheme over $k$ and $C$ be a proper curve with a proper morphism $f : C \rightarrow X$. For any locally free coherent sheaf $E \in K^0(X)$, one has
$$
\deg_C(E) = \int_C c_1(f^*E) \cap [C] = \int_X c_1(E) \cap f_*[C],
$$
where $c_1(f^*E) \in A^1(C), c_1(E) \in A^1(X)$ are the first Chern classes of $f^*E$ and $E$, respectively, and $[C] \in A_1(C), f_*[C] \in A_1(X)$. 
\end{lem}
\begin{proof}
Assume $\ox_X^{\bullet} = f_X^!(k)$ is the dualizing complex of $X$ associated to the structure morphism $X \xrightarrow{f_X} \mathrm{Spec}(k)$. Then we have $\ox_C^{\bullet} = f_C^!(k) = (f_X \circ f)^!(k) = f^!\circ f_X^!(k) = f^!(\ox_X^{\bullet})$ for the dualizing complex of $C$ associated to $C \xrightarrow{f_C} \mathrm{Spec}(k)$. By the adjunction $f_* \dashv f^!$,
\begin{align*}
\deg_C(E) &= -\chi\left(f_*\mathcal O_C, E\otimes \ox_X^{\bullet}\right) + \mathrm{rk}(E)\chi(f_*\mathcal O_C, \ox_X^{\bullet})\\ &= -\chi_C\left(\mathcal O_C, f^*E\otimes \ox_C^{\bullet}\right) + \mathrm{rk}(E)\chi_C(\mathcal O_C, \ox_C^{\bullet})\\ &= -\int_C (\mathrm{ch}(f^*E) \cap \tau(\ox_C^{\bullet})-\mathrm{rk}(E)\tau(\ox_C^{\bullet}))\\ &= \int_C c_1(f^*E) \cap [C].
\end{align*} 
For the last equality we use the fact $\tau(\ox_C^{\bullet}) = -[C] + \ax$ for $\ax \in A_0(C)$ and $\mathrm{rk}(E) = \mathrm{rk}(f^*E)$.
\end{proof}

Let us turn to cases of proper surfaces $X$, i.e., integral schemes of finite type over a perfect field $k$ of dimension $n = 2$. We are going to show that one can find a Weil divisor on X so that for any Cartier divisor of trivial degree with respect to this Weil divisor its self-intersection number can not be positive.
\begin{theorem}[Hodge Index Theorem]\label{hodge1}
Suppose $X$ is a proper surface over a perfect field~$k$. Then there exists a proper curve $C \subset X$ such that if $L$ is an invertible sheaf with $\deg_C(L) = \int_Xc_1(L) \cap [C] = 0$, then $c_1(L)^2 := \int_X(c_1(L).c_1(L))\cap [X] \leq 0$.
\end{theorem}
\begin{proof}
By the Chow lemma~\cite{MR217084, MR0463157}, there is a proper birational surjective morphism $\pi : X' \rightarrow X$ with a projective integral scheme $X'$ of finite type over $k$. If $X'$ is not normal, then by normalization there is a normal projective integral scheme $X''$ of finite type over $k$ with a finite birational morphism $\pi': X'' \rightarrow X'$. Thus we may assume $X'$ is normal.

Let $\mathcal O_{X'}(H)$ be a very ample invertible sheaf on $X'$ and $[H] \in A_1(X')$ be the associated cycle. Set $[C] := \pi_*[H] \in A_1(X)$ and consider an invertible sheaf $L \in \mathrm{Pic}(X)$ with trivial $\deg_C(L)$. The projection formula implies $0 = \deg_C(L) = \int_Xc_1(L)\cap \pi_*[H] = \int_{X'}c_1(\pi^*L)\cap [H] = \deg_H(\pi^*L)$ and thus $\deg_H(\pi^*(L)^{\otimes m}) =  \int_{X'}m\cdot c_1(\pi^*L)\cap [H] = 0$ for any integer $m > 0$. Using the boundedness theorem~\ref{ebound} and the fact that $\pi^*L^{\otimes m}$ is stable for all $m$, we have $$\chi(\mathcal O_{X'}, \pi^*L^{\otimes m}) \leq d\cdot \hat P(\ox_{X'}, d)$$
with $d=\deg_{H}X'$, and $\hat P(\ox_{X'}, d)$ is a constant depending on $\ox_{X'}, d$ as $\mu(\pi^*L^{\otimes m}) = 0$.

On the other hand, the Riemann-Roch formula for normal surfaces~\cite{MR1644323} says $$\chi(\mathcal O_{X'}, \pi^*L^{\otimes m}) = \frac{m^2}{2} \int_{X'}(c_1(\pi^*L).c_1(\pi^*L)) \cap [X'] + \frac m 2 \int_{X'}c_1(\pi^*L) \cap [K] + \chi(\mathcal O_{X'}, \mathcal O_{X'}).$$ As $m$ goes to infinity, this will contradict to $\chi(\mathcal O_{X'}, \pi^*L^{\otimes m}) \leq d\cdot \hat P(\ox_{X'}, d)$ if $c_1(\pi^*L)^2 > 0$. Hence $c_1(L)^2 = c_1(\pi^*L)^2 \leq 0$.
\end{proof}
\begin{remark}
If there is an invertible sheaf $\bar H$ on an irreducible surface $X$ over an algebraically closed field with $c_1\left(\bar H\right)^2 > 0$, then for any invertible sheaf $L$, not numerically equivalent to $0$, of $\int_{X}\left(c_1(L).c_1\left(\bar H\right)\right) \cap [X] = 0$, Kleiman shows that $c_1(L)^2 < 0$  in~\cite{10.1007/BFb0066296}.
\end{remark}

More generally, the same techniques as above yield the following version of Hodge Index Theorem.
\begin{theorem}[Hodge Index Theorem]
Let $X$ is a proper surface over a perfect field~$k$ with a proper curve $C \subset X$ as the one in theorem~\ref{hodge1}. Then given an invertible sheaf $L \in \mathrm{Pic}(X)$, we have the inequality $$c_1(L)^2 \cdot [C]^2 \leq \left(\int_X c_1(L)\cap [C]\right)^2,$$ where $[C]^2$ is a positive integer depending on $[C]$.
\end{theorem}
\begin{proof}
Let $\pi : X' \rightarrow X$ be the finite birational surjective morphism with a normal projective surface $X'$ and a very ample invertible sheaf $\mathcal O_{X'}(H)$ so that $[C] = \pi_*[H]$. Again by the boundedness theorem~\ref{ebound}, for $m > 0$ and $L \in \mathrm{Pic}(X)$, 
$$
\chi(\mathcal O_{X'}, \pi^*L^{\otimes m}) \leq d\cdot \hat P(m \cdot \deg_H(\pi^*L), \ox_{X'}, d) = d\cdot \left(\frac{m^2}{2d^2} \deg_H(\pi^*L)^2 \right) + O(m)
$$
with $d = \deg_H X' = \deg_H(\mathcal O_{X'}(H)) = c_1(\mathcal O_{X'}(H))^2$. Compared with Riemann-Roch formula, set $[C]^2:= d$ and one can easily see $$[C]^2\cdot c_1(L)^2 = d \cdot c_1(\pi^*L)^2 \leq \deg_H(\pi^*L)^2 = \left(\int_X c_1(L)\cap [C]\right)^2.$$
\end{proof}

Next, we derive the Bogomolov Inequality on normal projective surfaces $X$ for strongly $\mu$-semistable locally free sheaves, whose tensor powers remain $\mu$-semistable. Indeed, if $\mathrm{char}\, k = 0$, then $E^{\otimes m}$ is still $\mu$-semistable for each $\mu$-semistable sheaf $E$ on $X$. However, for $\mathrm{char}\, k = p > 0$, this statement is not ture in general. Moreover, the Frobenius pullback $F^*E$ of $\mu$-semistable sheaf $E$ need not to be $\mu$-semistable. The sheaf whose all Frobenius pullbacks $(F^n)^*E$ remain $\mu$-semistable is called strongly $\mu$-semistable. Langer in ref.~\cite{MR2051393} shows that a tensor product of strongly $\mu$-semistable sheaves is strongly $\mu$-semistable as well, and this result had also been observed earlier in theorem 3.23 of~\cite{MR742599}.
\begin{theorem}[Bogomolov Inequality]\label{bogomolov}
Suppose $X$ is a normal projective surfaces over a perfect field $k$ with a very ample invertible sheaf $\mathcal O_X(H)$. If $E$ is a  strongly $\mu$-semistable locally free sheaf, then 
$$
\Dx(E) := \int_X\Dx(E)\cap [X] := (\mathrm{rk}(E)-1)c_1(E)^2 - 2\cdot \mathrm{rk}(E)c_2(E) \leq 0.
$$
Here, $c_2(E) := \int_Xc_2(E)\cap [X]$.
\end{theorem}
\begin{proof}
As $E$ is locally free, the Chern character of $E$ is 
\begin{align*}
\mathrm{ch}(E) &= \mathrm{rk}(E) + c_1(E) + \frac 1 2\left(c_1(E)^2 - 2\cdot c_2(E)\right)\\ &=  \mathrm{rk}(E) + c_1(E) + \frac{1}{2\cdot \mathrm{rk}(E)}\left(\Dx(E) + c_1(E).c_1(E)\right).
\end{align*}
$E^{\vee}$, dual of $E$, is strongly $\mu$-semistable and thus $E' := E\otimes E^{\vee}$ is also strongly $\mu$-semistable with $c_1(E') = 0$ and $\Dx(E') = 2\cdot\mathrm{rk}(E)^2\Dx(E)$. Let $n$ be a positive integer, then $\Dx(E'^{\otimes n}) = n\cdot \mathrm{rk}(E')^{2(n-1)}\Dx(E')$. As above, we can compute $\chi(\mathcal O_X, E'^{\otimes n})$ by the Riemann-Roch formula: 
$$
\chi(\mathcal O_X, E'^{\otimes n}) = \frac{n}{2} \mathrm{rk}(E')^{n-2}\Dx(E') + \mathrm{rk}(E')^{n}\chi(\mathcal O_X, \mathcal O_X).
$$
On the other hand, $E'^{\otimes n}$ is strongly $\mu$-semistable as $E$ is strongly $\mu$-semistable, thus the boundedness theorem~\ref{ebound} implies 
$$
\frac{\chi(\mathcal O_X, E'^{\otimes n})}{\mathrm{rk}(E')^n} \leq d\cdot \hat P\left(\ox_X, d = \deg_H X\right).
$$
Therefore, if $\Dx(E)$ is positive, as $n$ goes to infinity, the Riemann-Roch formula yields
$$
\frac{\chi(\mathcal O_X, E'^{\otimes n})}{\mathrm{rk}(E')^n} = \frac{n}{2\cdot \mathrm{rk}(E')^{2}} \Dx(E') + \chi(\mathcal O_X, \mathcal O_X)
$$
which would contradict the boundedness theorem above.
\end{proof}
\begin{cor}
If $X$ is nonsingular, then $\Dx(E) \leq 0$ for all strongly $\mu$-semistable torsion free sheaf $E$.
\end{cor}
\begin{proof}
The double dual $E^{\vee\vee}$ is a strongly $\mu$-semistable locally free sheaf since $X$ is nonsingular, and then $\Dx(E)$ and $\Dx(E^{\vee\vee})$ are related by $\Dx(E^{\vee\vee}) = \Dx(E) + 2\mathrm{rk}(E)l(E^{\vee\vee}/E) \geq \Dx(E)$, where $l(-)$ denotes the length of zero schemes in parentheses. Hence $0 \geq \Dx(E^{\vee\vee}) \geq \Dx(E)$ by theorem~\ref{bogomolov}.
\end{proof}
\begin{remark}
This corollary is one of the main results proved by Langer in~\cite{MR2051393} where he shows that for every torsion free sheaf $E$ on nonsingular projective variety $X$ there exists $k_0$ such that all HN factors of $(F^{k_0})^*E$ are strongly $\mu$-semistable. Note that the definition of the discriminant differs from the one Langer uses by the sign. Moreover, if $\mathrm{char}\, k = 0$ then all $\mu$-semistable sheaves are strongly $\mu$-semistable and the corollary implies the classical Bogomolove Inequality.
\end{remark}
In a similar fashion it is straightforward to establish the correspondence between the discriminant $\Dx(E)$ and the geometry of  ambient normal projective surface $X$, where the $\mu$-~semistable torsion free sheaf $E$ is represented by a bounded complex of locally free coherent sheaves. The characters of geometry of $X$ are obtained by the Riemann-Roch formula and the universal polynomial $\hat P$ in the boundedness theorem~\ref{ebound}. We illustrate the result in the following proposition.
\begin{prop}
Assume $X$ is a normal projective surface over a perfect field $k$ with a very ample invertible sheaf $\mathcal O_X(H)$. Let $E$ be a $\mu$-semistable torsion free sheaf represented by a finite complex of locally free coherent sheaves. Then the upper bound of $\Dx(E)$ is given by
$$
\Dx(E) \leq 2d\cdot \mathrm{rk}(E)^2\hat P(\mu(E), \ox_X, d) - c_1(E)^2 - \mathrm{rk}(E) \int_Xc_1(E)\cap [K] - 2\cdot\mathrm{rk}(E)^2\chi(\mathcal O_X, \mathcal O_X)
$$
with $d = \deg_H X, \mu(E) = \frac{\deg_H(E)}{\mathrm{rk}(E)}$ and $\mathrm{Td}(X) = [X] +\frac 1 2 [K] + \frac{1}{12}\ax$ for $[K] \in A_1(X), \ax \in A_0(X)$. Here, $\hat P(\mu(E), \ox_X, d)$ is the binomial coefficient function determined in the boundedness theorem~\ref{ebound}.
\end{prop}
\begin{remark}
If $X$ is nonsingular, compared with Langer's Bogomolov inequality for $\mu$-semistable torsion free sheaves in~\cite{MR2051393}, our bound is independent of $\mathrm{char}\, k = p > 0$ with lower order in $\mathrm{rk}(E)$, but corresponds to $c_1(E).c_1(H), c_1(E).c_1(\ox_X), c_1(\ox_X).c_1(H)$ and $c_1(E)^2$. For the case of rank two $\mu$-semistable bundle on nonsingular projective surfaces, Hein in~\cite{MR2240369} proves a weak version of Bogomolov's inequality independent of $\mathrm{char}\, k = p$. Indeed, our form is more closer to Hein's Weak Bogomolov Inequality.
\end{remark}
In particular, this inequality in the proposition gives an upper bound of $\mathrm{ch}_2(E)$:
$$
\int_X\mathrm{ch}_2(E)\cap [X] \leq d\cdot\mathrm{rk}(E)\hat P(\mu(E), \ox_X, d) - \frac 1 2  \int_Xc_1(E)\cap [K] - \mathrm{rk}(E)\chi(\mathcal O_X, \mathcal O_X).
$$
Conversely, the Hirzebruch-Riemann-Roch formula combined with Hodge Index Theorem and classical Bogomolov Inequality yields an upper bound of the Euler characteristic of $\mu$-semistable torsion free sheaf on nonsingular projective surfaces over $k$ of $\mathrm{char}\, k = 0$, that is, a boundedness theorem similar to our boundedness theorem~\ref{ebound} but a different polynomial $\hat P$. However, coherent sheaves on singular surfaces may not have bounded locally free resolutions, in this case we have no way at present to extend the definitions of Chern classes and Chern characters fulfilling the Whitney formula and the product formula to all coherent sheaves in general.

\section{Effective restriction theorem in higher dimension} \label{higher}

In this section we first show the effective restriction theorem of $\mu$-semistable torsion free sheaves on normal projective surfaces over a perfect field $k$ by the boundedness theorem~\ref{ebound}. Finally we extend our boundedness theorem to higher dimensional normal integral projective scheme over $k$ and thus the construction of $\Dx$-stability of degree $(\dim X - 1)$, Hodge Index Theorem, Bogomolov Inequality, and effective restriction theorem automatically hold for normal integral projective schemes of arbitrary dimension by the same arguments as in surfaces.

At first we recall the technical lemma which is needed in proofs of Theorem~\ref{lbound} and Theorem~\ref{ltbound}.
\begin{lem}[Lemma 1.4.~\cite{MR2051393}]\label{lan}
Let $r=\sum_{i=0}^m r_i$ with positive real numbers $r_i$ and $\mu_0 > \mu_1 > \cdots > \mu_m$ be real numbers. Set $r\mu = \sum_{i=1}^mr_i\mu_i$, then one has the inequality:
$$
\sum_{i<j}r_ir_j(\mu_i - \mu_j)^2 \leq r^2(\mu_0 - \mu)(\mu - \mu_m).
$$
\end{lem}
Before proving the effective restriction theorem, let us give an alternative approach with this technical lemma to the binomial polynomial $\hat P$ in the boundedness theorem~\ref{ebound} for $\mu$-semistable torsion free sheaves. Fix a normal integral projective surface X over $k$ with a finite separable surjective morphism $f: X \rightarrow \mathbf P^2$ such that $\mathcal O_X(H) = f^*(\mathcal O_{\mathbf P^2}(1))$. Given a $\mu$-semistable torsion free sheaf $E$, the Euler characteristic of $E$ is $\chi(\mathcal O_X, E) = \chi(\mathcal O_{\mathbf P^2}, f_*E) = \sum_{i=0}^m\chi(\mathcal O_{\mathbf P^2}, F_i)$ where all $F_i$ are $\mu$-semistable factors of HN filtration of $f_*E$ on $\mathbf P^2$ with $\mu(F_0) > \mu(F_1) > \cdots > \mu(F_m)$. The lemma~\ref{p2} says that $\chi(\mathcal O_{\mathbf P^2}, F_i) \leq \chi(k(p), F_i) \cdot \binom{\hat{\mu}(F_i)}{2}$ with $\hat{\mu}(F_i) = -\frac{\chi(\mathcal O_{\mathbf P^1}, F_i)}{\chi(k(p), F_i)}$ for $i = 0, \dots, m$. Therefore, by directly computation, 
\begin{align*}
\chi(\mathcal O_X, E) - \chi(\mathcal O_{H^2}, E)\binom{\hat{\mu}(E)}{2} &= \sum_{i=0}^m\left(\chi(\mathcal O_{\mathbf P^2}, F_i) - \chi(k(p), F_i) \binom{\hat{\mu}(F_i)}{2} \right) \\
&\quad + \sum_{i=0}^m\chi(k(p), F_i) \binom{\hat{\mu}(F_i)}{2} - \chi(\mathcal O_{H^2}, E)\binom{\hat{\mu}(E)}{2} \\
&\leq  \sum_{i=0}^m\chi(k(p), F_i) \binom{\hat{\mu}(F_i)}{2} - \chi(\mathcal O_{H^2}, E)\binom{\hat{\mu}(E)}{2}.
\end{align*}
Set $d:= \deg_H X, r_i := \mathrm{rk}(F_i) = \chi(k(p), F_i)$ and $r:= \mathrm{rk}(E)$, then $r\cdot d =  \chi(\mathcal O_{H^2}, E) = \sum_{i=0}^m\chi(k(p), F_i) = \sum_{i=0}^m r_i$ and $r d \hat{\mu}(E) = -\chi(\mathcal O_H, E) = -\sum_{i=0}^m\chi(\mathcal O_{\mathbf P^1}, F_i) = \sum_{i=0}^m r_i \hat{\mu}(F_i)$. Then one can deduce
\begin{align*}
\sum_{i=0}^m\chi(k(p), F_i)\binom{\hat{\mu}(F_i)}{2} - \chi(\mathcal O_{H^2}, E)\binom{\hat{\mu}(E)}{2} &= \frac 1 2\sum_{i=0}^mr_i\hat{\mu}(F_i)^2 - \frac 1 2 \left(\sum_{i=0}^mr_i\right)\frac{\left(\sum_{i=0}^m r_i \hat{\mu}(F_i)\right)^2}{\left(\sum_{i=0}^mr_i\right)^2}\\
&= \frac{1}{2rd}\left(\sum_{j=0}^mr_j\sum_{i=0}^mr_i\hat{\mu}(F_i)^2 - \left(\sum_{i=0}^m r_i \hat{\mu}(F_i)\right)^2\right).
\end{align*}
Here,  the technical lemma~\ref{lan} implies 
\begin{align*}
\sum_{j=0}^mr_j\sum_{i=0}^mr_i\hat{\mu}(F_i)^2 - \left(\sum_{i=0}^m r_i \hat{\mu}(F_i)\right)^2 &= \sum_{i<j}r_ir_j\left(\hat{\mu}(F_i)^2 + \hat{\mu}(F_j)^2\right) - 2r_ir_j\hat{\mu}(F_i)\hat{\mu}(F_j)\\ &= \sum_{i<j}r_ir_j\left(\hat{\mu}(F_i) - \hat{\mu}(F_j)\right)^2 \\&\leq r^2 d^2(\hat{\mu}(F_0) - \hat{\mu}(f_*E))(\hat{\mu}(f_*E) - \hat{\mu}(F_m)).
\end{align*}
Furthermore, by the Prop.~\ref{max} and $\hat{\mu}(F_0) = \hat{\mu}_{\max}(f_*E)$, we obtain 
$$
\hat{\mu}(F_0) - \hat{\mu}(f_*E) \leq \frac{\chi(\mathcal O_H, \mathcal O_X)}{\chi(\mathcal O_{H^2}, \mathcal O_X)} + 2 = 2 - \hat{\mu}(\mathcal O_X).
$$
Similarly, $\hat{\mu}(F_m) = \hat{\mu}_{\min}(f_*E)$ and the Prop.~\ref{min} lead to 
$$
\hat{\mu}(f_*E) - \hat{\mu}(F_m) \leq - \frac{\chi(\mathcal O_H, \ox_X)}{\chi(\mathcal O_{H^2}, \mathcal O_X)} + 1 = 1 + \hat{\mu}(\ox_X).
$$
Using these relations, we finally arrive at
\begin{theorem}\label{lbound}
For any $\hat{\mu}$-semistable torsion free sheaf $E$ on a normal integral projective surface $X$ over a perfect field $k$ with a fixed very ample invertible sheaf $\mathcal O_X(H)$, the upper bound of the Euler characteristic of $E$ is given by 
\begin{align*}
\chi(\mathcal O_X, E) &\leq  \chi(\mathcal O_{H^2}, E)\hat P(\hat{\mu}(E), \ox_X, d) \\ &= \chi(\mathcal O_{H^2}, E)\left(\binom{\hat{\mu}(E)}{2} + \frac 1 2 (2 - \hat{\mu}(\mathcal O_X))(1 + \hat{\mu}(\ox_X)) \right),
\end{align*}
where $\ox_X$ is the dualizing sheaf of $X$ and $d = \deg_H X$.
\end{theorem}

\subsection{Effective restriction theorem in surfaces} \label{efr}

To prove the effective restriction theorem, we will need the following more precise formula of the universal polynomial $\hat P$ in the boundedness theorem~\ref{ebound} for general torsion free coherent sheaves on $X$.
\begin{theorem}\label{ltbound}
Suppose $X$ is a normal integral projective surface over a perfect field $k$ with a very ample invertible sheaf $\mathcal O_X(H)$. Let $G$ be a torsion free coherent sheaf on $X$, then
$$
\chi(\mathcal O_X, G) \leq  \chi(\mathcal O_{H^2}, G) \left(\hat P(\hat{\mu}(G), \ox_X, d) + \frac{1}{2}\left(\hat{\mu}_{\max}(G) - \hat{\mu}(G)\right)(\hat{\mu}(G) - \hat{\mu}_{\min}(G))\right),
$$
where $d = \deg_H X$, $\ox_X$ is the dualizing sheaf of $X$, and $\hat P$ is the binomial polynomial for $\mu$-semistable torsion free sheaves in Theorem~\ref{ebound}, or the polynomial in Theorem~\ref{lbound}.
\end{theorem}
\begin{proof}
This proof is similar to Theorem~\ref{lbound}. Let $G_i$'s be the $\hat{\mu}$-semistable factors of the HN filtration of $G$ on $X$ with $\hat{\mu}(G_0) > \hat{\mu}(G_1) > \cdots > \hat{\mu}(G_m)$. Set $r_i:=\mathrm{rk}(G_i)$ and $r:= \mathrm{rk}(G)$. Then 
\begin{align*}
\chi(\mathcal O_X, G) &- \chi(\mathcal O_{H^2}, G)\hat P(\hat{\mu}(G), \ox_X, d) = \sum_{i=0}^m\left(\chi(\mathcal O_X, G_i) - \chi(\mathcal O_{H^2}, G_i)\hat P(\hat{\mu}(G_i), \ox_X, d)\right)\\
&\quad + \sum_{i=0}^m\chi(\mathcal O_{H^2}, G_i) \hat P(\hat{\mu}(G_i), \ox_X, d) - \chi(\mathcal O_{H^2}, G)\hat P(\hat{\mu}(G), \ox_X, d) \\
&\leq  \sum_{i=0}^m\chi(\mathcal O_{H^2}, G_i) \hat P(\hat{\mu}(G_i), \ox_X, d) - \chi(\mathcal O_{H^2}, G)\hat P(\hat{\mu}(G), \ox_X, d)\\
&=  \frac{1}{2d}\sum_{i=0}^mr_i\mu(G_i)^2 - \frac{1}{2d}\left(\sum_{i=0}^mr_i\right)\frac{\left(\sum_{i=0}^m r_i\mu(G_i)\right)^2}{\left(\sum_{i=0}^mr_i\right)^2} \\
&=  \frac{1}{2rd}\left(\sum_{j=0}^mr_j\sum_{i=0}^mr_i\mu(G_i)^2 - \left(\sum_{i=0}^m r_i \mu(G_i)\right)^2\right)\\
&\leq \frac{\mathrm{rk}(G)}{2d}(\mu_{\max}(G) - \mu(G))(\mu(G) - \mu_{\min}(G)).
\end{align*}
\end{proof}

Now we are ready to prove the effective restriction theorem below.
\begin{theorem}\label{effective}
Let $E$ be a $\mu$-stable torsion free sheaf of $\mathrm{rk}(E) > 1$ on a normal projective surface $X$ over a perfect field $k$ with a very ample invertible sheaf $\mathcal O_X(H)$. Assume $D \in |\mathcal O_X(lH)|$ is a normal divisor such that $E|_D$ is torsion free. If
$$
l > 2(1 - \mathrm{rk}(E))\left(\chi(\mathcal O_X, E) - d\cdot\mathrm{rk}(E)\hat P(\hat{\mu}(E), \ox_X, d)\right) + \frac{1}{d\cdot\mathrm{rk}(E)(\mathrm{rk}(E) - 1)},
$$
then $E|_D$ is $\mu$-stable.
\end{theorem}
\begin{proof}
If $E|_D$ is not $\mu$-stable, there exists a short exact sequence $0 \rightarrow S \rightarrow E|_D \rightarrow T \rightarrow 0$ with the minimal destablizing quotient sheaf $T$ of $E|_D$, or $E|_D$ is $\mu$-semistable on $D$. Either would implies $\mu(T) \leq \mu(E|_D)$. Then this short exact sequence induces two short exact sequences $0 \rightarrow G \rightarrow E \rightarrow T \rightarrow 0$ and $0 \rightarrow E(-D) \rightarrow G \rightarrow S \rightarrow 0$. Set $\rho := \mathrm{rk}(S)$ on $D$ and $r := \mathrm{rk}(E)$, then 
$$
\mu(G) = \frac{\deg_H G}{r} = \frac{\deg_H E - (r - \rho)ld}{r} = \mu(E) - \frac{r - \rho}{r}ld = \mu(E(-D)) + \frac{\rho}{r}ld.
$$
Since $E$ and $E(-D)$ are $\mu$-stable, we obtain 
\begin{align*}
\mu_{\max}(G) - \mu(G) &= \mu_{\max}(G) - \mu(E) +  \frac{r - \rho}{r}ld \leq \frac{r - \rho}{r}ld - \frac{1}{r(r-1)},\\
\mu(G) - \mu_{\min}(G) &= \mu(E(-D)) - \mu_{\min}(G) + \frac{\rho}{r}ld \leq \frac{\rho}{r}ld - \frac{1}{r(r-1)}.
\end{align*}
On the other hand, the lemma~\ref{Hilbert} says
\begin{align*}
\chi(\mathcal O_X, E(-D)) &= \chi(\mathcal O_{H^2}, E)\binom{-l}{2} + \chi(\mathcal O_H, E)l + \chi(\mathcal O_X, E), \\
\chi(\mathcal O_X, E|_D) &= - \chi(\mathcal O_{H^2}, E)\binom{-l}{2} -  \chi(\mathcal O_H, E)l.
\end{align*}
Then $\mu(T) \leq \mu(E|_D)$ implies
$$
\chi(\mathcal O_X, T) \leq -(r-\rho)d\binom{-l}{2} - \frac{r - \rho}{r}l\chi(\mathcal O_H, E),
$$
and thus $\chi(\mathcal O_X, G) = \chi(\mathcal O_X, E) - \chi(\mathcal O_X, T) \geq \chi(\mathcal O_X, E) + (r-\rho)d\binom{-l}{2} + \frac{r - \rho}{r}l\chi(\mathcal O_H, E)$.
Hence, one has 
\begin{align*}
\chi(\mathcal O_X, G) &- \chi(\mathcal O_{H^2}, G)\hat P(\hat{\mu}(G), \ox_X, d) \geq \chi(\mathcal O_X, E) - \chi(\mathcal O_{H^2}, E)\hat P(\hat{\mu}(E), \ox_X, d)\\
& +(r-\rho)d\binom{-l}{2} + \frac{r - \rho}{r}l\chi(\mathcal O_H, E) + (r - \rho)ld\left(\hat{\mu}(E) - \frac 1 2\right) - \frac{(r-\rho)^2}{2r}l^2d.
\end{align*}
Applying Theorem~\ref{ltbound} to $G$ we finally arrive at
$$
\chi(\mathcal O_X, E) - \chi(\mathcal O_{H^2}, E)\hat P(\hat{\mu}(E), \ox_X, d) \leq \frac{-l}{2(r -1)} + \frac{1}{2dr(r-1)^2}.
$$
Therefore, 
$$
l \leq 2(1 - \mathrm{rk}(E))\left(\chi(\mathcal O_X, E) - d\cdot\mathrm{rk}(E)\hat P(\hat{\mu}(E), \ox_X, d)\right) + \frac{1}{d\cdot\mathrm{rk}(E)(\mathrm{rk}(E) - 1)},
$$
a contradiction.
\end{proof}
\begin{cor}
Let $E$ be a $\mu$-semistable torsion free sheaf of $\mathrm{rk}(E) > 1$ on $X$ and $D \in |\mathcal O_X(lH)|$ is a normal divisor such that the restrictions to $D$ of all $\mu$-stable factors of Jordan-H\"{o}lder filtration of $E$ are torsion free. If 
$$
l > 2(1 - \mathrm{rk}(E))\left(\chi(\mathcal O_X, E) - d\cdot\mathrm{rk}(E)\hat P(\hat{\mu}(E), \ox_X, d)\right) + \frac{1}{d\cdot\mathrm{rk}(E)(\mathrm{rk}(E) - 1)},
$$
then $E|_D$ is $\mu$-semistable.
\end{cor}
\begin{proof}
Note that the first term of $E$ on the right side of the inequality is larger than the same term of any $\mu$-stable factors.
\end{proof}
\begin{remark}
The ideal of our argument follows the proofs of Langer's effective restriction theorem and Bogomolov's effective restriction theorem. Either needed the inequality of Chern classes or discriminant of $\mu$-semistable torsion free sheaves. However, as in the previous discussion the only way possible to define Chern classes and Chern characters on all coherent sheave is to resolve each coherent sheaf by a bounded complex of locally free coherent sheaves, which may not exist on singular schemes in general.
\end{remark}

\subsection{Boundedness for higher dimensional schemes} \label{bh}

The final task is to extend the boundedness theorem~\ref{ebound} or~\ref{lbound} on normal projective surfaces to normal integral projective scheme of arbitrary dimension of finite type over a perfect field $k$. Let us now turn to a normal integral projective scheme $X$ of dimension $n$ of finite type over $k$ with a finite separable morphism $f : X \rightarrow \mathbf P^n$ such that the very ample invertible sheaf $\mathcal O_X(H) = f^*(\mathcal O_{\mathbf P^n}(1))$. As in the context of Sec.~\ref{sec:nv}, we first generalize Lemma~\ref{p2} for $\mu$-semistable sheaves on the projective plane $\mathbf P^2$ to the $n$-dimensional projective space $\mathbf P^n$. We would need the following restriction theorem of strongly $\mu$-semistable sheaves on projective spaces which is non-effective version of Theorem~4.1 in ref.~\cite{MR2727611}. This proof is a slightly modification of Langer's proof of Theorem~4.1 in~\cite{MR2727611} which is a combination of proofs of Grauert-M\"ulich, Flenner's~\cite{MR780080}, and Mehta-Ramanathan restriction theorem~\cite{MR649194}. For the detailed proofs of these theorems we refer to the book~\cite{MR2665168}.
\begin{theorem}\label{ls}
Given any $\mu$-semistable torsion free sheaf $E$ of rank $r \geq 2$ on $\mathbf P^n$ over an algebraically closed field $k$, and let $a >> 0$ be any sufficiently large integer, then the restriction of $E$ to the generic hypersurface in $|aH|$ is strongly $\mu$-semistable.
\end{theorem}
\begin{proof}
If char $k = 0$, by Flenner's, or Mehta-Ramanathan restriction theorem the restriction of $E$ to a general hypersurface $D \in |aH|$ with $a >> 0$ is again $\mu$-semistable and torsion free, thus strongly $\mu$-semistable. 

Assume char $k > 0$, by Mehta-Ramanathan restriction theorem $E|_D$ is $\mu$-semistable and torsion free for general $D \in |aH|$ with $a >> 0$. Suppose $E|_D$ is not strongly $\mu$-semistable and $\Pi = |aH|$, then we have $\Pi \xleftarrow{\, p\, } Z \xrightarrow{\, q\, } \mathbf P^n$ for the incidence variety $Z = \{(D, x) \in \Pi \times \mathbf P^n : x \in D \}$ such that $(F^m)^*E$ has a relative HN filtration for some $m > 0$ and there exists an open subset $U$ of $\Pi$ so that for all $s \in U$ the fiber of the relative HN filtration forms the HN filtration of $E|_D \simeq (q^*E)_{Z_s}$ with strongly $\mu$-semistable quotients by Theorem~2.7 in~\cite{MR2051393}. Here $F^m : \mathbf P^n \rightarrow \mathbf P^n$ is the $m$-th absolute Frobenius morphism.

Let us take such minimal $m$ and thus $E_0 \subset E_1 \subset \cdots \subset E_l = q^*((F^m)^*E)$ be the relative HN filtration with $E^i := E_i/E_{i+1}$ which does not descend. Then by Theorem~5.1 in~\cite{MR291177}, we get a nontrivial morphism 
$$
E_i \rightarrow \Ox_Z \otimes (F^m)^*E/E_i
$$
such that
$$
\mu_{\min}(E_i \otimes ((F^m)^*E/E_i)^*)_{Z_s} ) \leq \max\{\mu_{\max}((\Ox_{Z/\mathbf P^n})_{Z_s}), \mu_{\max}(q^*\Ox_{\mathbf P^n})_{Z_s}\}
$$
for generic fiber $Z_s$ of $p$. Note that $\Ox_{\mathbf P^n} \hookrightarrow \mathcal O_{\mathbf P^n}(-1)^{\oplus n+1}$, one has $\mu_{\max}(\Ox_{\mathbf P^n}) < 0$ and the restriction of quotients of HN filtration of $\Ox_{\mathbf P^n}$ to a general hypersurface is $\mu$-semistable for $a >> 0$ by  Mehta-Ramanathan restriction theorem. Indeed, the general fact proved by Maruyama~\cite{MR571438} said that the restriction of a $\mu$-semistable sheaf of rank $< n$ to a general hyperplane is $\mu$-semistable as well. Thus we get $\mu_{\max}(q^*\Ox_{\mathbf P^n})_{Z_s} < 0$, and
$$
\mu\left(E^i_{Z_s}\right) - \mu\left(E^{i+1}_{Z_s}\right) \leq \mu_{\max}((\Ox_{Z/\mathbf P^n})_{Z_s}).
$$
Here $\det\left(E^i\right) \simeq p^*\mathscr M_i \otimes q^*\mathscr L_i$ for some line bundles $\mathscr M_i$ on $\Pi$ and $\mathscr L_i$ on $\mathbf P^n$, and thus $\deg\left( E^i_{Z_s}\right) = \deg(\mathscr L_i|_{Z_s}) = aL_i.H^{n-1}$ with $L_i = c_1(\mathscr L_i)$. Hence,
$$
\mu\left(E^i_{Z_s}\right) - \mu\left(E^{i+1}_{Z_s}\right) \geq \frac{a}{\max\left\{\frac{r^2 - 1}{4}, 1\right\}}.
$$
On the hand, the kernel $K$ of the evaluation map $H^0(\mathcal O_{\mathbf P^n}(a)) \otimes \mathcal O_{\mathbf P^n} \rightarrow \mathcal O_{\mathbf P^n}(a)$ is $\mu$-semistable presented by Brenner in~\cite{MR2435644}, and from the proof of Flenner's theorem (see Step 4 of proof of Theorem~7.1.1 in~\cite{MR2665168}) one can find 
$$
\mu_{\max}((\Ox_{Z/\mathbf P^n})_{Z_s}) \leq \frac{a^2}{\binom{a+n}{n} - a - 1},
$$
and then by simple computation with the inequalities above, it turns out that
$$
a \geq \frac{1}{\max\left\{\frac{r^2 - 1}{4}, 1\right\} + 1}\left(\binom{a+n}{n} - 1 \right),
$$
a contradiction to the assumption $a >> 0$.
\end{proof}
\begin{remark}
In Theorem~4.1 in~\cite{MR2727611} Langer proves an effective restriction theorem of strongly $\mu$-semistability on a smooth projective variety $X$ of dimension $n$ with $\mu_{\max}(\Ox_X) \leq 0$ by using his version of Bogomolov Inequality in Theorem~2.2 in~\cite{MR2214897}. However, we can not assume Bogomolov Inequality holds for $n > 2$ which is what we are planing to show now. Thus, instead of Langer's Bogomolov Inequality we use Mehta-Ramanathan restriction theorem~\cite{MR649194} to prove this non-effective restriction theorem of strongly $\mu$-semistability on projective spaces.
\end{remark}

Together with the restriction theorem of strongly $\mu$-semistability on $\mathbf P^n$ and arguments in previous section we are ready to prove the  generalized lemma below.
\begin{lem}\label{pn}
Let $E$ be a $\hat{\mu}$-semistable torsion free sheaf on $\mathbf P^n$. Then the Euler characteristic of $\mathcal O_{\mathbf P^2}$ and $E$ has an upper bound
$$
(-1)^{n-2}\chi(\mathcal O_{\mathbf P^2}, E) \leq (-1)^{n}\chi(k(p), E)\binom{\hat{\mu}(E)}{2}
$$
with $\mathrm{rk}(E) = (-1)^n\chi(k(p), E)$ and $\hat{\mu}(E) = -\frac{\chi(\mathcal O_{\mathbf P^1}, E)}{\chi(k(p), E)}$.
\end{lem}
\begin{proof}
Using the Riemann-Roch formula for coherent sheaves on $\mathbf P^n$ and the fact $\mathrm{Pic}(\mathbf P^n) \simeq \mathbb Z$, the inequality is equivalent to the Bogomolov Inequality for $\mu$-semistable torsion free sheaves on $\mathbf P^n$, i.e.,
$$
\Dx(E).H^{n-2} = \left((\mathrm{rk}(E)-1)c_1(E)^2 - 2\cdot \mathrm{rk}(E)c_2(E)\right).H^{n-2} \leq 0,
$$ 
where $H := \mathcal O_{\mathbf P^n}(1)$. We prove this inequality by induction on the dimension $n$. Assume this inequality hold for $\mathbf P^{n-1}$. Note that if the restriction of a $\mu$-semistable torsion free sheaf on $\mathbf P^n$ to a general hyperplane $\mathbf P^{n-1}$ is also $\mu$-semistable torsion free, then this inequality holds simultaneously. Hence for any torsion free coherent sheaf $E$ of $\mathrm{rk}(E) \leq n - 1$ on $\mathbf P^n$, one has $\Dx(E).H^{n-2} \leq 0$ by Maruyama's theorem in~\cite{MR571438}.

Therefore, the previous discussion in Sec.~\ref{sec:nv} implies the boundedness theorem and Bogomolov Inequality (see Theorem~\ref{hebound} and~\ref{bogomolovn}) on normal projective varieties of dimension $= n - 1$. From Theorem~\ref{ls} the restriction of a $\mu$-semistable torsion free sheaf $E$ on $\mathbf P^n$ to a general hypersurface $D$ in $|aH|$ for $a >> 0$ is also strongly $\mu$-semistable. Then by Bogomolov Inequality for strongly $\mu$-semistable sheaves on $D$, the inequality $\Dx(E).H^{n-2} \leq 0$ holds and thus we are done. 
\end{proof}
\begin{remark}
Theorem~3.2 presented by Langer in ref.~\cite{MR2051393} implies $\Dx(E).H^{n-2} \leq 0$ for any strongly $\mu$-semistable sheaf $E$ on smooth projective varieties. Note that $\Ox_{\mathbf P^n} \hookrightarrow \mathcal O_{\mathbf P^n}(-1)^{\oplus n+1}$, one has $\mu_{\max}(\Ox_{\mathbf P^n}) < 0$ then each $\mu$-semistable sheaves is also strongly $\mu$-semistable. Indeed, in homogeneous spaces, abelian and toric varieties all semistable sheaves are strongly semistable. Thus this also yields the inequality in the lemma above. Indeed, the idea of Langer's proof considers the generalized stability with respect to a collection of nef divisors $D_1,\dots, D_{n-1}$ with the numerically nontrivial 1-cycle $D_1.\cdots.D_{n-1}$, which is much more than what we really need. For our purpose, it is sufficient to use this restriction theorem of strongly $\mu$-semistability on projective spaces.
\end{remark}

As for the previously discussed boundedness theorem on normal projective surfaces, the established correspondence between normal integral projective schemes of dimension $n$ and the $n$-dimensional projective space carries over to the higher dimensional boundedness theorem as well, which we illustrate in the following generalized theorem. 
\begin{theorem}\label{hebound}
Let $E$ be a torsion free coherent sheaf on a normal integral projective scheme $X$ over a perfect field $k$ of $\dim X = n$ with a fixed very ample invertible sheaf $\mathcal O_X(H)$. Then there exists a polynomial $\hat P(\hat{\mu}_{\max}(E), \hat{\mu}_{\min}(E), \ox_X, d = \deg_H X)$ depending on the HN filtration of $E$, the degree of the canonical sheaf $\ox_X$ of $X$ and the degree of $X$ with respect to $\mathcal O_X(H)$ such that
\begin{align*}
&(-1)^{n-2}\chi(\mathcal O_{H^{n-2}}, E) \leq (-1)^n\chi(\mathcal O_{H^n}, E)\hat P(\hat{\mu}_{\max}(E), \hat{\mu}_{\min}(E), \ox_X, d)\\
&=(-1)^n\chi(\mathcal O_{H^n}, E)\left(\hat P(\hat{\mu}(E), \ox_X, d) + \frac{1}{2}\left(\hat{\mu}_{\max}(E) - \hat{\mu}(E)\right)(\hat{\mu}(E) - \hat{\mu}_{\min}(E))\right)
\end{align*}
with $(-1)^n\chi(\mathcal O_{H^n}, E) = d\cdot \mathrm{rk}(E)$ and $\hat{\mu}(E) = -\frac{\chi(\mathcal O_{H^{n-1}}, E)}{\chi(\mathcal O_{H^n}, E)}$. Here, the polynomial $\hat P$ is the same as the one in Theorem~\ref{ebound} or~\ref{lbound}, i.e.,
$$
\hat P(\hat{\mu}(E), \ox_X, d) = \binom{\hat{\mu}(E)}{2} + \frac 1 2 (n - \hat{\mu}(\mathcal O_X))(1 + \hat{\mu}(\ox_X)).
$$
\end{theorem}
\begin{remark}\label{rem0}
(i) The fact that the performed analysis of the boundedness theorem for higher dimensional varieties closely parallels the study of the boundedness theorem on surfaces does not come as a surprise, because the correction terms arise from the difference in the binomial term $\binom{\hat{\mu}(E)}{2}$ between the $\mu$-semistable sheaf and the sum of upper bounds of HN factors of its direct image in $\mathbf P^n$. As a result, the topological data of the normal integral projective scheme $X$ for the degree, structure sheaf and canonical sheaf resulting from a given complete linear system $|\mathcal O_X(H)|$ are the same for the discussed boundedness theorem.

(ii) As $X$ is integral projective scheme over $k$, the polynomial $\hat P$ has the crude form below
$$
\hat P(\hat{\mu}(E), d) = \binom{\hat{\mu}(E)}{2} + \frac{d^2}{2},
$$
since $\mu_{\max}(f_*F) - \mu(f_*F) \leq d$ and $\mu(f_*F) - \mu_{\min}(f_*F) \leq d$ for each $\mu$-semistable torsion free sheaf $F$ by Lemma~6.2.2 in ref.~\cite{MR2214897}. Therefore, for a torsion free sheaf $E$ on $X$ we have the following bound 
$$
(-1)^{n-2}\chi(\mathcal O_{H^2}, E) \leq (-1)^n\chi(\mathcal O_{H^n}, E)\left(\hat P(\hat{\mu}(E), d) + \frac{1}{2}\left(\hat{\mu}_{\max}(E) - \hat{\mu}(E)\right)(\hat{\mu}(E) - \hat{\mu}_{\min}(E))\right).
$$
\end{remark}

In the normal projective surfaces over a perfect field the construction of $\Dx$-stability of degree one, Hodge Index Theorem, Bogomolov Inequality and effective restriction theorem rely on Theorem~\ref{ebound} or~\ref{lbound} of normal projective surfaces such that for any $\mu$-semistable sheaf $E$ the Euler characteristic $\chi(\mathcal O_X, E)$ has an upper bound of the form $\chi(\mathcal O_{H^2}, E)\hat P(\hat{\mu}(E), \ox_X, d)$ with a binomial function $\hat P$. Therefore, the boundedness theorem for higher dimension discussed here directly deduces higher dimensional versions of those theorems in the context of Sec.~\ref{sec:nv} and~\ref{efr}. Specifically, Bogomolov Inequality for strongly $\hat\mu$-semistable locally free sheaves on a normal integral projective scheme is also valid.
\begin{theorem}[Bogomolov Inequality]\label{bogomolovn}
Suppose $X$ is a normal integral projective scheme of dimension $n$ over a perfect field $k$ with a very ample invertible sheaf $\mathcal O_X(H)$. If $E$ is a strongly $\hat\mu$-semistable locally free sheaf, then 
$$
\Dx(E).H^{n-2} := \left((\mathrm{rk}(E)-1)c_1(E)^2 - 2\cdot \mathrm{rk}(E)c_2(E)\right).H^{n-2} \leq 0.
$$
\end{theorem}
\begin{remark}
This theorem generalizes Theorem~3.2 in~\cite{MR2051393} to strongly $\mu$-semistable locally free sheaves on normal projective  varieties. Indeed, for any tensor power of a $\mu$-semistable locally free sheaf which is again $\mu$-semistable, this inequality holds. Thus Hodge Index Theorem is a special case of Bogomolov Inequality.
\end{remark}

Moreover, in this setting an upper bound of this form on normal integral projective schemes $X$ yields a $\Dx$-stability of degree $= \dim X - 1$ with the associated numerical slope polynomial $P_t(-) = \chi(L_{-t}, -)$ induced by a slope sequence $\{L_{-t}\}$ in the tilted heart of a bounded t-structure on $\mathrm{D^b}(X)$. 
\begin{theorem}\label{stabilityn}
Suppose $X$ is a normal integral projective scheme of dimension $n$ over a perfect field $k$ with a very ample invertible sheaf $\mathcal O_X(H)$. Given a rational number $q = \frac{m_1}{m_2} \in \mathbb Q$, there exists the tilted heart $\mathcal A_q$ of a bounded t-structure on the bounded derived category of coherent sheaves $\mathrm{D^b}(X)$ with a one parameter family of slope sequences $\left\{L_{-t}^s(m_0) : m_0 > m_{\min}\right\}$ and $L_{-t}^s(m_0)[1] \in \mathrm{Ext}^n\left(\mathcal O_{H^n}^{\oplus m_2\binom t n + m_1\cdot \binom{t}{n-1} + m_0\binom{t}{n-2}} , \mathcal O_X(-tH)^{\oplus m_2}\right)$ such that the associated slope polynomials $P^s_{t, m_0}(-) = \chi\left(L^s_{-t}(m_0), -\right)$ induce $\Dx$-stabilities of degree $=n-1$ on $\mathrm{D^b}(X)$. Here, $m_{\min}$ is the smallest integral which is larger than $m_2 \cdot \hat P(q, \ox_X, d)$, i.e.,
$$
m_{\min} = \min\left\{m \in \mathbb Z : m > m_2 \cdot \hat P(q, \ox_X, d = \deg_H X)\right\}
$$
with the polynomial $\hat P$ appearing in Theorem~\ref{hebound}.
\end{theorem}
\begin{remark}
Compared with the notions of Gieseker stability, Mumford-Takemoto stability, or $\Dx$-stability of degree~$=\dim X$ which are well defined on any projective scheme $X$ over a field, we could not expect to construct $\Dx$-stability of degree~$=\dim X - 1$ by using the tilted heart from the abelian category of coherent sheaves on arbitrary projective schemes. Indeed, the tilted heart may not be weakly $\Dx$-Noetherian. For a general treatment compare Theorem~\ref{tHN}.
\end{remark}

Finally, more general effective restriction theorem occurs from the similar calculation in the proof of Theorem~\ref{effective}.
\begin{theorem}\label{effectiven}
Let $E$ be a $\hat\mu$-(semi)stable torsion free sheaf of $\mathrm{rk}(E) > 1$ on a normal integral projective scheme $X$ of $\dim X = n$ over a perfect field $k$ with a very ample invertible sheaf $\mathcal O_X(H)$. Assume $D \in |\mathcal O_X(lH)|$ is a normal divisor such that $E|_D$ is torsion free. If
$$
l > 2(1 - \mathrm{rk}(E))\left((-1)^{n-2}\chi(\mathcal O_{H^{n-2}}, E) - d\cdot\mathrm{rk}(E)\hat P(\hat{\mu}(E), \ox_X, d)\right) + \frac{1}{d\cdot\mathrm{rk}(E)(\mathrm{rk}(E) - 1)},
$$
then $E|_D$ is $\hat\mu$-(semi)stable. Here, $\hat P$ is the polynomial in Theorem~\ref{hebound}.
\end{theorem}
\begin{remark}\label{rem1}
(i) The theorem gives a positive answer to the question that whether for any $\mu$-semistable torsion free sheaf $E$ on a normal projective variety $X$ there exists a sufficiently large integer $d_0$ such that the restriction of $E$ to the generic hypersurface of degree $d > d_0$ is $\mu$-semistable again. Indeed, the lower bound $d_0$ can be numerically controlled by the cohomological datum of $E$ and $X$.

(ii) For an integral projective scheme $X$ over $k$, using the inequality in Remark~\ref{rem0} the similar argument implies that for a $\hat\mu$-(semi)stable torsion free sheaf $E$, the restriction of $E$ to a general divisor $D \in |\mathcal O_X(lH)|$ is $\hat\mu$-(semi)stable again if $E|_D$ is torsion free and
$$
l > 2(1 - \mathrm{rk}(E))\left((-1)^{n-2}\chi(\mathcal O_{H^{n-2}}, E) - d\cdot\mathrm{rk}(E)\hat P(\hat{\mu}(E), d)\right) + \frac{1}{d\cdot\mathrm{rk}(E)(\mathrm{rk}(E) - 1)}.
$$
\end{remark}

\begin{center}
\textbf{Acknowledgements} 
\end{center}We would like to thank 
Sheng-Fu Chu, Rung-Tzung Huang, Chin-Yu Hsiao, Yuan-Pin Lee and I-Hsun Tsai 
for discussions and correspondences. Main part of this paper was completed during the author's stay at Institute of Mathematics, Academia Sinica. The author was supported by Taiwan Ministry of Science and Technology projects 109-2811-M-008-532. 


\bibliographystyle{alpha}
\bibliography{nvpaper}
\end{document}